\documentclass[12pt, a4paper, parskip=half, abstracton, bibliography=totoc]{scrartcl}
\pdfoutput=1
\usepackage{array}
\usepackage{marginnote}
\usepackage[dvipsnames]{xcolor}
\usepackage{hyperref}
\usepackage{amscd,amssymb,amsfonts,amsmath,latexsym,amsthm}
\usepackage[capitalize]{cleveref}
\usepackage[all,cmtip]{xy}
\usepackage{mathrsfs}
\usepackage{graphicx}
\usepackage{bm} 
\usepackage{mathtools} 

\usepackage{etoolbox}

\let\counterwithout\relax

\usepackage{chngcntr} 

\usepackage{defs_pp1}
\usepackage{slashed}
\usepackage[utf8]{inputenc}
\usepackage{microtype}
\usepackage[english]{babel}

\title{Injectivity results for coarse homology theories}
\date{\today}
\author{
Ulrich Bunke\thanks{Fakult{\"a}t f{\"u}r Mathematik,
Universit{\"a}t Regensburg,
93040 Regensburg,
Germany\newline
ulrich.bunke@mathematik.uni-regensburg.de} 
\and Alexander Engel\thanks{Fakult{\"a}t f{\"u}r Mathematik,
Universit{\"a}t Regensburg,
93040 Regensburg,
Germany\newline
alexander.engel@mathematik.uni-regensburg.de
	}
\and
Daniel Kasprowski\thanks{
	Rheinische Friedrich-Wilhelms-Universit\"at Bonn, Mathematisches Institut, Endenicher Allee 60,\newline 53115 Bonn, Germany\newline
	kasprowski@uni-bonn.de
}
\and
Christoph Winges\thanks{
		Rheinische Friedrich-Wilhelms-Universit\"at Bonn, Mathematisches Institut, Endenicher Allee 60,\newline 53115 Bonn, Germany\newline
	winges@math.uni-bonn.de}
}

\numberwithin{equation}{section}
\counterwithout{footnote}{section}

\newtheorem{theorem}{Theorem}[section] 
\newtheorem{prop}[theorem]{Proposition}
\newtheorem{ass-alt}[theorem]{Assumption}
\newtheorem{lem}[theorem]{Lemma}
\newtheorem{kor}[theorem]{Corollary}
\theoremstyle{remark}
\theoremstyle{definition}
\newtheorem{ddd-alt}[theorem]{Definition}
\newtheorem{ex-alt}[theorem]{Example}
\newtheorem{constr-alt}[theorem]{Construction}
\newtheorem{rem-alt}[theorem]{Remark}

\newenvironment{ddd}    
{%
	\pushQED{\qed}\begin{ddd-alt}}
	{\popQED\end{ddd-alt}}

\newenvironment{ex}    
{%
	\pushQED{\qed}\begin{ex-alt}}
	{\popQED\end{ex-alt}}

	\newenvironment{constr}    
{%
	\pushQED{\qed}\begin{constr-alt}}
	{\popQED\end{constr-alt}}

\newenvironment{rem}    
{%
	\pushQED{\qed}\begin{rem-alt}}
	{\popQED\end{rem-alt}}

{%
	\pushQED{\qed}\begin{ass-alt}}
	{\popQED\end{ass-alt}}

	\crefname{lem}{Lemma}{Lemmas}

	\newcommand{\weak}{\mathrm{weak}}

	\newcommand{\fin}{\mathrm{fin}}
	 \newcommand{\udisc}{\mathrm{udisc}}
\newcommand{\Simpl}{\mathbf{Simpl}}

\newcommand{\free}{\mathrm{free}}

\newcommand{\Tw}{\mathbf{Tw}}

\newcommand{\All}{\mathbf{All}}

\newcommand{\UBC}{\mathbf{UBC}}

\newcommand{\pt}{\mathrm{pt}}

\newcommand{\Rips}{\mathrm{Rips}}

\DeclareMathOperator{\indd}{ind}

\newcommand{\Fix}{\mathrm{Fix}}

\DeclareMathOperator{\Yo}{Yo}

\DeclareMathOperator{\yo}{yo}

\DeclareMathOperator{\Res}{Res}

\newcommand{\Orb}{\mathbf{Orb}}

\newcommand{\cR}{\mathcal{R}}

\newcommand{\BC}{\mathbf{BornCoarse}}

\newcommand{\Fin}{\mathbf{Fin}}

\newcommand{\wt}{\widetilde}
\newcommand{\hlg}{\mathrm{hlg}}

\newcommand{\bL}{\mathbf{L}}
\newcommand{\bM}{\mathbf{M}}

\newcommand{\cP}{\mathcal{P}}

\newcommand{\cK}{\mathcal{K}}

\newcommand{\cL}{{\mathcal{L}}}

\newcommand{\PSh}{{\mathbf{PSh}}}

\newcommand{\bA}{{\mathbf{A}}}

\newcommand{\cO}{{\mathcal{O}}}
\newcommand{\cU}{{\mathcal{U}}}
\newcommand{\cY}{{\mathcal{Y}}}

 \newcommand{\Cat}{{\mathbf{Cat}}}

\newcommand{\itconn}{\mathit{conn}}
\newcommand{\itprop}{\mathit{prop}}
\newcommand{\itfin}{\mathit{fin}}

\newcommand{\Born}{\mathbf{Born}}
\newcommand{\Coarse}{\mathbf{Coarse}}
\newcommand{\Spc}{\mathbf{Spc}}

\newcommand{\FDC}{\mathbf{FDC}}
\newcommand{\cFDC}{\mathbf{FDC^{cp}}}
\newcommand{\VCyc}{\mathbf{VCyc}}
\newcommand{\cp}{\mathbf{CP}}
 \newcommand{\As}{\mathrm{Asmbl}}
 \newcommand{\homolg}{\mathrm{hlg}}
 
 \renewcommand{\tilde}{\widetilde}
 
 \DeclareMathOperator{\ff}{\overline{Fix}}
 
\begin{document}
\maketitle

\begin{abstract}
We show injectivity results for assembly maps using equivariant coarse homology theories with transfers.
Our method is based on the   descent principle and applies to a large class of linear groups or, more generally, groups with finite decomposition complexity.  

\end{abstract}

\tableofcontents
\section{Introduction}

For a group $G$ we consider a functor $M\colon G\Orb\to \bC$ from the orbit category of $G$ to a cocomplete $\infty$-category $\bC$. Often one is interested in
the calculation of the object $\colim_{G\Orb} M$ in $\bC$, or equivalently, in the value $M(*)$ at the final object $*$ of $G\Orb$.
Given a family of subgroups $\cF$ of $G$ one can then ask which information about this colimit can be obtained from the 
restriction of $M$ to the subcategory $G_{\cF}\Orb$ of orbits with stabilizers in $\cF$. To this end one considers the assembly map
\[\As_{\cF,M}\colon \colim_{ G_{\cF}\Orb}M \to \colim_{ G  \Orb}M\ .\]
If $M$ is algebraic or topological $K$-theory, then such  assembly maps appear in the   Farrell-Jones  or Baum-Connes conjectures{; see} for example L\"{u}ck and Reich \cite{MR2181833} and Bartels \cite{MR3598160}.

In the present paper we show  split injectivity results about the assembly map by proving a descent principle. This method was first applied by   Carlsson and Pederson \cite{calped}.
For the application of the descent principle, on the one hand we will use geometric properties of the group $G$ like finite decomposition complextity as introduced by  Guentner, Tessera and Yu \cite{Guentner:2010aa,GTY}. On the other hand, we use that $M$ extends to an equivariant coarse homology theory with transfers as introduced in \cite{coarsetrans}. The main theorem of the paper is \cref{thm:main-injectivity-kor1}. 

We now start by introducing the notation which is necessary to state the theorem  and its assumptions  in detail.
Let $G$ be a group and $\cF$ be a set of subgroups of $G$.   

\begin{ddd} The set $\cF$ is called a \emph{family} of subgroups if it is non-empty, closed under conjugation in $G$, and taking subgroups.
\end{ddd} 

Let $\cF$ be a family of subgroups of $G$.
\begin{ddd}\label{gbiofgergergergrg}	\mbox{}
 \begin{enumerate}
\item $G\Set$ denotes the category of $G$-sets and equivariant maps.
\item $G_\cF\Set$ denotes the full subcategory of $G\Set$ of  $G$-sets with stabilizers in $\cF$.
\item $G\Orb$ denotes the full subcategory of $G\Set$ of  transitive $G$-sets.
\item $G_\cF\Orb$  denotes the full subcategory of $G_{\cF}\Set$ of  transitive $G$-sets  with stabilizers in $\cF$.\qedhere
\end{enumerate}\end{ddd}

The $\infty$-category of spaces will be denoted by   $\Spc$. For any small $\infty$-category $\cC$ (ordinary categories are considered as $\infty$-categories using the nerve) we use the notation $\PSh(\cC):=\Fun(\cC^{op},\Spc)$ for the $\infty$-category of $\Spc$-valued presheaves.
  \begin{ddd}\label{def:efg}
We denote by $E_\cF G$ the object of the presheaf category $\PSh(G\Orb)$, which is essentially uniquely determined by
\begin{equation*}\label{ef2uhfi23f32f2}
E_\cF G(T)\simeq \begin{cases}
                  *& \text{if }T\in G_{\cF}\Orb \\
                  \emptyset&\text{else}
                 \end{cases}\qedhere
\end{equation*}
\end{ddd}

In \cite[Def.~3.14]{trans} we defined the notion of $G$-equivariant finite decomposition complexity ($G$-FDC) for a $G$-coarse space (\cref{ergergregreg5t45t}). $G$-FDC is  an equivariant version of the notion of  finite decomposition complexity FDC  {which was originally} introduced by Guentner, Tessera and Yu \cite{GTY}.

For $S$ in $G\Set$ we let $S_{min}$ denote the $G$-coarse space with underlying $G$-set $S$ and the minimal coarse structure (see \cref{ufewiofheiwfhuewifewew}). In the definition below $\otimes$ denotes the cartesian product in the category $G\Coarse$ of $G$-coarse spaces.

Let $\cF$ be a family of subgroups of $G$ and $X$ be a $G$-coarse space.
\begin{ddd}\label{rfgui32r23r32r32r23-neu}
	$X$ has \emph{$G_\cF$-equivariant finite decomposition complexity} (abbreviated by $G_\cF$-FDC) if $S_{min}\otimes X$ has $G$-FDC for every  $S$ in $G_{\cF}\Set$.  \end{ddd}

We will consider the following families of subgroups.
\begin{ddd}\label{grregergerg}
	~\begin{enumerate}
		\item   $\Fin$ denotes the family of finite subgroups of $G$.
		\item  $\VCyc$ denotes the family of virtually cyclic subgroups of $G$.
		\item\label{giorgjerogiergreg}   $\FDC$ denotes the family of subgroups $V$ of $G$ such that $V_{can}$ has $V_\Fin$-FDC.
		\item \label{giorgjerogiergreg1} $\cp$ denotes the family of subgroups of $G$ generated by those subgroups $V$ such that $E_\Fin V$ is a compact object of $\PSh(V\Orb)$.
		\item   $\cFDC$ denotes  the intersection of $\FDC$ and $\cp$.\qedhere
	\end{enumerate}
\end{ddd}
 
\begin{rem}
 The notation $V_{can}$ in the definition of the family $\FDC$ refers to the group $V$ with the canonical coarse structure described below in \cref{ufewiofheiwfhuewifewew}. 

In order to see that $\FDC$ is a family of subgroups we use that the condition that $V_{can}$ has $V_{\Fin}$-FDC is stable under taking subgroups,  {see \cref{lem:fdc-subgroups}}.

An object  $A$ of an $\infty$-category $\bD$ is called compact if the functor $\Map(A,-)\colon\bD\to \Spc$ commutes with filtered colimits. The word {\em compact} in the definition of $\cp$ is understood in this sense.

The family of subgroups of $G$ generated by a set of subgroups of $G$  is the smallest family containing this subset. The condition that $E_{\Fin}V$ is compact is not stable under taking subgroups. Hence the family $\cp$ may also contain subgroups $V^{\prime}$ with noncompact $E_{\Fin}V^{\prime}$.
\end{rem}

Let $\bC$ be a cocomplete $\infty$-category and let
\begin{equation}\label{gegijioergergrg}M\colon G\Orb \to \bC\end{equation} be a functor. 
Let $\cF$ and $\cF^{\prime}$ be families of subgroups  such that $\cF^{\prime}\subseteq \cF$.
\begin{ddd}\label{ergoiegererg}
	The {\em relative assembly map} $\As_{\cF^{\prime},M}^{\cF}$  is the morphism 
	\[\As_{\cF^{\prime},M}^{\cF}\colon \colim_{ G_{\cF^{\prime}}\Orb}M \to \colim_{ G_\cF \Orb}M \]
	in $\bC$ canonically induced by the inclusion $G_{\cF^{\prime}}\Orb\to G_{\cF}\Orb$.
	
	If $\cF^{\prime}=\Fin$ and $\cF=\All$, then we omit the symbol $\All$ and call 
	$\As_{\Fin,M}$ simply
	the {\em assembly map}. 
\end{ddd}

In order  to capture the large-scale geometry of metric spaces like $G$ (with its word metric), we introduced   the category of $G$-bornological coarse spaces $G\BC$  in  \cite{buen}, \cite{equicoarse}. We further defined the notion of an  equivariant coarse   homology theory. All this will be recalled in detail in \cref{sec:homologytheories}.

We can embed the orbit category $G\Orb$ into  $G\BC$ by a functor 
\[i\colon G\Orb\to G\BC\] which sends a $G$-orbit $S$ to the $G$-bornological coarse space  
$S_{min, {max}}${; see} \cref{ufewiofheiwfhuewifewew}. {Note that the convention is that the first index specifies the coarse structure while the second index specifies the bornology.
We say that a functor $M\colon G\Orb
\to \bC$ can be extended to an equivariant coarse homology theory if there exists an equivariant coarse   homology theory 
 $F\colon G\BC\to \bC$ such that $M\simeq F\circ i $.

We will need various additional properties or  structures for an equivariant coarse homology theory.
\begin{enumerate}
\item The property of  {\em continuity}   of an equivariant  coarse homology theory was defined in \cite[Def.~5.15]{equicoarse},  {see \cref{rguierog34trt34g}.}
\item The property of  {\em strong additivity} of an equivariant  coarse homology theory  was defined in \cite[Def.~3.12]{equicoarse},  {see \cref{regerger}.}
\item The additional structure of transfers for an equivariant  coarse homology theory  is encoded in the notion of a {\em  coarse homology theory with transfers}
 which was   defined in \cite{coarsetrans},  {see \cref{rgiorhgiuregergergergergerg}.}
 \end{enumerate}

 Let $G_{can,min}$ denote the $G$-bornological coarse space consisting of $G$ with the canonical coarse and the minimal bornological structures{; see} \cref{ufewiofheiwfhuewifewew}. We  furthermore consider a stable $\infty$-category   $\bC$  and an  equivariant coarse homology theory  {(see \cref{def:coarsehom})} 
\[E\colon G\BC\to \bC\ .\]
To $E$ and $G_{can,min}$ we associate a new   equivariant coarse homology theory \[E_{G_{can,min}}\colon G\BC\to \bC\ ,\quad 
X\mapsto E(G_{can,min}\otimes X)\] called the twist of $E$ by $G_{can,min}${; see} \cref{def:twist}.

We can now  introduce the following assumption on a functor $M\colon G\Orb\to \bC$.
 \begin{ddd}\label{bioregrvdfb}
We call $M$ a \emph{CP-functor} if it satisfies the following assumptions:
		\label{def:cpfunctor}
				\begin{enumerate}
			\item $\bC$ is stable, complete, cocomplete, and compactly generated;  
			\item There exists an equivariant coarse homology theory $E$ satisfying:
			\begin{enumerate}
				\item $M$ is equivalent to $E_{G_{can,min}}\circ i$;
				\item $E$ is strongly additive;
				\item $E$ is continuous;
				\item $E$ extends to a coarse homology theory with transfers.\qedhere
			\end{enumerate}
		\end{enumerate}
	\end{ddd}

\begin{rem}
	We call $M$ a CP-functor since
	the above assumptions will allow us to apply methods similar to those from Carlsson and Pedersen \cite{calped}.
\end{rem}

 \begin{ex} \label{ergerge}~
 	\begin{enumerate}
 		\item {We claim that the} equivariant $K$-theory functor
 		\[K\bA^{G}\colon G\Orb\to \Sp\]
 		associated to an additive category  with $G$-action  $\bA$ (see  \cite[Def. 2.1]{bartels-reich-coeff})   is an example of a CP-functor.
 		Indeed, by  \cite[Cor. 8.25]{equicoarse} we have an equivalence 
 		\[K\bA^{G}\simeq K\bA\cX_{G_{can,min}}^{G}\circ i\ ,\] where $K\bA \cX^{G}\colon G\BC\to \Sp$ denotes the  coarse algebraic $K$-homology  functor.
 		By \cite[Thm.~1.4]{coarsetrans} the functor $K\bA\cX^{G}$ admits an extension to an equivariant coarse homology theory with transfers. Furthermore, $K\bA\cX^{G}$ is continuous by \cite[Prop.~8.17]{equicoarse} and strongly additive by \cite[Prop.~8.19]{equicoarse}.  
 		\item For a group $G$, let $P$ be the total space of a principal $G$-bundle and let $\bA$ denote the functor of nonconnective $A$-theory (taking values in the $\infty$-category of spectra). Then
 		$P$ gives rise to a $G\Orb$-spectrum $\bA_P$ sending a transitive $G$-set $S$ to the spectrum $\bA(P \times_G S)$. By \cite[Thm.~5.17]{coarseA}, $\bA_P$ is a CP-functor.
 		\item {More generally, every right-exact $\infty$-category with $G$-action $\cC$ gives rise to a functor $K\cC_G \colon G\Orb \to \Sp$. Taking $\cC = \Ch^b(\bA)$ or $\cC = \Sp$, this recovers $K\bA^G$ and $\bA_{EG}$, but one may also consider categories of perfect modules over an arbitrary ring spectrum. Also in this generality, $K\cC_G$ is a CP-functor. See \cite{unik} for details and proofs.}\qedhere
 	\end{enumerate}
 \end{ex}

We can now state the main theorem of this paper.
Let $G$ be a group and $M\colon G\Orb\to \bC$ be a functor. Let $\cF$ be a family of subgroups.  
\begin{theorem}
	\label{thm:main-injectivity-kor1}
	Assume that  {$M$ is a CP-functor (\cref{def:cpfunctor}). Furthermore, assume that one of the following conditions holds:
		\begin{enumerate}
			\item\label{thm:main-it1} $\cF$ is a subfamily of $\cFDC$ 
			such that $\Fin\subseteq \cF$;
			\item\label{thm:main-it2} $\cF$ is a subfamily of $\FDC$
			such that $\Fin\subseteq \cF$  and $G$ admits a finite-dimensional model for $E^{top}_\Fin G$.
		\end{enumerate}
		}
	Then the relative assembly map $\As_{\Fin,M}^\cF$ admits a left inverse.
\end{theorem}

 \begin{rem}\label{efweu9u9wefwef} By Elmendorf's theorem the homotopy theory  of $G$-spaces is modeled by the presheaf category $\PSh(G\Orb)$. More precisely, we have a functor
\begin{equation} \label{h56h5h56h5h56h56h}\mathrm{Fix}\colon G\Top\to \PSh(G\Orb)
\end{equation} which sends a $G$-topological space $X$ to the $\Spc$-valued presheaf which associates to $S$ in $G\Orb$ the
mapping space $\ell(\Map_{G\Top}(S_{disc},X))$. Here $S_{disc}$ is $S$ considered as a discrete $G$-topological space,
$\Map_{G\Top}(S_{disc},X)$ in $\Top$ is the topological space of equivariant maps from  $S_{disc}$ to $X$, and
$\ell \colon \Top\to \Top[W^{-1}]\simeq \Spc$ is the localization  functor inverting the weak equivalences in $\Top$ in the realm of $\infty$-categories.
Let $W_{G}$ be the morphisms in $G\Top$ which are sent by the functor $\mathrm{Fix}$  to equivalences. Then Elemendorf's theorem asserts that $\Fix$ induces an equivalence
of $\infty$-categories \begin{equation}\label{efwoiwggwegf} \ff \colon
G\Top[W^{-1}_{G}]\stackrel{\simeq}{\to} \PSh(G\Orb)\ .
\end{equation}
A model $E^{top}_{\cF}G$ for a classifying space $E_{\cF}G$ of a family $\cF$ is a $G$-CW complex $X$ whose fixed point spaces $X^{H}$ are contractible for all subgroups $H$ in $\cF$ and empty otherwise. Such a model is uniquely determined up to equivariant homotopy equivalence. It represents the object $E_{\cF}G$ from \cref{def:efg} under the equivalence \eqref{efwoiwggwegf}.
\end{rem}

Let $G$ be a group and $M\colon G\Orb\to \bC$ be a functor. 
\begin{kor}
	\label{kor:vcyc-bartels}
	If $M$ is a CP-functor, then
	  the relative assembly map $\As_{\Fin,M}^\VCyc$ admits a left inverse.
\end{kor}
\begin{proof}
	Every virtually cyclic subgroup $V$ admits a compact model for $E_\Fin V$. Furthermore, it has $V_{\Fin}$-FDC{; see} \cref{ex:groups1}. We conclude that $\Fin\subseteq \VCyc\subseteq \FDC^{\cp}$ and hence the corollary follows from Case \ref{thm:main-it1} of  \cref{thm:main-injectivity-kor1}.
\end{proof}
For algebraic $K$-theory  (\cref{ergerge}), \cref{kor:vcyc-bartels} was first proven by Bartels~\cite{bartels-fin-vcyc}.

Let $G$ be a group and  $M\colon G\Orb\to \bC$ be a functor.
\begin{kor}
	\label{kor:oldversion}
	Assume that:
	\begin{enumerate}
		\item $M$ is a CP-functor;
		\item $G$ admits a finite-dimensional model for $E^{top}_\Fin G$;
		\item \label{it:gfinfdc} $G_{can}$ has $G_\Fin$-FDC.
	\end{enumerate}
	Then the assembly map $\As_{\Fin,M}$ admits a left inverse.
\end{kor}
For algebraic $K$-theory  (\cref{ergerge}) this was first proven in \cite{KasFDC}.
\begin{proof}
 The corollary follows from Case \ref{thm:main-it2} of  \cref{thm:main-injectivity-kor1}.
\end{proof}

As an application of \cref{thm:main-injectivity-kor1} we also obtain the following new injectivity result for algebraic $K$-theory.
 
\begin{theorem}
	\label{thm:relhyp}
	Suppose $G$ is relatively hyperbolic to groups $P_1,\ldots, P_n$. Assume that each $P_i$ is contained in $\FDC$ or satisfies the $K$-theoretic Farrell--Jones conjecture.  {Furthermore, assume that each $P_i$ admits a finite-dimensional model for $E_{\Fin}^{top}P_i$.}
	Then $\As_{\Fin,K\bA^{G}}$ admits a left inverse.	
\end{theorem}

\begin{proof}
	{Let $\cF$ be a the smallest family of subgroups of $G$ that contains all finite subgroups and all $P_i$. By \cite[Thm~1.1]{martinez-pedroza_przytycki_2019} there is a cocompact model for $E_\cF^{top}G$. Since there are only finitely many $P_i$, there is a uniform upper bound on the dimension of $E_\Fin^{top} H$ for all $H$ in $\cF$. By {\cref{lem:resolutions} below}, there is a finite-dimensional model for $E_\Fin^{top}G$.}
		
	Let $\cP$ be the smallest family of subgroups of $G$ that contains all virtually cyclic subgroups and all $P_i$. By \cite[Thm.~4.4]{Bartels-relhyp} the assembly map $\As_{\cP,K\bA^{G}}$ is an equivalence. Thus by the transitivity principle \cite[Thm.~2.4]{MR2210223} the assembly map $\As_{\cP\cap\FDC,K\bA^{G}}$ is an equivalence  (here we have to use the assumptions on the groups $P_i$  {as well as that the Farrell--Jones conjecture passes to subgroups \cite[Thm.~4.5]{bartels-reich-coeff}}).  By \cref{thm:main-injectivity-kor1}, the relative assembly map $\As_{\Fin,K\bA^{G}}^{\cP\cap\FDC}$ admits a left inverse. The theorem now follows by combining these results.
\end{proof}

Let $G$ be a group and let $\cF$ and $\cF'$ be families of subgroups of $G$ such that $\cF' \subseteq \cF$. We denote the restriction of $\cF'$ to a subgroup $H$ of $G$ by $\cF'(H)$; see \cref{def:familyrestriction}.
\begin{lem}\label{lem:resolutions}
 If $G$ admits a finite-dimensional model for $E^{top}_\cF G$ and all subgroups $H$ in $\cF$ admit a model for $E^{top}_{\cF'(H)} H$ with a uniform upper bound on their dimension, then $G$ admits a finite-dimensional model for $E^{top}_{\cF'}G$.
\end{lem}
\begin{proof}
 By assumption, there exists $n$ in $\nat$ and an $n$-dimensional $G$-simplicial complex $X$ modelling $E^{top}_\cF G$.
 Choose a set of representatives $S$ for the $G$-orbits of vertices in $X$.
 Again by assumption, there exists for some $k$ in $\nat$ and every $s$ in $S$ an at most $k$-dimensional simplicial complex $Y(s)$ modelling $E^{top}_{\cF'(G_s)}G_s$. Then the projections $Y(s) \to *$ induce a $G$-equivariant map
 \[ \upsilon_0 \colon \cY := \coprod_{s \in S} G \times_{G_s} Y(s) \to \coprod_{s \in S} G \times_{G_s} * \cong X_0\ .\]
 Now apply the construction of \cite[Def.~2.2]{winges15} to obtain a $G$-simplicial complex $X[\cY,\upsilon_0]$ whose dimension is bounded by $nk + n + k$.
 After observing that this construction is compatible with taking fixed points in the sense that $X[\cY,\upsilon_0]^H \cong X^H[\cY^H,\upsilon_0^H]$ for all subgroups $H$ of $G$, \cite[Cor.~2.5]{winges15} implies that $X[\cY,\upsilon_0]$ is a model for $E^{top}_{\cF'}G$.
\end{proof}

\begin{rem}
Most of the groups for which the Farrell--Jones conjecture is known by now also have finite asymptotic dimension. But for example for CAT(0)-groups, which satisfy the Farrell--Jones conjecture \cite{wegner-cat0}, this is an open problem. Hence taking some $P_i$ to be CAT(0)-groups that are not known to have FDC and some $P_i$ to be groups that have FDC but for which the Farrell--Jones conjecture is not known, we obtain examples of groups for which \cref{thm:relhyp} applies and the split-injectivity was not known before.
\end{rem}

\subsection*{Acknowledgements} 
{The authors thank the referee for their very helpful and detailed report.}

U.B.~and A.E.~were supported by the SFB 1085 \emph{Higher Invariants} funded by the Deutsche Forschungsgemeinschaft DFG. A.E. was furthermore supported by the SPP 2026 \emph{Geometry at Infinity} also funded by DFG. C.W.~acknowledges support from the Max Planck Society {and Wolfgang L\"uck's ERC Advanced Grant ``KL2MG-interactions'' (no.~662400). D.K.~and C.W.~were funded by the Deutsche Forschungsgemeinschaft (DFG, German Research Foundation) under Germany's Excellence Strategy---GZ 2047/1 [390685813].} 

\section{Injectivity results for linear groups}
In general it is not an easy task to verify the assumptions on the group $G$ and the family $\cF$ appearing in \cref{thm:main-injectivity-kor1} and its corollaries.
In this section we provide various cases where the required properties can be shown. Furthermore, we show how \cref{thm:main-injectivity-kor1} can be applied to linear groups.

For a family $\cF$ of subgroups of $G$ we consider the $G$-coarse space $S_{\cF,min}$ consisting of the $G$-set $S_{\cF}:=\bigsqcup_{H\in \cF} G/H$ with the minimal coarse structure. Let $X$ be a $G$-coarse space. The condition that $X$ has $G_{\cF}$-FDC is equivalent to the condition that $S_{\cF,min}\otimes X$ has $G$-FDC.	

The space $(G/H)_{min}\otimes X$ has $G$-FDC if and only if the space $X$ has $H$-FDC. This can be seen by taking an $H$-equivariant decomposition of $X$ and extending it $G$-equivariantly to $(G/H)_{min}\otimes X$. Hence morally, $S_{\cF,min}\otimes X$ has $G$-FDC if and only if $X$ has $H$-equivariant FDC for every group $H$ in the family $\cF$ in a uniform way. More precisely, the condition that $S_{\cF,min}\otimes X$ has $G$-FDC is equivalent to the condition, formulated in \cite{KasLin}, that the family $\{(X,H)\}_{H\in\cF}$ has FDC.

Applying this  equivalence of conditions we can transfer the results from \cite{KasLin}.
 We consider the case $X=G_{can}$ and $\cF=\Fin$. Then we see that  {Assumption~\ref{it:gfinfdc} of Corollary~\ref{kor:oldversion}} is equivalent to the condition that the family $\{(G,H)\}_{H\in\Fin}$ has FDC. In \cite{KasLin} instead of general coarse spaces only metric spaces were considered. For a countable group $G$, the canonical coarse structure agrees with the metric coarse structure for any proper, left invariant metric $d$ on $G${; see} \cite[Rem.~2.8]{equicoarse}. Given a proper, left invariant metric $d$ on $G$, we can define a metric $d_H$ on the quotient $H\backslash G$ for every subgroup $H$ of $G$ by setting
\[d_H(Hg,Hg'):=\min_{h\in H}d(g,hg^\prime).\]
By \cite[Prop.~A.7]{KasLin} $\{(G,H)\}_{H\in\Fin}$ has FDC if and only if the family  $\{H\backslash G\}_{H\in \Fin}$ has (unequivariant) FDC (for any proper, left invariant metric on $G$). This reformulation is the statement proved in the references given in the next example.

\begin{ex}
	\label{ex:groups1}
	 {Assumption~\ref{it:gfinfdc} of Corollary~\ref{kor:oldversion}} is satisfied for finitely generated linear groups over commutative rings with unit and trivial nilradical \cite[Thm.~4.3]{KasLin}.
	
	By \cite[Thms.~2.13, 5.3, 5.21 and 5.28]{RegFDC},  {Assumption~\ref{it:gfinfdc} of Corollary~\ref{kor:oldversion}} is satisfied for groups with a uniform upper bound on the cardinality of their finite subgroups, and belonging to one of the following classes.
	\begin{enumerate}
		\item Elementary amenable groups.
		\item Countable subgroups of $GL_n(R)$, where $R$ is any commutative ring with
		unit.
		\item Countable subgroups of virtually connected Lie groups.
		\item Groups with finite asymptotic dimension.\qedhere
	\end{enumerate}
\end{ex}

See Hillman \cite{hillman} for the definition of the Hirsch length $h(G)$ of an elementary amenable group $G$. If $G$ has a finitely generated abelian subgroup $A$ of finite index, then $h(G)$ is the rank of $A$ by definition. In particular, $h(G)=0$ if $G$ is finite.

\begin{ex}
	\label{ex:groups2}
	Let $G$ be a finitely generated, linear group $G$ over a commutative ring with unit or a finitely generated subgroup of a virtually connected Lie group. By \cite[Prop.~1.3]{KasLie} and \cite[Prop.~1.2]{KasLin}, there exists a finite-dimensional CW-model for the space $E_\Fin G$ if and only if there is a natural number $N$ such that the Hirsch length of every solvable subgroup $A$ of $G$ is bounded by $N$. 
\end{ex}

Combining \cref{kor:oldversion} with \cref{ex:groups1} and \cref{ex:groups2} we obtain injectivity results for linear groups over commutative rings with unit and trivial nilradical and for subgroups of virtually connected Lie groups with a uniform upper bound on the cardinality of their finite subgroups. We will now extend these to recover the injectivity results from \cite{KasLie, KasLin} for algebraic  {$K$-theory}{; see}  \cref{cor:injectivity-linear} below.

 Before we start, we show that the family $\FDC$ is closed under subgroups.

	 Let $G$ be a group and let $H$ be a subgroup of $G$. 
	 Let  $\cF$ be a family of subgroups of $G$. \begin{ddd}
	 	\label{def:familyrestriction}
	 By \[\cF(H):=\{F\in \cF\mid F\leq H\}\] we denote the \emph{restriction of the family} $\cF$ to $H$.
\end{ddd}
	\begin{lem}
		\label{lem:fdc-subgroups}
		If $G_{can}$ has $G_\Fin$-FDC, then $H_{can}$ has $H_{\Fin(H)}$-FDC.
	\end{lem}
	\begin{proof}
		Fix a proper, left invariant metric on $G$ and consider its restriction to $H$.
		
		Recall from the discussion preceding \cref{ex:groups1} that $H_{can}$ has $H_{\Fin(H)}$-FDC if and only if $\{F\backslash H\}_{F\in \Fin(H)}$ has FDC.
		
		Each element of $\{F\backslash H\}_{F\in \Fin(H)}$ is a subspace of an element of $\{F\backslash G\}_{F\in \Fin(H)}$ which is contained in $\{F'\backslash G\}_{F'\in \Fin}$.
		If $G_{can}$ has $G_\Fin$-FDC, then $\{F'\backslash G\}_{F'\in\Fin}$ has FDC. Hence $\{F\backslash H\}_{F\in \Fin(H)}$ has FDC by \cite[Coarse Invariance 3.1.3]{GTY}.
	\end{proof}

We now consider a functor $M\colon G\Orb\to \bC$. Recall \cref{bioregrvdfb} of a CP-functor.
 \begin{ddd}\label{bioregrvdfb1}
 	We call $M$ a \emph{hereditary CP-functor} if $M\circ \Res_\phi$ is a CP-functor for every surjective homomorphism $\phi\colon G\to Q$.  \end{ddd}

\begin{ex}~
	\begin{enumerate}
		\item Recall that $K_{\bA}^{G}$ is a CP-functor by \cref{ergerge}. It is also a hereditary CP-functor since by \cite[Cor.~2.9]{bartels-reich-coeff} we have $K_\bA^G\circ \Res_\phi\simeq K_{\indd_\phi \bA}^Q$ for every surjective homomorphism $\phi\colon G\to Q$.
		\item The functor $\bA_P$ from \cref{ergerge} is also a hereditary CP-functor by \cite[Thm.~5.17]{coarseA}.\qedhere
	\end{enumerate}
\end{ex}

We will need the following well-known facts about the Hirsch length, for a proof see \cite[Thm.~1]{hillman}. For a subgroup $H$  {of $G$} we have $h(H)\leq h(G)$ and, if $H$ is normal in $G$, $h(G)=h(H)+h(G/H)$. Recall that, for finitely generated abelian groups, the Hirsch length coincides with the rank of the group.

\begin{lem}
	\label{lem:virtab}
	Every countable virtually abelian group $G$ of finite Hirsch length $n$ has $G_\Fin$-FDC.
\end{lem}
\begin{proof}
	 		Fix a left invariant, proper metric on $G$.	It suffices to show that $\{F\backslash G\}_{F\in \Fin}$ has FDC{; see} the discussion preceding \cref{ex:groups1}. More precisely, we will show that this family has  asymptotic dimension at most $n$. Then it has FDC by \cite[Thm.~4.1]{GTY}.
	
	Let $G'$ be a normal, abelian subgroup of finite index $k$.
	
	Now let $R>0$ be given. Let $H$ denote the subgroup of $G'$ generated by all elements of distance at most $R$ from the neutral element. Since $H$ is a finitely generated abelian group of  {rank at most $n$}, it has asymptotic dimension at most $n$. Moreover, there is an upper bound on the cardinality of the finite subgroups of $H$. Hence by \cite[Cor.~1.2]{KasQuot}, the family $\{F'\backslash H\}_{F'\in \Fin(H)}$ has asymptotic dimension at most $n$. In particular, there is $S>0$ such that for every $F'\backslash H$ there is a cover $U^{F^{\prime}}_0\cup\ldots \cup U^{F^{\prime}}_n$, such that for every $i$ in  $\{0,\dots,n\}$ the subset $U^{F^{\prime}}_i$ is an $R$-disjoint union of subspaces of diameter at most $S$.
	
	 Let $F$ be a finite subgroup  of $G'$ and $h,h'$ be elements of $H$.  If the condition $d(Fh,Fh')<R$ holds in $F\backslash FH$, then there exists an element $f$ of $F$ with $d(h,fh')<R$, or equivalently, $d(e,h^{-1}fh')<R$. It follows that $h^{-1}fh'\in H$ and therefore $f\in H$. Hence we get that $d((F\cap H)h,(F\cap H)h')<R$ in $(F\cap H)\backslash H$. Therefore, for every $i$ in $\{0,\dots,n\}$, the image of $U^{F\cap H}_i$ under the canonical bijection $q\colon (F\cap H)\backslash H\to F\backslash FH$ is still an $R$-disjoint union of subspaces of diameter at most $S$.
	
	 {Let $F$ be a finite subgroup of $G'$, let $h,h'$ be elements of $H$ and let $g,g'$ be elements of $G$.
	 If we have $d(Fgh,Fg'h')<R$ in $F\backslash G'$, then there} is an $f$ in $F$ with $d(e, h^{-1}g^{-1}fg'h')<R$, and hence $h^{-1}g^{-1}fg'h'\in H$.
	 Therefore, $g^{-1}fg'\in H$, so $FgH=Fg'H$. Hence the quotient $F\backslash G'$ is an $R$-disjoint union of spaces of the form $F\backslash FgH$.
	
	 For $g$ in $G $ we set $F^g:=g^{-1}Fg$. For every $h$ in $H$ we have the equalities
	 \[\min_{f\in F}d(gh,fgh')=\min_{f\in F}d(h,g^{-1}fgh')=\min_{f'\in F^g}d(h,f'h {'})\ ,\]
	  i.e., the map $F\backslash FgH \to F^g\backslash F^gH, Fgh \mapsto F^g h$ is an isometry.
		Hence we can use the covers for the spaces $ {F^g \cap H}\backslash H$  {as $g$ varies} to obtain for every $F\backslash G'$ a cover $U_0\cup\ldots\cup U_n$, such that for every $i$ in $\{0,\dots,n\}$ the subset $U_i$ is an $R$-disjoint union of subspaces of diameter at most $S$. This shows that $\{F\backslash G'\}_{F\in \Fin(G')}$ has asymptotic dimension at most~$n$.
	
	For $g$ in $G$ and $F$ a finite subgroup of $G'$, $F\backslash FgG'$ is isometric to $F^g\backslash G'$ as above. Since $G'$ is normal in $G$, the group $F^g$ is again a finite subgroup of $G'$.
	Therefore, every element of $\{F\backslash G\}_{F\in \Fin(G')}$ is a union of at most $k$ subspaces isometric to elements of $\{F\backslash G'\}_{F\in \Fin(G')}$. Hence also $\{F\backslash G\}_{F\in \Fin(G')}$ has asymptotic dimension at most $n$ by the Finite Union Theorem of \cite{bell-dranishnikov}.
	
	Every finite subgroup $F$ of $G$ has a normal subgroup $F'$ of index at most $k$ contained in~$G'$. Then $ {F'\backslash F}$ acts isometrically on $F'\backslash G$ with quotient $F\backslash G$. Hence we can again apply \cite[Cor.~1.2]{KasQuot} to see that $\{F\backslash G\}_{F\in \Fin}$ has asymptotic dimension at most $n$.
\end{proof}

Let 
\[1\to S\to G\xrightarrow{\phi} Q\to 1\]
be an extension of countable groups and let $S'$ be a subgroup of $S$ that is normal in $G$. 
\begin{lem}
	\label{lem:dimension}
	Assume:
	\begin{enumerate}
		\item $S$ is elementary amenable with finite Hirsch length $n$;
		\item $Q$ admits a $k$-dimensional model for $E^{top}_{\Fin(Q)} Q$.
	\end{enumerate}
	Then $G/S'$ admits an $n+k+2$-dimensional model for $E^{top}_{\Fin(G/S')} G/S'$.
\end{lem}
\begin{proof}
	Consider the extension
	\[1\to S/S'\to G/S'\xrightarrow{p} Q\to 1.\]
	Then $h(S/S')\leq h(S)=n$ and also for every finite subgroup $F$ of $Q$, we have \[h(p^{-1}(F))=h(S/S')+h(F)\leq n+0=n.\]
	Hence by Flores and Nucinkis \cite[Cor.~4  {and the discussion preceding it}]{flores-nucinkis}, there exists a model for $E^{top}_{\Fin(p^{-1}(F))}p^{-1}(F)$ of dimension at most $n+2$. 
	Since $Q$ admits a $k$-dimensional model for $E^{top}_{\Fin(Q)} Q$ and for every finite subgroup $F$ of $Q$ there exists a model for $E^{top}_{\Fin(p^{-1}(F))}p^{-1}(F)$ of dimension at most $n+2$, there is an $n+k+2$-dimensional model for $E^{top}_{\Fin(G/S')} G/S'$ by \cite[Thm.~5.16]{lueck-survey}.
\end{proof}

Let 
\[1\to S\to G\xrightarrow{\phi} Q\to 1\]
be an extension of groups. Denote by $\Fin(Q)$ the family of finite subgroups of $Q$. 
By $\phi^{-1}(\Fin(Q))$ we denote the family of subgroups of $G$ whose image under $\phi$ belongs to $\Fin(Q)$.
Let $M\colon G\Orb\to\bC$ be a functor.
\begin{theorem}
	\label{thm:solv}
	Assume:
	\begin{enumerate}
		\item $M$ is a hereditary CP-functor;
		\item $S$ is virtually solvable and has Hirsch length $n<\infty$;
		\item\label{gegregregrege} $Q$ admits a finite dimensional model for $E^{top}_{\Fin(Q)} Q$.
	\end{enumerate}
	Then the relative assembly map $\As_{\Fin,M}^{\phi^{-1}(\Fin(Q))}$ admits a left-inverse.
\end{theorem}
\begin{proof} 
	We argue by induction on the derived length $k$ of $S$.
	
	If {$k=1$}, then $S$ is virtually abelian and every group in $\phi^{-1}(\Fin(Q))$ is virtually abelian {of Hirsch length at most $n$}, too. Hence the statement follows from case \ref{thm:main-it2} of \cref{thm:main-injectivity-kor1} since its assumptions are verified by \cref{lem:virtab} and \cref{lem:dimension} applied with $S'$ the trivial group.
	
	Now suppose that the statement holds for $k$ and assume $S$ has derived length $k+1$. Note that $[S,S]$ is normal in $G$ and has derived length $k$. We set $G^{\prime}:=G/[S,S]$. Then there is a finite dimensional model for $E^{top}_{\Fin(G')}G'$ by \cref{lem:dimension}. We consider the factorization of $\phi$ as 
	\[\phi\colon G\stackrel{\psi}{\to} G^{\prime} \stackrel{p}{\to} Q\ .\]
	The inclusions
	\[\Fin\subseteq \psi^{-1} (\Fin(G^{\prime}))\subseteq \phi^{-1}(\Fin(Q))\]
	of families of subgroups of $G$ induce a factorization 
	\[\As_{\Fin,M}^{\phi^{-1}(\Fin(Q))}\simeq  \As^{\phi^{-1}(\Fin(Q))}_{ \psi^{-1} (\Fin(G^{\prime})),M}\circ \As_{\Fin,M}^{ \psi^{-1} (\Fin(G^{\prime}))}\] of the relative assembly map. Because $\As_{\Fin,M}^{\psi^{-1}(\Fin(G^{\prime}))}$ admits a left-inverse by the induction assumption, it remains to show that $  \As^{\phi^{-1}(\Fin(Q))}_{ \psi^{-1} (\Fin(G^{\prime})),M}$ admits a left-inverse. We have a commuting diagram  of categories \[\xymatrix{G^{\prime}_{\Fin(G^{\prime})}\Orb\ar[r]^-{\Res_{\psi}}\ar[d]&G_{\psi^{-1}(\Fin(G^{\prime}))}\Orb\ar[d]\\G^{\prime}_{p^{-1}(\Fin(Q))}\Orb\ar[r]^-{\Res_{\psi}}&G_{\phi^{-1}(\Fin(Q))}\Orb}\]
	where the vertical functors are the fully faithful inclusions induced by the inclusions of families
	$\Fin(G^{\prime})\subseteq p^{-1}(\Fin(Q))$ and $\psi^{-1}(\Fin(G^{\prime}))\subseteq \phi^{-1}(\Fin(Q))$.
	We now note that the horizontal maps are fully faithful inclusions as well and cofinal. 
	We obtain an induced square in $\bC$
	\[\xymatrix{\colim_{G^{\prime}_{\Fin(G^{\prime})}\Orb}M\circ \Res_{\psi}\ar[r]^-{\simeq}\ar[d]^{\As^{p^{-1}(\Fin(Q))}_{\Fin(G^{\prime}),M\circ \Res_{\psi}}}&\colim_{G_{\psi^{-1}(\Fin(G^{\prime}))}\Orb}M\ar[d]^{ \As^{\phi^{-1}(\Fin(Q))}_{ \psi^{-1} (\Fin(G^{\prime})),M}}\\ \colim_{G^{\prime}_{p^{-1}(\Fin(Q))}\Orb} M\circ \Res_{\psi}\ar[r]^-{\simeq}&\colim_{G_{\phi^{-1}}(\Fin(Q))\Orb}M}\]
	The existence of a left-inverse   of $\As_{\Fin {(G')},M\circ\Res_{\psi}}^{p^{-1}(\Fin(Q))}$ again follows from the case  {$k=1$} since
	$M\circ \Res_{\psi}$ is also a  CP-functor. 
\end{proof}
\begin{rem}
For algebraic $K$-theory $K\bA^{G}$ (see \cref{ergerge}) in place of $M$ and under the same assumptions on $S$ and $G$ 
  as in  \cref{thm:solv} the existence of a left-inverse for
  $\As^{\phi^{-1}(\Fin(Q))}_{\Fin,K\bA^{G}}$ has been shown by combining the split-injectivity of the relative assembly map from finite to virtually cyclic subgroups with the Farrell--Jones conjecture for solvable groups, cf.~\cite[Prop.~4.1]{KasLie}.  
  With the new techniques to understand relative assembly maps developed in this article the use of the Farrell--Jones conjecture can be avoided. 
\end{rem}

For convenience, we repeat the arguments from \cite{KasLie} and \cite{KasLin} to obtain split-injectivity for finitely generated subgroups of linear groups and  {of} virtually connected Lie groups with a finite-dimensional classifying space.

Let $M\colon G\Orb\to \bC$ be a functor.
\begin{kor}
\label{cor:injectivity-linear}
	Assume:
	\begin{enumerate}
		\item $M$ is a hereditary CP-functor;
		\item $G$ admits a finite-dimensional model for $E^{top}_\Fin G$;
		\item $G$ is a finitely generated subgroup of a linear group over a commutative ring with unit or of a virtually connected Lie group.
	\end{enumerate}
	Then the assembly map $\As_{\Fin,G}$ is split injective.
\end{kor}
\begin{proof}
	Let $G$ be a finitely generated subgroup of a virtually connected Lie group. The adjoint representation induces an extension with abelian kernel and quotient a finite index supergroup $Q$ of a finitely generated subgroup of $GL_n(\mathbb{C})$. The group $Q$ has $Q_\Fin$-FDC by \cref{ex:groups1}. Since $G$ admits a finite-dimensional model for $E_\Fin^{top}G$, so does $Q$ using the characterization from \cref{ex:groups2}. By \cref{kor:oldversion} the assembly map $\As_{\Fin(Q),M\circ \Res_Q^G}$ is split-injective. This assembly map is equivalent to $\As_{p^{-1}(\Fin(Q)),M}$, where $p\colon G\to Q$ is the projection. Because the kernel of $p$ is abelian, the assembly map $\As_{\Fin,M}^{p^{-1}(\Fin(Q))}$ is split-injective by \cref{thm:solv}.
	
	Now let $G$ be a finitely generated subgroup of a linear group over a commutative ring $R$ with unit. Let $\mathfrak{n}$ be the nilradical of $R$. Then we have an extension
	\[1\to (1+M_n(\mathfrak{n}))\cap G\to G\xrightarrow{p} Q\to 1\ ,\]
	where $Q$ is a finitely generated subgroup of $GL_n(R/\mathfrak{n})$.
	Arguing as above, the assembly map $\As_{p^{-1}(\Fin(Q)),M}$ is split-injective by \cref{ex:groups1} since $R/\mathfrak{n}$ has trivial nilradical.
	Since {the group} $(1+M_n(\mathfrak{n}))$ is nilpotent, the assembly map $\As_{\Fin,M}^{p^{-1}(\Fin(Q))}$ is split-injective by \cref{thm:solv}.
\end{proof}

\section{\texorpdfstring{$\boldsymbol{G}$}{G}-bornological coarse spaces and coarse homology theories}
\label{sec:homologytheories}

In this section we recall the definition of the category $G\BC$ of $G$-bornological coarse spaces and provide basic examples. 
We further recall the notion of an equivariant coarse homology theory, in particular its universal version $\Yo^{s}$  with values in the stable $\infty$-category $G\Sp\cX$ of
equivariant coarse motivic spectra.
Most of this material has been developed in 
\cite{equicoarse} (see also \cite{buen} for the non-equivariant case).

In the  definitions below we will use the following notation. 
\begin{enumerate}
\item For a set $Z$ we let $\cP(Z)$ denote the power set of $Z$. 
\item If  a group $G$ acts on a set $X$, then it acts diagonally on $X\times X$ and therefore on $\cP(X\times X)$.
For $U$ in $\cP(X\times X)$ we set \begin{equation*}\label{betrboklvree}
GU:=\bigcup_{g\in G} gU\ .
\end{equation*} 
\item For $U$ in $\cP(X\times X)$ and $B$ in $\cP(X)$ we define the
 $U$-{\em thickening} $U[B]$   by
\begin{equation*}\label{viojoewfwefvvdvsdv} U[B]:=\{x\in X \mid \exists y\in B: (x,y) \in U \}\ .\end{equation*}
\item For $U$ in $\cP(X\times X)$ we define the {\em inverse} by \[U^{-1}:=\{(y,x) \mid (x,y)\in U  \}\ .\]  \item For $U,V$ in $\cP(X\times X)$ we define their {\em composition}  by \begin{equation}\label{rv3roih43iuoff3fwe}
U\circ V:=\{(x,z) \mid \exists y\in X : (x,y)\in U \wedge (y,z)\in V  \}\ .
\end{equation}
\end{enumerate}

Let $G$ be a group and let $X$ be a $G$-set.

\begin{ddd}
 A \emph{$G$-coarse structure} $\cC$ on $X$ is a subset of $\cP(X\times X)$ with the following properties:
\begin{enumerate}
\item $\cC$ is closed under composition, inversion, and forming finite unions or subsets;
\item $\cC$ contains the diagonal $\diag(X)$ of $X$;
\item for every $U$ in $\cC$, the set $GU$
is also in $\cC$. 
\end{enumerate}
The pair $(X,\cC)$ is called a \emph{$G$-coarse space},
 {and the members of $\cC$ are called (coarse) \emph{entourages} of $X$.}
\end{ddd}

Let $(X,\cC)$ and $(X^{\prime},\cC^{\prime})$ be $G$-coarse spaces and  {let} $f\colon X\to X^{\prime}$ be an equivariant map between the underlying sets.
\begin{ddd}
The map $f$ is \emph{controlled} if for every $U$ in $\cC$ we have  $(f\times f)(U)\in \cC^{\prime}$.  \end{ddd}

We obtain a category $G\Coarse$ of $G$-coarse spaces and controlled equivariant maps.

\begin{ddd}\label{erjgoijgoiergjergegrrr}A \emph{$G$-bornology} $\cB$ on $X$ is a subset of $\cP(X)$ with the following properties:
\begin{enumerate}
\item $\cB$ is closed under {forming} finite unions and subsets;
\item $\cB$ contains all finite subsets of $X$;
\item $\cB$ is $G$-invariant. 
\end{enumerate} 
The pair $(X,\cB)$ is called a \emph{$G$-bornological space}, and the members of $\cB$ are called \emph{bounded subsets} of $X$.
\end{ddd}

Let $(X, {\cB})$ and $(X^{\prime}, {\cB^{\prime}})$ be $G$-bornological spaces and {let} $f \colon X\to X^{\prime}$ be an equivariant map between the underlying sets.

\begin{ddd}\label{erjgoijgoiergjergegrrr1}
 {The map} $f$ is \emph{proper} if for every $B^{\prime}$ in $\cB^{\prime}$ we have $f^{-1}(B^{\prime})\in  \cB$.   \end{ddd}
 
We obtain a category $G\Born$ of $G$-bornological spaces and proper equivariant maps.

Let $X$ be a $G$-set with a $G$-coarse structure $\cC$ and a $G$-bornology $\cB$.
\begin{ddd}
The coarse structure $\cC$ and the bornology $\cB$ are said to be  
  \emph{compatible} if 
for every $B$ in $\cB$ and $U$ in $\cC$ 
 {the $U$-thickening $U[B]$ lies in $\cB$.}
\end{ddd}

\begin{ddd}\label{ergergregreg5t45t}
A \emph{$G$-bornological coarse space} is a triple $(X,\cC,\cB)$ consisting of a $G$-set $X$, a $G$-coarse structure $\cC$, and a $G$-bornology $\cB$ such that $\cC$ and $\cB$  are  compatible. \end{ddd}
 
\begin{ddd}
A \emph{morphism} $f\colon (X,\cC,\cB)\to (X^{\prime},\cC^{\prime},\cB^{\prime})$ between $G$-bornological coarse spaces
is an equivariant  map $f\colon X\to X^{\prime}$ of the underying $G$-sets which is controlled and proper.
\end{ddd}

We obtain a category $G\BC$ of $G$-bornological coarse spaces and morphisms.
If the structures are clear from the context, we will use the notation  $X$ instead of $(X,\cC,\cB)$ in order to denote $G$-bornological coarse spaces.

Let $X$ be a $G$-set.

\begin{ex}\label{ufewiofheiwfhuewifewew}
If $W$ is a subset of $\cP(X\times X)$, then the $G$-coarse structure generated by $W$ is the minimal $G$-coarse structure containing $W$,
 i.e., it is the coarse structure $\cC\langle \{GU\mid U\in W\}\rangle$  generated by the set of invariant entourages  $GU$ for all $U$ in $W$.

We can define the following $G$-coarse structures on $X$:
\begin{enumerate}
\item The \emph{minimal} coarse structure on $X$ is the $G$-coarse structure generated by the empty family.
 It consists of all subsets of $\diag(X)$. {We denote the corresponding $G$-coarse space by $X_{min}$.}
 \item The \emph{canonical} coarse structure on $X$ is the $G$-coarse structure generated by the  entourages $ B\times B$ for all finite subsets $B$ of $X$.
  {We denote the corresponding $G$-coarse space by $X_{can}$.}
   \item $\cP(X\times X)$ is the \emph{maximal} coarse structure on $X$.  {We denote the corresponding $G$-coarse space by $X_{max}$.}
 \item If $X$ comes equipped with a  quasi-metric\footnote{The notion of a quasi-metric generalizes the notion of a metric. The difference is that for a quasi metric we admit the value $\infty$.} $d$, then the \emph{metric} coarse structure on $X$ is generated by the subsets
 $\{ (x,y) \mid d(x,y) \leq r \} $ of $X\times X$
  for all  $r$ in $[0,\infty)$.
   {We denote the corresponding  coarse space by $X_d$.}
  If the quasi-metric $d$ is  {$G$-}invariant, then  we obtain a $G$-coarse structure and $X_{d}$ is a $G$-coarse space.
\end{enumerate}

If $A$ is a subset of $\cP(X)$, then the $G$-bornology generated by $A$ is the minimal $G$-bornology containing $A$, i.e., it is the bornology $\cB\langle \{gB\mid g\in G, B\in A\}\rangle$ generated by the set of all $G$-translates of elements of $A$.

We can define the following $G$-bornologies on $X$:
\begin{enumerate}
 \item The \emph{minimal} $G$-bornological structure consists of the finite subsets.
  {We denote the corresponding $G$-bornological space by $X_{min}$.}
 \item The \emph{maximal} $G$-bornological structure consists of all subsets.
  {We denote the corresponding $G$-bornological space by $X_{max}$.}
 \item If $X$ comes equipped with a  quasi-metric $d$, the \emph{metric bornology} on $X$ is generated by the sets
  $\{ y \mid d(x,y) \leq r \}$
  for all  $x$ in $X$ and $r$ in $[0,\infty)$.
   {We denote the corresponding bornological space by $X_d$.}
  If $d$ is $G$-invariant, then we obtain a $G$-bornology and $X_{d}$ is a $G$-bornological space.
\end{enumerate}
Taking any pair of compatible coarse and bornological structures as above, we can form a $G$-bornological coarse space.
These will be denoted by two subscripts, where the first subscript refers to the coarse structure and the second subscript to the bornology.
Examples include $X_{can,min}$, $X_{can,max}$, $X_{min,min}$, $X_{min,max}$, $X_{max,max}$ and, if $X$ comes equipped with an invariant metric, $X_{d,d}$. 
\end{ex}

Let $X$  be a $G$-coarse space with coarse structure $\cC$. Then
\begin{equation}\label{vfwrekjnkewjfwefewfwef}
\cR_{\cC}:=\bigcup_{U\in \cC} U
\end{equation} is an invariant equivalence relation on $X$.

\begin{ddd}\label{fiofhjweio23r23r23r}
We let $\pi_{0}(X)$ denote the $G$-set of equivalence classes with respect to $\cR_{\cC}$.
 {The elements of $\pi_{0}(X)$ are called the \emph{coarse components} of $X$.}
\end{ddd}

\begin{ddd}
 A $G$-coarse space $(X,\cC)$ is \emph{coarsely connected} if $\pi_{0}(X)$ {is a singleton set.} \end{ddd}

We now introduce the notion of an equivariant coarse homology theory{; see} \cite[Sec.~3]{equicoarse} for details.  

Let $X$ be a $G$-bornological coarse space. 
\begin{ddd}
	An \emph{equivariant big family} on $X$ is a filtered family of $G$-invariant subsets $(Y_{i})_{i\in I}$ of $X$  such that for every entourage $U$ of $X$ and $i$ in $I$ there exists $j$ in $I$ such that $U[Y_{i}]\subseteq Y_{j}$.
	
	An \emph{equivariant complementary pair} $(Z,\cY)$ on $X$ is a pair of a $G$-invariant subset $Z$ of $X$ and an equivariant big family $\cY=(Y_{i})_{i\in I}$ on $X$ such that there exists $i$ in $I$ with $Z\cup Y_{i}=X$.
\end{ddd}

Let $g,f\colon X\to X'$ be two morphisms in $G\BC$. Then  we say that $f$  is close to $g$ if $(f\times g)(\diag(X))$ is a coarse entourage of $X'$. This notion will be used in Condition~\ref{qiorffweqewfqew}.\ of the definition below.

Let $X$ be a $G$-bornological coarse space.
\begin{ddd}  The space $X$ is \emph{flasque} if it admits a morphism $f\colon X\to X$ such that
	\begin{enumerate}
		\item \label{qiorffweqewfqew} $f$ is close to $\id_{X}$.
		\item For every entourage $U$ of $X$ the subset $\bigcup_{n\in \nat}(f^{n}\times f^{n})(U)$ is an entourage of $X$.
		\item For every bounded subset $B$ of $X$ there exists an integer  $n$ such that
		$GB\cap f^{n}(X)=\emptyset$.
	\end{enumerate}
	
	We say that flasqueness of $X$ is \emph{implemented} by $f$.
\end{ddd}

 The category $G\BC$ has a symmetric monoidal structure $\otimes${; see} \cite[Ex.~2.17]{equicoarse}. 
 If $X$ and $Y$ are $G$-bornological coarse spaces, then
 $X\otimes Y$ has the following description:
 \begin{enumerate}
 \item The underlying $G$-coarse space of $X\otimes Y$ is the  {cartesian} product in $G\Coarse$ of the underlying $G$-coarse spaces of $X$ and $Y$. More explicitly, the underlying $G$-set of $X\otimes Y$ is $X\times Y$ with the diagonal $G$-action, and the coarse structure is generated by the entourages $U\times V$ for all coarse entourages $U$ of $X$ and $V$ of $Y$.
 
 \item  The bornology on $X\otimes Y$ is generated by the products $A\times B$ for all bounded subsets $A$ of $X$ and $B$ of $Y$.
 \end{enumerate}
 Note that $X\otimes Y$ in general differs from the {cartesian} product $X\times Y$ in $G\BC$.

Let $\bC$ be a cocomplete stable $\infty$-category and let
\[E\colon G\BC\to \bC\]
be a functor.  If $\cY=(Y_{i})_{i\in I}$ is a filtered family of  $G$-invariant subsets of $X$, then we set
\begin{equation}
E(\cY):=\colim_{i\in I} E(Y_{i})\ .
\end{equation} 
In this formula we consider the subsets $Y_{i}$ as $G$-bornological coarse spaces with the structures induced from $X$.

If $Z$ is another invariant subset, then we use the notation $Z\cap \cY:=(Z\cap Y_{i})_{i\in I}$.

Let $\bC$ be a {cocomplete} stable $\infty$-category and consider a functor
\[ E \colon G\BC\to \bC\ .\]
\begin{ddd}\label{def:coarsehom}
	A \emph{$G$-equivariant $\bC$-valued coarse homology theory} is a functor
	\[E\colon G\BC\to \bC\] with the following properties:
	\begin{enumerate}
		\item (Coarse invariance) For all $X$ in $G\BC$ the functor $E$ sends the projection $\{0,1\}_{max,max}\otimes  X\to X$ to an equivalence.
		\item(Excision) $E(\emptyset)\simeq 0$ and for every equivariant complementary pair   $(Z,\cY)$ on  a  $G$-bornological coarse space $X$ the square  
		\[\xymatrix{E(Z\cap \cY)\ar[r]\ar[d]&E(Z)\ar[d]\\E(\cY)\ar[r]&E(X)}\]
		is  a push-out.
		\item (Flasqueness) If  a  $G$-bornological coarse space $X $ is  flasque, then $E(X)\simeq 0$.
	\item(u-Continuity) For every  $G$-bornological coarse space $X $ the natural map
	\[\colim_{U\in \cC^{G}(X)} E(X_{U})\to E(X)\]
	is an equivalence. Here $X_U$ denotes the $G$-bornological coarse space $X$ with the coarse structure replaced by the one generated by $U$, and $\cC^{G}(X)$ is the  poset of $G$-invariant coarse entourages of $X$.
\end{enumerate}
If the group $G$ is clear from the context, then we will often just speak of an equivariant coarse homology theory.
\end{ddd}

We have a universal equivariant coarse homology theory 
\[ \Yo^{s} \colon G\BC\to G\Sp\cX \]
(see \cite[Def. 4.9]{equicoarse}), where $G\Sp\cX$ is a stable presentable $\infty$-category called the category of coarse motivic spectra.
{More precisely, we have the following.
\begin{prop}[{\cite[Cor.~4.9]{equicoarse}}]\label{prop:motives-universal-prop} 
 Restriction along $\Yo^s$ induces an equivalence between the $\infty$-categories of colimit-preserving functors $G\Sp\cX \to \bC$ and $\bC$-valued equivariant coarse homology theories.
\end{prop}}
 
  The symmetric monoidal structure $\otimes$ descends to $G\Sp\cX$ such that $\Yo^{s}$ becomes a symmetric monoidal functor \cite[Lem. 4.17]{equicoarse}.
  
   \begin{ex}\label{eriogjoqergwfdefwefqwef}
  The following is an illustrative example of the usage of some of the axioms  of a coarse homology theory for $\Yo^{s}$. Let $X$ be in $G\BC$. On    $\R\otimes X$ we consider the subset $Z:=[0,\infty)\times X$ and
  the big family $\cY:= ((-\infty,n]\times X)_{n\in \nat}$. Then $(Z,\cY)$ is a complementary pair on $\R\otimes X$. By the excision  axiom we get a push-out square
 \begin{equation}\label{qwefqwefweewfqwefqwfwef}
\xymatrix{\Yo^{s}(Z\cap \cY) \ar[r]\ar[d]&\Yo^{s}(Z)\ar[d]\\\Yo^{s}(\cY)\ar[r]&\Yo^{s}(\R\otimes X)}\ . 
\end{equation}  
  We now observe that $Z$ is flasque with flasqueness implemented by the map $f(t,x):=(t+1,x)$. Similarly, all members of $\cY$ are flasque. Since $\Yo^{s}$ vanishes on flasques, we get $\Yo^{s}(Z)\simeq 0$ and $\Yo^{s}(\cY)\simeq 0$. The inclusion
  $X\cong \{0\}\times X\to \R\times X$ induces an equivalence of $X$ with every member of $Z\cap \cY$. Consequently, we have a canonical equivalence $\Yo^{s}(X)\simeq \Yo^{s}(Z\cap \cY)$. Therefore, the push-out square in \eqref{qwefqwefweewfqwefqwfwef} is equivalent to a push-out square   \begin{equation*}\label{qwefqwefweeeccwcwecwecwcwfqwefqwfwef}
\xymatrix{\Yo^{s}(X) \ar[r]\ar[d]&0\ar[d]\\0\ar[r]&\Yo^{s}(\R\otimes X)}\ . 
\end{equation*}  
This square provides an equivalence  \begin{equation}\label{qwefijqweofqewfqewfqewfqewfq}
\Sigma \Yo^{s}(X)\simeq \Yo^{s}(\R\otimes X)\ .\qedhere
\end{equation}
   \end{ex}

 Let $E\colon G\BC\to \bC$ be a functor  and let $X$ be a $G$-bornological coarse space.
\begin{ddd}\label{def:twist}
 The \emph{twist} $E_X$ of $E$ by $X$ is the functor \[ E(X \otimes -) \colon G\BC \to \bC\ .\qedhere \]
\end{ddd}
\begin{lem}\label{lem:twist} If $E$ is an equivariant  coarse homology theory, then
 the twist $E_X$ is an equivariant coarse homology theory, too.
\end{lem}
\begin{proof}
 This follows from \cite[Lem.~4.17]{equicoarse}.
\end{proof}

Let $(X,\cB)$ be a $G$-bornological space.
\begin{ddd}\label{rgieog3434gergeg}
 A subset $F$ of $X$ is   \emph{locally finite} if $F \cap B$ is a finite set for every~$B$  in~$\cB$.
\end{ddd}

{\em Continuity} is an additional property  of  an equivariant coarse homology theory $E$. We refer to  
\cite[Def.~5.15]{equicoarse} for the  precise definition. For our purposes, it suffices to know the following.

{Let $X$ be a $G$-bornological coarse space and let $\cL(X)$ denote the poset of all $G$-invariant locally finite subsets of the underlying bornological  space of $X$. We consider $F$ in $\cL(X)$ with the  $G$-bornological coarse structure induced from $X$.}
\begin{lem} [{\cite[Rem. 5.16]{equicoarse}}]\label{rguierog34trt34g}
   If $E$ is continuous, then  
    the canonical map
\[ \colim_{F \in \cL(X)} E(F) \to E(X) \]
is an equivalence.
\end{lem}

In order to capture continuity of equivariant coarse homology theories motivically we  introduce the universal 
continuous equivariant coarse homology theory
\begin{equation}\label{brtoij4o5g4fwefwefewfewfbbt4}
 \Yo^{s}_{c} \colon G\BC\to G\Sp\cX_{c}
\end{equation} 
 whose target $G\Sp\cX_c$ is  the stable presentable $\infty$-category
  of continuous equivariant motivic coarse spectra (see \cite[Def.~5.21]{equicoarse}). 
 \begin{prop} [{\cite[Cor.~5.22]{equicoarse}}] \label{prop:motives-universal-prop-cont}    Restriction along $\Yo_{c}^s$ induces an equivalence between the $\infty$-categories of  colimit-preserving functors $G\Sp\cX_{c} \to \bC$ and $\bC$-valued continuous equivariant coarse homology theories.
\end{prop}
 {We have a canonical colimit-preserving functor \begin{equation}\label{oih4oifhoifh23f2f23dd2d2d2}
C^s\colon G\Sp\cX\to G\Sp\cX_c
\end{equation}
    such that $\Yo^s_c \simeq C^s \circ \Yo^s$ (see \cite[(5.6)]{equicoarse}).}
\begin{ddd}\label{grioer34g43g43ggg4}
 A morphism in $G\Sp\cX$ or $G\BC$ is a \emph{continuous equivalence} if it becomes an equivalence after application of $C^{s}$ or $\Yo^{s}_{c}$, respectively.
 
 Two morphisms in $G\Sp\cX$ or $G\BC$ are \emph{continuously equivalent} if they become equivalent after application of $C^{s}$ or $\Yo^{s}_{c}$, respectively.
\end{ddd}

\section{Cones and the forget-control map}\label{sec:cones}

In this section we recall the cone construction and the cone sequence. We further  introduce the  forget-control map and show its  compatibility with induction and twisting.
 
We start with discussing $G$-uniform bornological coarse spaces and  the cone construction.
Let $X$ be a $G$-set. 
\begin{ddd}\label{rgwefwwefwefef}
A $G$-{\em uniform structure} on $X$ is  a subset  $\cU$ of $\cP(X\times X)$  with the following properties:
\begin{enumerate}
\item Every element of $\cU$ contains the diagonal;
\item $\cU$ is  {closed under inversion}, composition, finite intersections, and supersets;
\item for every $U$ in $\cU$ there exists $V$ in $\cU$ with $V\circ V\subseteq U$;
\item {for every $U$ in $\cU$ we have $\bigcap_{g\in G} gU \in \cU$.}\qedhere
\end{enumerate}
\end{ddd}
The first three conditions define the notion of a uniform structure, and the last condition reflects the compatibility with the action of $G$. A $G$-{\em uniform space} is a pair $(X,\cU)$ of a $G$-set $X$ and a $G$-uniform structure $\cU$. 

Let $(X,\cU)$ and $ (X^{\prime},\cU^{\prime})$ be $G$-uniform spaces and $f\colon X\to X^{\prime}$ be an equivariant map between the underlying sets.

\begin{ddd}
$f$ is {\em uniform}  if $f^{-1}(U^{\prime})\in \cU$
for every $U^{\prime}$ in $\cU^{\prime}$.
\end{ddd}

Let $X$ be a $G$-set with a $G$-uniform structure $\cU$ and a $G$-coarse structure $\cC$.
\begin{ddd}
We say that $\cU$ and $\cC$ are {\em compatible} if $\cU\cap \cC$ is not empty.
\end{ddd}

\begin{ddd}[{\cite[Def.~9.9]{equicoarse}}]\label{rgfergerg} 
 A $G$-{\em uniform bornological coarse space} is a tuple $(X,\cC,\cB,\cU)$, where $(X,\cC,\cB)$ is a $G$-bornological coarse space and $\cU$ is a $G$-uniform structure which is compatible with $\cC$.
 \end{ddd}
 
\begin{ddd}
A {\em morphism}   between $G$-uniform bornological coarse spaces  \[f\colon (X,\cC,\cB,\cU)\to (X^{\prime},\cC^{\prime},\cB^{\prime},\cU^{\prime})\]
 is a morphism between $G$-bornological coarse spaces $f\colon (X,\cC,\cB)\to (X^{\prime},\cB^{\prime},\cC^{\prime})$ 
 which, as a morphism $(X,\cU)\to (X^{\prime},\cU^{\prime})$, is uniform.
\end{ddd}

We obtain the category $G\UBC$ of $G$-\emph{uniform bornological coarse spaces}.
 {We have the forgetful functor \begin{equation}\label{veroihi3uuhfove}
\cF\colon G\UBC\to G\BC
\end{equation}  which forgets the uniform structure.}

 \begin{ex}\label{wrgoijo42reg} Let $X$ be a $G$-set with a quasi-metric $d$.
 Then we get a uniform structure on $X$ generated by the subsets
 $\{(x,y)\in X\times X\mid d(x,y)<r\}$ for all $r$ in $(0,\infty)$. We let $X_{d}$ denote the corresponding uniform space.
 If $d$ is invariant, then  we obtain a $G$-uniform structure and $X_{d}$ is a $G$-uniform space.

 Expanding the notation for $G$-bornological coarse spaces, we use triple subscripts to indicate $G$-uniform bornological coarse spaces,
 where the first subscript indicates the $G$-uniform structure, the second subscript indicates the $G$-coarse structure, and the third subscript indicates the $G$-bornology.
 
 In particular, if $X$ is a $G$-set with an invariant quasi-metric $d$, then we obtain the $G$-uniform bornological coarse spaces $X_{d,d,d}$ and $X_{d,max,max}$.
\end{ex}

\begin{ex}
 {Let $S$ be a $G$-set. Then the $G$-bornological coarse space $S_{min,min}$ equipped with the uniform structure containing all supersets of the diagonal is a $G$-uniform bornological coarse space
 which we denote by $S_{disc,min,min}$.}
\end{ex}

 {Let $X$ be a $G$-uniform bornological coarse space and let $\cY = (Y_i)_{i \in I}$ be an equivariant big family. Let $\cC$ and $\cU$ denote the coarse and uniform structures of $X$.}
\begin{ddd}[{\cite[Def.~9.15]{equicoarse}}]\label{wfefuehwifefw}
An order-preserving function
\[\psi \colon I\to \cP(X\times X)^{G}\]
(where we consider the target with the opposite of the inclusion relation) is \emph{$\cU$-admissible} if for every $U$ in $\cU^{G}$ there is $i$ in $I$ such that $\psi(i)\subseteq U$.
 Given a function $\psi \colon I\to \cP(X\times X)^{G}$ we define the entourage
 \[ U_{\psi}:=\bigcup_{i\in I} \left[(Y_{i}\times Y_{i})\cup \psi(i)\right]\ .\]
 The {\em hybrid structure} $\cC_{h}$ on $X$ is the $G$-coarse structure generated by the entourages
 $U\cap U_{\psi}$ for all $U$ in $\cC^{G}$ and all $\cU$-admissible functions $\psi$.
 
 We let $X_{h}$ denote the bornological coarse space obtained from $X$  by forgetting the uniform structure and replacing the coarse structure by the hybrid coarse structure.
\end{ddd}
Since $\cC_{h}\subseteq \cC$ by construction, we have a morphism of $G$-bornological coarse spaces $X_{h}\to \cF(X)$, where $\cF$ is the forgetful functor \eqref{veroihi3uuhfove}.

 {\begin{ddd}\label{345go3ihjoi43gg34g}
We have the functor
\[ \cO^{\infty}_{geom} \colon G\UBC\to G\BC \]
which sends a $G$-uniform bornological coarse space $X$ to the $G$-bornological coarse space
\begin{equation*}
 \cO^{\infty}_{geom}(X):=(\R\otimes X)_{h}\ ,
\end{equation*}
where $\R:=\R_{d,d,d}$ is the $G$-uniform bornological coarse space with structures induced from the standard metric and the trivial $G$-action.
The subscript $h$ stands for the hybrid coarse structure associated to the equivariant big family $((-\infty,n]\times X)_{n\in \nat}${; see} \cref{wfefuehwifefw}.

If $f\colon X\to X'$ is a morphism in $G\UBC$, then 
$\cO^{\infty}_{geom}(f)\colon \cO^{\infty}_{geom}(X)\to \cO^{\infty}_{geom}(X')$ is given by the map
$\id_{\R}\times f\colon \R\times X\to \R\times X'$.
\end{ddd}}

\begin{ddd}\label{iwoijgwegewfefewfefw}
The functor 
 \begin{equation*}
\cO^{\infty}:=\Yo^{s}\circ \cO^{\infty}_{geom} \colon G\UBC\to G\Sp\cX
\end{equation*}
  is called the \emph{cone-at-infinity functor}.
\end{ddd}

 \begin{ddd}
 The \emph{cone functor}
 \[ \cO \colon G\UBC\to G\BC \]
 sends a $G$-uniform bornological coarse space $X$ to
 \[ \cO(X):=([0,\infty)\times X)_{\cO^{\infty}_{geom}(X)}\ , \]
where the subscript indicates that we equip the subset with the structures induced from $\cO^{\infty}_{geom}(X)$.
In particular, $\cO(X)$ is a subspace of $\cO^{\infty}_{geom}(X)$.
\end{ddd}

\begin{rem}
 We refer to \cite[Sec. 9.4 and 9.5]{equicoarse} for more details and properties of these functors.
 Note that $\cO^{\infty}_{geom}$ is denoted by $\cO^{\infty}_{-}$ in the reference. The definition of $\cO^{\infty}$ given above is equivalent to \cite[Def.~9.29]{equicoarse} in view of \cite[Prop.~9.31]{equicoarse}. 
\end{rem}

By \cite[Cor.~9.30]{equicoarse} we have a fibre sequence of functors $G\UBC\to G\Sp\cX$
\begin{equation}\label{feojp2r2}
\dots\to \Yo^{s}\circ \cF\to \Yo^{s}\circ \cO\to \cO^{\infty}\stackrel{\partial}{\to} \Sigma \Yo^{s}\circ \cF\to \dots\ , 
\end{equation}
which is called the \emph{cone sequence}.
 The first map of the cone sequence is induced by the inclusion $X\to [0,\infty)\times X$ given by including the point $0$ into $[0,\infty)$.
The second map is induced by the inclusion $\cO(X)\to \cO^{\infty}_{geom}(X)$. Finally, the cone boundary $\partial$
is given by
\begin{equation}\label{t4blopg34g3g3g}
 {\Yo^s}(\cO^{\infty}_{geom}(X))\to \Yo^{s}(\R\otimes \cF(X))\simeq \Sigma \Yo^{s}(\cF(X))\ ,
\end{equation}
where the first map is induced by the identity of the underlying sets, and the equivalence is the equivalence~\eqref{qwefijqweofqewfqewfqewfqewfq} explained in \cref{eriogjoqergwfdefwefqwef}. We use  \cite[Prop.~9.31]{equicoarse} in order to see that this description of the sequence is equivalent to the original definition from \cite[Cor.~9.30]{equicoarse}.

In various constructions we form a colimit over the poset of invariant entourages $\cC^{G}(X)$ of a $G$-bornological coarse space $X$.
In order to suppress these colimits in an approriate language we use the following procedure.
We let $G\BC^{\cC}$ denote the category of pairs $(X,U)$, where $X$ is a $G$-bornological coarse space and
$U$ is an invariant entourage of $X$ containing the diagonal. A morphism $(X,U)\to (X^{\prime},U^{\prime})$ is a morphism $f\colon X\to X^{\prime}$ in $G\BC$ such that $(f\times f)(U)\subseteq U^{\prime}$.
We have a forgetful functor
\begin{equation}\label{vreoi34fg3vefv}
G\BC^{\cC}\to G\BC\ ,\quad (X,U)\mapsto X\ .
\end{equation}

Let \[F\colon G\BC^{\cC}\to \bC\] be a functor to a cocomplete target $\bC$ and
let $E$ be the left Kan extension of $F$ along~\eqref{vreoi34fg3vefv}. The evaluation of $E$ on a $G$-bornological coarse space $X$ is then given as follows.
\begin{lem} \label{rgoipewerfwefewfew} We have an equivalence
\[
E(X)\simeq \colim_{U\in \cC^{G}(X)} F(X,U)\ .
\]
\end{lem}
\begin{proof}
By the pointwise formula for the left Kan extension we have an  equivalence
\[E(X)\simeq \colim_{((X^{\prime},U^{\prime}),f\colon X^{\prime}\to X) \in G\BC^{\cC}/X} F(X^{\prime},U^{\prime})\ .\]
 If $((X^{\prime},U^{\prime}),f\colon X^{\prime}\to X)$ belongs to $G\BC^{\cC}/X$, then we have a morphism \[(X^{\prime},U^{\prime})\to (X,f(U^{\prime})\cup \diag(X))\] in $G\BC^{\cC}/X$. 
This easily implies that the full subcategory of objects of the form $((X,U),\id_{X})$ of $G\BC^{\cC}/X$ with $U$ in $\cC^{G}(X)$ is cofinal in $G\BC^{\cC}/X$.
 \end{proof}

\begin{constr}
	\label{def:rips}
Let $X$ be a $G$-bornological coarse space and let $U$ be an
invariant entourage of $X$.  Then we can form the $G$-simplicial complex $P_{U}(X)$ of
finitely supported $U$-bounded probability measures on $X$ (see \cite[Def.~11.1]{equicoarse} and the subsequent text). We equip $P_{U}(X)$ with the path quasi-metric in which every simplex has the spherical metric. The path quasi-metric determines the uniform and the coarse structure on $P_{U}(X)$. We equip
$P_{U}(X)$ with the bornology generated by all subcomplexes $P_{U}(B)$ of measures supported on $B$ for a bounded subset $B$ of $X$. 
The resulting $G$-uniform bornological coarse space will be denoted by $P_{U}(X)_{d,d,b}$.
We denote by $P_U(X)_{d,b}$ the underlying bornological coarse space.
Note that the bornology in general differs from the metric bornology which would be indicated by a subscript $d$ in the {last} slot.

Let $f\colon X\to X^{\prime}$ be a morphism of $G$-bornological coarse spaces and $U^{\prime}$ be an invariant entourage of $X^{\prime}$ such that $(f\times f)(U)\subseteq U^{\prime}$. Then the push-forward of measures induces a morphism 
\[f_{*}\colon P_{U}(X)_{d,d,b}\to P_{U^{\prime}}(X^{\prime})_{d,d,b}\] in a functorial way. 
We have thus constructed a functor \begin{equation*}\label{viuhviwevwvwcw}
P\colon G\BC^{\cC}\to G\UBC\ , \quad (X,U)\mapsto P_{U}(X)_{d,d,b}\ .\qedhere
\end{equation*}
\end{constr}

If we compose the functor $P$ with the fibre sequence \eqref{feojp2r2}, then we obtain a fibre sequence of functors $G\BC^{\cC}\to G\Sp\cX$ which sends $(X,U)$ to
\begin{equation}\label{efweflkwef234rergrg1}
\Yo^{s}(P_{U}(X)_{d,b})\to \Yo^s(\cO (P_{U}(X)_{d,d,b}))\to\cO^{\infty}(P_{U}(X)_{d,d,b})\stackrel{\partial}{\to}\Sigma \Yo^{s}(P_{U}(X)_{d,b}) \ .
\end{equation} 
\begin{ddd} \label{gioowegfwefwfwef} We define the fibre sequence of functors $G\BC \to G\Sp\cX$
	\[ F^{0} \to F\to F^{\infty} \xrightarrow{{\partial}} \Sigma F^{0} \]
	by left Kan extension of \eqref{efweflkwef234rergrg1}
	along the forgetful functor \eqref{vreoi34fg3vefv}. 
\end{ddd}

In order to justify this definition note that a colimit of a diagram of fibre sequences in a stable $\infty$-category is again a fibre sequence. Since a fibre sequence of functors can be detected objectwise,  it is a consequence of the pointwise formula for  the Kan extension that a Kan extension of a fibre sequence of functors with values in a stable $\infty$-category is again a fibre sequence.

If $S$ is a  {$G$-set}, then we 
  have a twist functor 
 \begin{equation}\label{gioreggwefewffewfwfwfwftw}
  T_{S}\colon G\BC\to G\BC\ ,  \quad X\mapsto S_{min,min}\otimes X\ .\end{equation} 
 By \cite[Lem.~4.17]{equicoarse} the twist functor  extends to {a functor
\[ T_{S}^{Mot} \colon G\Sp\cX\to G\Sp\cX \]
 on motives} such that
 \begin{equation}\label{rgeieo4343frfwfefweftw}\xymatrix{G\BC\ar[r]^{T_{S}}\ar[d]^{\Yo^{s}}&G\BC\ar[d]^{\Yo^{s}}\\G\Sp\cX\ar[r]^{T_{S}^{Mot}}&G\Sp\cX}\end{equation}
commutes. Note that $T_{S}^{ Mot}\simeq \Yo^{s}(S_{min,min})\otimes -$, and this functor is equivalent to the left Kan-extension of $\Yo^{s} \circ T_{S} $ along $\Yo^{s} $, so in particular it commutes with colimits.

We can extend the twist  functor  to a functor
\[ T_{S}^{\cC} \colon G\BC^{\cC}\to G\BC^{\cC}\ , \quad  (X,U)\mapsto ( S_{min,min}\otimes X, \diag(S)\times U)\ .\]
Then we have a commuting diagram
\[\xymatrix{
G\BC^{\cC}\ar[r]^{T_{S}^{ \cC}}\ar[d]^{\eqref{vreoi34fg3vefv} }&G\BC^{\cC}\ar[d]^{\eqref{vreoi34fg3vefv} }\\G\BC\ar[r]^{T_{S} }&G\BC
}\]
The  twist   functor \eqref{gioreggwefewffewfwfwfwftw} further extends to a twist functor
\[ T_{S}^{ \cU } \colon G\UBC\to G\UBC\ , \quad X\mapsto S_{disc,min,min}\otimes X \]
for  uniform bornological coarse spaces.

\begin{lem}\label{gioergregregregtw}
We have a natural isomorphism of functors
\[ T_{S}^{ \cU}\circ P \xrightarrow{\cong} P  \circ T_{S}^{ \cC} \colon G\BC^{\cC}\to G\UBC\ .\]
\end{lem}
\begin{proof}
 For  $(X,U)$  in $G\BC^{\cC}$  we construct an isomorphism of $G$-simplicial complexes
\begin{equation}\label{geroih34iog34g3g3g}
S\times P_{U}(X)\xrightarrow{\cong}  P_{ \diag(S)\times U}( S_{min,min}\otimes X)\ 
\end{equation}
 which induces the desired isomorphism of $G$-uniform bornological coarse spaces.
Let $(s,\mu)$ be a point in $S\times P_{U}(X)$.
Then there is some $n$ in $\nat$, a collection of points $x_{0},\dots,x_{n}$ in $X$ and numbers $\lambda_{i}\in [0,1]$
such that $(x_{i},x_{j})\in U$ for all pairs $i,j$, $\sum_{i=0}^{n}\lambda_{i}=1$ and
\[ \mu=\sum_{i=0}^{n} \lambda_{i}\delta_{x_{i}}\ .\]
The map~\eqref{geroih34iog34g3g3g} sends the point $(s, \mu)$ to the point
$\sum_{i=0}^{n} \lambda_{i} \delta_{(s,x_{i})}$ in $P_{ \diag(S)\times U}( S_{min,min}\otimes X)$.
In order to see that this map is invertible, note that if  $\nu=\sum_{i=0}^{n^{\prime}}\lambda^{\prime}_{i} \delta_{(s_{i},x^{\prime}_{i})}$ is a point in
$P_{ \diag(S)\times U}( S_{min,min}\otimes X)$, then  $s_{i}=s_{0}$ for all $i=1,\dots,n^{\prime}$ and 
 $( x^{\prime}_{i}, x^{\prime}_{j})\in U$ for all $i,j$. 
Therefore, the inverse of the isomorphism \eqref{geroih34iog34g3g3g} sends $\nu$ to the point $(s_{0},\sum_{i=0}^{n^{\prime}} \lambda_{i}^{\prime}\delta_{ x^{\prime}_{i}})$.

It is straightforward to check that the isomorphism is $G$-equivariant, natural in $(X,U)$ and compatible with the bornologies.
\end{proof}

\begin{lem}\label{ergierjgioerewfefwefwe111tw}
We have a commuting diagram of   functors $G\UBC\to G\BC$
\begin{equation*}\label{ergierjgioerewfefwefwe111tweq}\xymatrix{T_{S}\circ \cF \ar[d]^{\cong} \ar[r]  &T_{S}\circ \cO \ar[d]^{\cong} \ar[r]&T_{S}\circ\cO^{\infty}_{geom}\ar[d]^{\cong}\ar[r] &T_{S}\circ( \R\otimes  \cF )  \ar[d]^{\cong} \\\cF \circ T_{S}^{\cU}\ar[r]&\cO\circ T_{S}^{\cU}\ar[r]&\cO^{\infty}_{geom}\circ T_{S}^{\cU}\ar[r]& \R\otimes (\cF\circ T_{S}^{\cU})}\end{equation*}
\end{lem}
\begin{proof}
We first discuss the isomorphism in the case of $\cO^{\infty}_{geom}$.
 {For $X$ in $G\UBC$ the desired  isomorphism
\[ S_{min,min} \otimes (\R \otimes X)_h  \xrightarrow{\cong} (\R \otimes S_{disc,min,min} \otimes X)_h \]
is induced by the natural bijection of $G$-sets
\begin{equation*}\label{ergierjgioerewfefwefwetw}f \colon S\times (\R\times X)\xrightarrow{\cong} \R\times (S\times X)\ , \quad (s,(r,x))\mapsto (r,(s,x))\ .\end{equation*}
We need to verify that the coarse structures agree.}

 {For an admissible function $\psi \colon \nat \to \cP((\R \times X)^2)^G$, define
\[ \psi_S \colon \nat \to \cP((\R \times S \times X)^2)^G \]
as the function sending $n$ to the image of $\diag(S) \times \psi(n)$ under the identification induced by $f$.
Then we have
\[ \diag(S) \times (U \cap U_\psi) = (\diag(S) \times U) \cap U_{\psi_S} \]
for all admissible functions $\psi$, so the bijection $f$ induces a controlled map.}

{Conversely, let $p \colon \R \times S \times X \to \R \times X$ be the projection map.
If $\phi \colon \nat \to \cP((\R \times S \times X)^2)^G$ is an admissible function,
then the function $\phi' \colon \nat \to \cP((\R \times X)^2)^G$ sending $n$ to $(p \times p)(\phi(n))$ is also admissible.
Moreover, we have
\[ \diag(S) \times (U \cap U_{\phi'}) = (\diag(S) \times U) \cap (f \times f)^{-1}(U_{\phi}) \]
for every admissible function $\phi$ and coarse entourage $U$ of $\R \otimes X$.
Hence, the generating entourages of $S_{min,min} \otimes (\R \otimes X)_h$ and $(\R \otimes S_{disc,min,min} \otimes X)_h$ agree under the identification induced by $f$.}

 {The other isomorphisms are induced by the same bijection of underlying $G$-sets, restricted to $[0,\infty)$ for the case $\cO$ and to $\{0\}$ for the case $\cF$.
Then the diagram commutes.}
\end{proof}

\begin{lem}\label{gregee4ergtw}
We have a commuting diagram of   functors $G\UBC\to G\Sp\cX$
\begin{equation}\label{ergierjgioerewfefwefwe333tw1111}\xymatrix{ T_{S}^{ Mot }\circ F^{0} \ar[d]^{\simeq} \ar[r]  & T_{S}^{ Mot }\circ F \ar[d]^{\simeq} \ar[r]& T_{S}^{ Mot }\circ F^{\infty} \ar[d]^{\simeq}\ar[r] & T_{S}^{ Mot }\circ \Sigma F^{0}  \ar[d]^{\simeq} \\ F^{0} \circ T_{S} \ar[r]&F \circ T_{S} \ar[r]&F^{\infty} \circ T_{S} \ar[r]& \Sigma F^{0} \circ   T_{S}   }\end{equation}
\end{lem}
\begin{proof}
In a first step we postcompose  {the diagram from \cref{ergierjgioerewfefwefwe111tw}} with $\Yo^{s} $ and precompose it with the functor
$P \colon G\BC^{\cC}\to G\UBC$.  Then we get a corresponding diagram of functors $G\BC^{\cC}\to G\Sp\cX$.
We  apply the {left} Kan extension along the forgetful functor $G\BC^{\cC}\to G\BC$ and get the commuting diagram
\begin{align}\label{ergierjgioerewfefwefwe333tw11}
\mathclap{\xymatrix{
LK(T_{S}^{ Mot }  \Yo^{s}\cF P) \ar[d]^{\simeq} \ar[r]  & LK(T_{S}^{ Mot }\Yo^{s} \cO P) \ar[d]^{\simeq} \ar[r]& LK(T_{S}^{ Mot }  \cO^{\infty}P) \ar[d]^{\simeq}\ar[r] &LK( T_{S}^{ Mot } \Sigma  \Yo^{s} \cF P)  \ar[d]^{\simeq} \\
LK(\Yo^{s}\cF  T^{\cU}_{S}  P)\ar[r]&LK(\Yo^{s}\cO  T^{\cU}_{S}P) \ar[r]&LK(\cO^{\infty}  T^{\cU}_{S} P)\ar[r]& \Sigma LK( \Yo^{s}\cF   T^{\cU}_{S}P)
}}\notag\\
\end{align}
Using \eqref{rgeieo4343frfwfefweftw} and the fact that $T_{S}^{Mot}$ preserves colimits,
the upper line of \eqref{ergierjgioerewfefwefwe333tw11} is equivalent to the upper line of the diagram
\eqref{ergierjgioerewfefwefwe333tw1111}.
It remains to identify the lower line. 

 {We use \cref{gioergregregregtw} to identify the lower line of \eqref{ergierjgioerewfefwefwe333tw11} with
\begin{equation}\label{heouirhioefghohfgor} LK(\Yo^{s}\cF P  T^{\cC}_{S}) \to LK(\Yo^{s}\cO  P T^{\cC}_{S}) \to LK(\cO^{\infty} P T^{\cC}_{S}) \to \Sigma LK( \Yo^{s}\cF P T^{\cC}_{S})\ .\end{equation}
Let $E$ denote any one  of the functors $\Yo^s \circ \cF$, $\Yo^s \circ \cO$ or $\cO^\infty$.
Because the restrictions of $LK(EPT^{\cC}_{S})$ and $LK(EP)T^{\cC}_{S}$ to $G\BC^{\cC}$ are equivalent,
the universal property of the left Kan extension provides a transformation from \eqref{heouirhioefghohfgor} to
\[ LK(\Yo^{s}\cF P) T^{\cC}_{S} \to LK(\Yo^{s}\cO  P) T^{\cC}_{S} \to LK(\cO^{\infty} P) T^{\cC}_{S} \to \Sigma LK( \Yo^{s}\cF P) T^{\cC}_{S}\ .\]}
We show that this {transformation} is an equivalence. To this end we use the pointwise formula from \cref{rgoipewerfwefewfew}. We therefore must show that the natural morphism
\[ \colim_{U\in \cC^{G}(X)} E(P_{ \diag(S)\times  U}( S_{min,min}\otimes X))\to
 \colim_{V\in \cC^{G}( S_{min,min}\otimes X)} E(P_{V}( S_{min,min}\otimes X))\]
 is an equivalence. This is clear since
 $U\mapsto \diag(S)\times U$ is an isomorphism of posets from  $\cC^{G}_{\diag}(X)$ to $\cC_{\diag}^{G}(S_{min,min}\otimes X)$.
We therefore get the desired identification of the lower line of  the diagram \eqref{ergierjgioerewfefwefwe333tw11}  with the lower line 
in \eqref{ergierjgioerewfefwefwe333tw1111}.
  \end{proof}

If  $H$ is a subgroup of $G$, then we have an induction functor \begin{equation}\label{kljoeirjfoirejfreffwfe}
\Ind_{H}^{G}\colon H\Set\to G\Set\ , \quad X\mapsto G\times_{H}X\ .
\end{equation}
 The elements of $G\times_{H}X$ will be written in the form $[g,x]$ for $g$ in $G$ and $x$ in $X$, and we have the equality
$[gh,h^{-1}x]=[g,x]$ for all $h$ in $H$. 
We have a natural projection \begin{equation}\label{frefpojfp2f23f23ff}
G\times X\to \Ind_{H}^{G}(X)= G\times_{H}X\ , \quad (g,x)\mapsto [g,x]\ .
\end{equation}
This induction functor refines to an  induction functor \begin{equation}\label{gioreggwefewffewfwfwfwf}
  \Ind_{H}^{G}\colon H\BC\to G\BC \end{equation} for bornological coarse spaces.
If $X$ is some $H$-bornological coarse space, then $\Ind_{H}^{G}(X)$ becomes a $G$-bornological coarse space with the following structures:
\begin{enumerate}
\item 
 The bornological structure on $\Ind_{H}^{G}(X)$    is generated by the images under \eqref{frefpojfp2f23f23ff} of the subsets
$\{g\}\times B $ of $G\times X$ for all $g$ in $G$ and bounded subsets $B$  of $X$.
\item The coarse structure is generated by  {the entourages $\Ind_H^G(U)$, which are the images of the entourages 
$\diag(G)\times U$ of $G\times X$ under the projection \eqref{frefpojfp2f23f23ff}, for all coarse entourages $U$ of $X$.}
\end{enumerate}

The induction functor extends to motives  \begin{equation}\label{3ljg34opg34f34f3}
\Ind_{H}^{G,Mot}\colon H\Sp\cX\to G\Sp\cX
\end{equation}
 such that \begin{equation}\label{rgeieo4343frfwfefwef}\xymatrix{H\BC\ar[r]^{\Ind_{H}^{G}}\ar[d]^{\Yo^{s}_{H}}&G\BC\ar[d]^{\Yo^{s}_{G}}\\H\Sp\cX\ar[r]^{\Ind_{H}^{G,Mot}}&G\Sp\cX}\end{equation}
commutes{; see} \cite[Sec.~6.5]{equicoarse}.

\begin{lem}\label{lem:ind-mot}
 The functor $\Yo^{s}_{G}\circ \Ind_{H}^{G} \colon H\BC\to G\Sp\cX$ is an $H$-equivariant coarse homology theory.
\end{lem}

\begin{proof}
By \eqref{rgeieo4343frfwfefwef},   we have an equivalence  
$\Yo^{s}_{G}\circ \Ind_{H}^{G}\simeq\Ind_{H}^{G,Mot} \circ \Yo^{s}_{H} $.  
In view of \cref{prop:motives-universal-prop}, it suffices to show that  $\Ind_{H}^{G,Mot}$
preserves colimits.  This is the case since 
   $\Ind_{H}^{G,Mot}$ is a left adjoint functor (see \cite[Sec.~6.5]{equicoarse}). \end{proof}

We can extend the induction functor  to a functor
\[\Ind_{H}^{G,\cC}\colon H\BC^{\cC}\to G\BC^{\cC}\ , \quad  (X,U)\mapsto (\Ind_{H}^{G}(X),\Ind_{H}^{G}(U))\ .\]
Then we have a commuting diagram
\[\xymatrix{
H\BC^{\cC}\ar[r]^-{\Ind_{H}^{G,\cC}}\ar[d]^{\eqref{vreoi34fg3vefv}}&G\BC^{\cC}\ar[d]^{\eqref{vreoi34fg3vefv}}\\
H\BC\ar[r]^-{\Ind_{H}^{G}}&G\BC
}\]

The induction  functor \eqref{gioreggwefewffewfwfwfwf} further extends to an induction functor 
\begin{equation}\label{vrevoijoi3jfoi3f34f3}
\Ind_{H}^{G,\cU }\colon H\UBC\to G\UBC
\end{equation}  for  uniform bornological coarse spaces.
If $X$ is an $H$-uniform bornological coarse space, then the uniform structure on $\Ind_{H}^{G,\cU}(X)$ is generated by  the images of the entourages 
$\diag(G)\times U$ of $G\times X$ for all uniform entourages $U$ of $X$ under the projection \eqref{frefpojfp2f23f23ff}.

In the following lemma  $P_{G}$ and $P_{H}$ are the versions of the functor  {$P$ from \cref{def:rips}} for the groups $G$ and $H$, respectively.
\begin{lem}\label{gioergregregreg}
We have a natural isomorphism of functors
\[ \Ind_{H}^{G,\cU}\circ P_{H}\xrightarrow{\cong} P_{G} \circ \Ind_{H}^{G,\cC} \colon H\BC^{\cC}\to G\UBC\ .\]
\end{lem}
\begin{proof}
 For  $(X,U)$  in $H\BC^{\cC}$  we construct an isomorphism of $G$-simplicial complexes
 \begin{equation}\label{rgkjrhni3u4f34f3f3f}
G\times_{H} P_{U}(X)\xrightarrow{\cong}  P_{\Ind_{H}^{G}(U)}(\Ind_{H}^{G}(X))
\end{equation} 
  {which} induces the desired isomorphism of $G$-uniform bornological coarse spaces.
Let   $[g,\mu]$ be a point in $ G\times_{H}P_{U}(X )$.
 Then there are some $n$ in $\nat$, a collection of points $x_{0},\dots,x_{n}$ in $X$, and numbers $\lambda_{0}\dots,\lambda_{n}$ in  $[0,1]$
such that $(x_{i},x_{j})\in U$ for all pairs $i,j$, $\sum_{i=0}^{n}\lambda_{i}=1$, and
\[ \mu=\sum_{i=0}^{n} \lambda_{i}\delta_{x_{i}}\ .\]
The isomorphism sends the point $[g,\mu]$ to the point
$\sum_{i=0}^{n} \lambda_{i} \delta_{[g,x_{i}]}$ in $ P_{\Ind_{H}^{G}(U)}(\Ind_{H}^{G}(X))$.
In order to see that this map is invertible, note that if  $\nu=\sum_{i=0}^{n^{\prime}}\lambda^{\prime}_{i} \delta_{[g_{i},x^{\prime}_{i}]}$ is a point in
$  P_{\Ind_{H}^{G}(U)}(\Ind_{H}^{G}(X)) $, then in view of the definition of $\Ind_{H}^{G}(U)$  there exist  elements $h_i$ in $H$ for $i=0,\dots,n$ such that $g_{i}h_{i}^{-1}=g_{0}$.
Consequently, $\nu=\sum_{i=0}^{n^{\prime}}\lambda^{\prime}_{i} \delta_{[g,h_{i}x^{\prime}_{i}]}$, and we have  $(h_{i}x^{\prime}_{i},h_{j}x^{\prime}_{j})\in U$ for all $i,j$. 
Therefore, the inverse of the isomorphism sends $\nu$ to the point $[g,\sum_{i=0}^{n^{\prime}} \lambda_{i}^{\prime}\delta_{h_{i}x^{\prime}_{i}}]$.

It is straightforward to check that the isomorphism is $G$-equivariant, natural in $(X,U)$ and compatible with the bornologies.  {From the explicit description of the coarse and uniform structure on the induction, it follows that $G\times_{H} P_{U}(X)_{d,d,b}\cong (G\times_{H} P_{U}(X))_{d,d,b}$ and hence $\Ind_{H}^{G,\cU}\circ P_{H}\cong P_{G} \circ \Ind_{H}^{G,\cC}$ as claimed.}
\end{proof}

 In the following statement we again added subscripts $G$ or $H$ in order to indicate on which categories the respective versions of the functors $\cF$, $\cO$ and $\cO^{\infty}_{geom}$ act.
\begin{lem}\label{ergierjgioerewfefwefwe111}
We have a commuting diagram of   functors $H\UBC\to G\BC$
\begin{equation*}\label{ergierjgioerewfefwefwe111eq}\xymatrix{ \Ind_{H}^{G }\circ \cF_{H}\ar[d]^{\cong} \ar[r]  & \Ind_{H}^{G }\circ \cO_{H}\ar[d]^{\cong} \ar[r]& \Ind_{H}^{G }\circ\cO^{\infty}_{H,geom}\ar[d]^{\cong}\ar[r] & \Ind_{H}^{G }\circ( \R\otimes  \cF_{H})  \ar[d]^{\cong} \\\cF_{G}\circ \Ind_{H}^{G,\cU}\ar[r]&\cO_{G}\circ \Ind_{H}^{G,\cU}\ar[r]&\cO^{\infty}_{G,geom}\circ \Ind_{H}^{G,\cU}\ar[r]& \R\otimes (\cF_{G}\circ \Ind_{H}^{G,\cU})}\end{equation*}
\end{lem}
\begin{proof}
We {first} discuss the isomorphism in the case of the functor $\cO^{\infty}_{geom}$.
For $X$ in $H\UBC$ the isomorphism is induced by the natural bijection of $G$-sets
\begin{equation*}\label{ergierjgioerewfefwefwe}f \colon G\times_{H}(\R\times X)\xrightarrow{\cong} \R\times (G\times_{H}X)\ , \quad [g,(r,x)]\mapsto (r,[g,x])\ ,\end{equation*}
 {which is obviously an isomorphism of $G$-bornological spaces.
We need to show that the hybrid coarse structures agree under $f$.}

 {Let $p \colon G \times \R \times X \to G\times_{H}(\R\times X)$ and $q \colon \R \times G \times X \to \R\times (G\times_{H}X)$ denote the projection maps.
For an admissible function $\psi \colon \nat \to \cP((\R \times X)^2)^G$, define
\[ \psi_G \colon \nat \to \cP((\R \times (G \times_H X))^2)^G \]
as the function sending $n$ to the image of $(p \times p)(\diag(G) \times \psi(n))$ under the identification induced by $f$.
Then we have
\[ (p \times p)(\diag(G) \times (U \cap U_\psi)) = (p \times p)(\diag(G) \times U) \cap U_{\psi_G} \]
for all admissible functions $\psi$, so the bijection $f$ induces a controlled map.}

 {Conversely, if $\phi \colon \cP((\R \times (G \times_H X))^2)^G$ is an admissible function,
then the function $\phi' \colon \nat \to \cP((\R \times X)^2)^G$ sending $n$ to $(q \times q)^{-1}(\phi(n)) \cap (\R \times \{1\} \times X)^2$ is also admissible.
Moreover, we have
\[ (p \times p)(\diag(G) \times (U \cap U_{\phi'})) = (p \times p)(\diag(G) \times U) \cap (f \times f)^{-1}(U_{\phi}) \]
for every admissible function $\phi$ and coarse entourage $U$ of $\R \otimes X$.

Hence, the generating entourages of $\Ind_H^G(\R \otimes X)_h$ and $(\R \otimes \Ind_H^{G,\cU}(X))_h$ agree under the identification induced by $f$.}

The other isomorphisms are induced by the same bijection of underlying $G$-sets, restricted to $[0,\infty)$ for the case $\cO$ and to $\{0\}$ for the case $\cF$.
Then the diagram commutes.
\end{proof}

\begin{lem}\label{gregee4erg}
We have a commuting diagram of   functors $H\UBC\to G\Sp\cX$
\begin{equation*}\label{ergierjgioerewfefwefwe333}\xymatrix{ \Ind_{H}^{G,Mot }\circ F^{0}_{H}\ar[d]^{\simeq} \ar[r]  & \Ind_{H}^{G,Mot }\circ F_{H}\ar[d]^{\simeq} \ar[r]& \Ind_{H}^{G,Mot }\circ F^{\infty}_{H}\ar[d]^{\simeq}\ar[r] & \Ind_{H}^{G,Mot }\circ \Sigma F^{0}_{H}  \ar[d]^{\simeq} \\ F^{0}_{G}\circ \Ind_{H}^{G}\ar[r]&F_{G}\circ \Ind_{H}^{G}\ar[r]&F^{\infty}_{G }\circ \Ind_{H}^{G }\ar[r]& \Sigma F^{0}_{G}\circ   \Ind_{H}^{G}  }\end{equation*}
\end{lem}

\begin{proof}
The proof is, mutatis mutandis, identical to the proof of the \cref{gregee4ergtw}.
More precisely, one replaces $T_S$ by $\Ind_H^G$, starts with the diagram from \cref{ergierjgioerewfefwefwe111} instead of the one from \cref{ergierjgioerewfefwefwe111tw},
and uses \cref{gioergregregreg} instead of \cref{gioergregregregtw}.
  \end{proof}

\section{A descent result}
\label{sec:descent}
The main result of the present section is \cref{thm:desc-orbits}. Morally it is a descent result stating that  a certain natural transformation from fixed points to homotopy fixed points is an equivalence. The proof  is based on the interplay between the covariant and  contravariant functoriality of coarse homology theories encoded in their extensions to the  $\infty$-category   $G\BC_{tr}$ of $G$-bornological coarse spaces with transfers. This  $\infty$-category was introduced in \cite{coarsetrans}. It  extends the category $G\BC$, which only  {captures} the covariant behaviour of coarse homology theories.

We start by briefly  recalling the   construction of the category $G\BC_{tr}$. Let $X$ be a  $G$-bornological coarse space.
Then we let $\cC(X)$ and $\cB(X)$ denote the coarse and bornological structures of $X$. 
 For a subset $B$ of $X$ we let $[B]$ denote the coarse closure of $B$, i.e., the closure of $B$ with respect to the equivalence relation $R_{\cC(X)}${; see} \eqref{vfwrekjnkewjfwefewfwef}.

 Let now $X$ and $Y$ be $G$-bornological coarse spaces and $f \colon X\to Y$ be an equivariant map between the underlying $G$-sets.
\begin{ddd}[{\cite[{Def.~2.14}]{coarsetrans}}]\label{rgeeiorjgergergreg}
The map $f$ is called a \emph{{bounded covering}} if:
\begin{enumerate}
\item $f$ is a morphism between the underlying $G$-coarse spaces;
\item the coarse structure $\cC(X)$ is generated by the sets
$(f\times f)^{-1}(U)\cap U_{\pi_{0}}$, where $U$ is in $\cC(Y)$ and \begin{equation}\label{v4kjhuiuhedcwcwcwc}
U_{\pi_{0}}:=\bigcup_{W\in \pi_{0}(X)}W\times W ;
\end{equation}
\item\label{wtiojoegergergweg2} for every $W$ in $\pi_{0}(X)$ the restriction $f_{|W} \colon W\to f(W)$ is an isomorphism  {of coarse spaces} between coarse components;  
\item $f$ is bornological, i.e., for every $B$ in $\cB(X)$ we have $f(B)\in \cB(Y)$;
\item \label{wtiojoegergergweg} {for every $B$ in $\cB(X)$ there exists a finite bound (which may depend on $B$) on the cardinality of the sets
 \[ \{  W\in \pi_{0}(X) \mid  \pi_{0}(f)(W)=V\ , W\cap B\not=\emptyset\} \]
 (the coarse components of $X$ over $V$ which intersect $B$ non-trivially) for all $V$ in $\pi_{0}(Y)$.}
\qedhere \end{enumerate}
\end{ddd}

Note that a bounded covering is not a morphism of bornological coarse spaces in general, since it may not be proper. 
The composition of two bounded coverings is again a bounded covering; see \cite[Lem.~2.18]{trans}.

\begin{rem}\label{qeroigjoqergfewfdeeqwd}
 {Conditions \ref{wtiojoegergergweg2} and \ref{wtiojoegergergweg} in \cref{rgeeiorjgergergreg}  together are  equivalent to the following single condition:
for every $B$ in $\cB(X)$ there exists a finite coarsely disjoint partition $(B_{\alpha})_{\alpha\in A}$ of $B$, i.e.\ a finite partition $(B_{\alpha})_{\alpha\in A}$ of $B$ such that $[B_{\alpha}]\cap [B_{\alpha'}]=\emptyset$ for all $\alpha\neq\alpha'$, such that
$f_{[B_{\alpha}]} \colon [B_{\alpha}]\to [f(B_{\alpha})]$ is an isomorphism of the underlying  coarse spaces.}

 {Our phrasing of \cref{rgeeiorjgergergreg} separates the assumptions on the coarse structures from the conditions on the bornologies.}
\end{rem}

\begin{ex}\label{gbioegeerrgdfbv}
Let $h \colon S\to T$ be a map between $G$-sets and $X$ be a $G$-bornological coarse space. Then the map
\[h\times \id_{X} \colon S_{min,min}\otimes X\to T_{min,min}\otimes X\] is a bounded covering{; see} \cite[{Ex.~2.16}]{coarsetrans}.

Let $X^{\cC}$ be a $G$-coarse space with two compatible $G$-bornological structures $\cB$ and $\cB^{\prime}$ such that $\cB^{\prime}\subseteq \cB$. We let $X$ and $X^{\prime}$ denote the corresponding $G$-bornological coarse spaces. Then the identity map of  the underlying sets is a bounded covering $X^{\prime}\to X${; see} \cite[{Ex.~2.17}]{coarsetrans}.
If $\cB^{\prime}\not=\cB$, then it is not a morphism of $G$-bornological coarse spaces.
\end{ex}

\begin{constr}\label{gfioweffwefwefwefwef}
 {We recall the $\infty$-category $G\BC_{tr}$ from \cite[Def.~2.29]{coarsetrans}.
Let $\Tw \colon \Delta \to \Cat$ denote the cosimplicial category which sends $[n]$ to $\Tw([n]) = [n]^{op} \star [n])$, the twisted arrow category of $[n]$ (as a simplicial set, this is the edgewise subdivision).
We denote by $\tilde{G\BC}$ the category whose objects are $G$-bornological coarse spaces and whose morphisms are morphisms of the underlying $G$-coarse spaces.}

 {Then $G\BC_{tr}$ is a certain sub-simplicial set of the simplicial set
\[ \Fun(\Tw,\tilde{G\BC}) \colon \Delta^{op} \to \Set,\quad [n] \mapsto \Fun(\Tw([n]),\tilde{G\BC})\ .\]
Since it turns out that $G\BC_{tr}$ is $2$-coskeletal \cite[Lem.~2.30]{coarsetrans}, we content ourselves with describing $2$-simplices. They are given by diagrams of the form
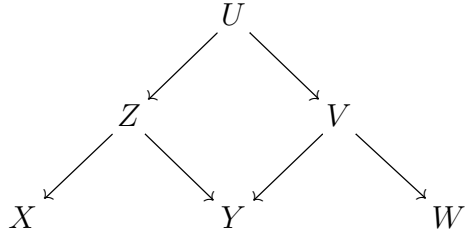
\[\xymatrix{&&U\ar[dr] \ar[dl] &&\\&Z\ar[dr] \ar[dl] &&V\ar[dr] \ar[dl] \\X&&Y&&W}\]
such that all morphisms going left are bounded coverings, all morphisms going right are proper and bornological and such that the square in the middle is a pullback on the level of the underlying $G$-coarse spaces. This $\infty$-category is an effective Burnside category in the sense of Barwick \cite[Def.~3.6]{Barwick:2014aa}; see \cite[Def.~4.40 \& Rem.4.41]{unik}.}
\end{constr}

We have a functor
 \begin{equation}\label{g354oijot343g3}
m\colon G\Set^{op} \times G\BC \to G\BC_{tr}
\end{equation}
which admits the following description. Consider the functor
\[ m' \colon G\Set \times \tilde{G\BC} \to \tilde{G\BC},\quad (S,X) \mapsto S_{min,min} \otimes X\ .\]
We have a cosimplicial $\infty$-category $\nu \colon \Delta \to \Cat_\infty$ which sends $[n]$ to the nerve of $[n]$. Then $\nu$ corepresents the identity functor on $\Cat_\infty$, while $(-)^{op} \circ \nu$ corepresents the functor $(-)^{op} \colon \Cat_\infty \to \Cat_\infty$.
Moreover, we have a transformation of cosimplicial $\infty$-categories $\pi \colon \Tw \to ((-)^{op} \circ \nu) \times \nu$. From this, we obtain the functor
\begin{align*}
 G\Set^{op} \times \tilde{G\BC} &\xrightarrow{\times} \Fun(((-)^{op} \circ \nu) \times \nu,G\Set \times \tilde{G\BC}) \\
 &\xrightarrow{\pi^*} \Fun(\Tw, \tilde{G\BC})\ .
\end{align*}
In fact, this functors restricts to a functor
\[ \tilde{m} \colon G\Set^{op} \times \tilde{G\BC} \to A^{eff}(\tilde{G\BC})\ ,\]
where the target is the effective Burnside category of $\tilde{G\BC}${; here we use that the effective Burnside category $A^{eff}$ is defined for every category with pullbacks \cite[Def.~3.6]{Barwick:2014aa}.}
We compose $\tilde{m}$ with the endofunctor $P$ of $\Fun(\Tw, \tilde{G\BC})$ which takes each simplex to the simplex represented by the same diagram of $G$-coarse spaces, but where we replace the bornologies on all entries which are the domain of a map by that bornology which turns the morphism going right into a bornological morphism. For example, in a diagram as in \cref{gfioweffwefwefwefwef}, we equip $Z$ with the bornology pulled back from $Y$ and we equip $U$ and $V$ with the bornologies pulled back from $W$.

Using \cref{gbioegeerrgdfbv}, one now checks that the composition $m := P \circ \tilde{m}$ defines a functor $G\Set^{op} \times G\BC \to G\BC_{tr}$.
The restriction of $m$ to the object $pt$ of $G\Set^{op}$ induces a functor
 \begin{equation}\label{verv3r3oijfoij3f3f}
\iota \colon G\BC\to G\BC_{tr}\ ,
\end{equation}
cf.~\cite[{Def.~2.33}]{coarsetrans}.

Let $\bC$ be a cocomplete stable $\infty$-category and let $E \colon G\BC_{tr} \to \bC$ be a functor.

\begin{ddd}[{\cite[{Def.~2.53}]{coarsetrans}}]\label{rgiorhgiuregergergergergerg}
$E$ is called a $\bC$-valued {\em equivariant coarse homology theory with transfers} if $E\circ \iota\colon G\BC\to \bC$ is a  $\bC$-valued  equivariant coarse homology theory (in the sense of  {\cref{def:coarsehom}}).
\end{ddd}

Let $E\colon G\BC_{tr}\to\bC$ be a functor.
\begin{ddd}\label{jfi3rhfiuofewfewf} We define the functor
	\[\underline{E}:=E\circ m \colon G\Set^{op}\times  G\BC\to \bC\ .\qedhere\]
\end{ddd}

Assume now that  $E$ is a coarse homology theory with transfers. For every $G$-set $T$, we have an equivalence   \[\underline{E}(T,-)\simeq (E\circ \iota)_{T_{min,min}}(-)\]  of functors $G\BC\to \bC${; see} \cref{def:twist} for notation. 
The right-hand side is a twist of an equivariant coarse homology theory and therefore again an equivariant coarse homology theory by \cref{lem:twist}.
 {By \cref{prop:motives-universal-prop}, we can extend $\underline E$ along $\Yo^{s}$ to a functor (denoted by the same symbol for simplicity)}
\[\underline{E} \colon G\Set^{op}\times G\Sp\cX\to \bC \]
which preserves colimits in its second argument.

From now on until the end of this section we assume that the $\infty$-category $\bC$ is stable, cocomplete and complete, and that $E$ is a $\bC$-valued equivariant coarse homology theory with transfers.
\begin{ddd}\label{fwpeou2903rf}
We define the functor
\[\wt E \colon  \PSh(G\Set)^{op}\times G\Sp\cX\to \bC \]
as a right Kan extension
of 
$\underline{E}$ along the functor
\[ \yo^{op}\times \id_{G\Sp\cX} \colon G\Set^{op}\times G\Sp\cX\to \PSh(G\Set)^{op}\times G\Sp\cX\ .\qedhere \]
\end{ddd}

From now on we consider $\wt E$ as a contravariant functor in its first argument.
\begin{rem}\label{rem_234ertewert}
{Since the Yoneda embedding is fully faithful, we have a commuting diagram
\[ \xymatrix{G\Set\times G\Sp\cX\ar[d]_-{\yo\times \id_{G\Sp\cX}}\ar[r]^-{\underline{E}}&\bC\\ \PSh(G\Set)\times G\Sp\cX\ar[ur]_(.67){\wt E}&}\]
As $\PSh(G\Set)$ is the free colimit completion of $G\Set$ (\cite[Thm.~5.1.5.6]{htt}), the functor $\wt E$ is essentially uniquely characterized by an equivalence
\[\wt E\circ (\yo\times \id_{G\Sp\cX})\simeq \underline{E}\]
and the property that it sends colimits in its first argument to limits.}

{Consequently, if $A \colon I\to G\Set$ and $X \colon J\to G\Sp\cX$ are some functors from small categories $I$ and $J$, then we have a canonical equivalence
\[ \wt E(\colim_{I} \yo(A),\colim_{J}X)\simeq \lim_{I} \colim_{J}\wt E(\yo(A),X)\ .\]
Note that the order of the limit and the colimit matters  in general.}
\end{rem}

Let $A$ be in $\PSh(G\Set)$ and let $E$ be a $\bC$-valued equivariant coarse homology theory with transfers. 
\begin{lem}
	\label{lem:homtheory}
	If $A$ is compact, then the functor $\wt E(A,-)\colon G\Sp\cX\to \bC$ preserves colimits.
	\end{lem}
 \begin{proof}
We have an equivalence $\wt E(\yo(S),-)\simeq \underline{E}(S,-)$ of functors  from $G\Sp\cX$ to $\bC$. Therefore, $\wt E(\yo(S),-)$ preserves colimits for every $G$-set $S$.  Since $A$ is compact, it is a retract of a finite colimit of objects of the form $\yo(S)$ with $S$ in $G\Set$ by  \cite[Prop.~5.3.4.17]{htt}. 

	If $A$ in $\PSh(G\Set)$ is a finite colimit of representables,  then  $\wt E(A,-)$ is a finite limit  of colimit preserving functors. Since $\bC$ is stable,  finite limits in $\bC$ commute with arbitrary colimits \cite[Prop.~1.1.4.1]{HA}. 
	Hence $\wt E(A,-)$ preserves colimits.

	 {If $A$ is a retract of a finite colimit $A'$ of representables, then $\wt E(A,-)$ is a retract of $\wt E(A',-)$. Consequently, the relevant comparison maps for $\wt E(A,-)$ are retracts of the analogous comparison maps for $\wt E(A',-)$. Since the comparison maps for $\wt E(A',-)$ are equivalences and retracts of equivalences are equivalences, the lemma follows.} 
\end{proof}

Recall that  $G\Orb$  denotes the full subcategory of $G\Set$ of transitive $G$-sets{; see} \cref{gbiofgergergergrg}.

\begin{rem}\label{ewfijoewffwefewf}
	By Elmendorf's theorem the homotopy theory  of $G$-spaces is modeled by the presheaf category $\PSh(G\Orb)${; see} \cref{efweu9u9wefwef}. This category is equivalent to the category of sheaves   $\Sh(G\Set)$  with respect to the Grothendieck topology on $G\Set$ given by disjoint decompositions into invariant subsets.  We prefer to identify  the sheafification morphism $\PSh(G\Set)\to \Sh(G\Set)$   with the restriction morphism along the
	inclusion $r\colon G\Orb \to G\Set$ since in our special situation it has an additional left adjoint $r_{!}$ which is not part of general sheaf theory.\end{rem}

The inclusion  \begin{equation}\label{qwelfjqwoifqfeefqefqwef}
r\colon G \Orb \to G\Set
\end{equation} induces an adjunction
\begin{equation}\label{efkjbnkjffrefw}
r_{!} \colon  \PSh(G\Orb) \leftrightarrows  \PSh(G\Set):r^{*}
\end{equation}
by \cite[Prop.~5.2.6.3]{htt}.
Later in the proof of  \cref{lem:u-equiv} we will need a formula for the counit
\begin{equation}\label{g5k3n4kf34f34f34f}
r_{!}r^{*}\to \id \end{equation}  of the adjunction \eqref{efkjbnkjffrefw}.  To this end we consider a $G$-set  $S$ and let $S\cong \bigsqcup_{R\in G\backslash S} R$ be the decomposition of $S$ into transitive $G$-sets.

\begin{lem}
	\label{lem:annoying}
	The counit \[ r_{!}r^{*}\yo(S) \to \yo(S)\] is equivalent to  the morphism
	\begin{equation}\label{efewjflwef23r}\coprod_{R\in G\backslash S} \yo(r(R))\to \yo(S)\ ,\end{equation} induced by the family of inclusions $(r(R)\to S)_{R\in G\backslash S}$.
\end{lem}
\begin{proof}
	We start with the morphism
	\[\coprod_{R\in G\backslash S} \yo(r(R))\to  \yo(S)\]
	induced by the collection of  inclusions $(r(R)\to S)_{R\in G\backslash S}$.
	We claim that it becomes an equivalence after application of $r^{*}$. Indeed, for $T$ in $G\Orb$ we have a commuting square
	\[\xymatrix{(r^{*}\coprod_{R\in G\backslash S} \yo(r(R)))(T)\ar[r]\ar[d]^{\simeq}&(r^{*}\yo(S))(T)\ar[d]^{\simeq}\\\coprod_{R\in G\backslash S} \Map_{G\Set}(r(T),r(R)) \ar[r]^(0.6){\simeq}&\Map_{G\Set}(r(T),S)}\]
	 The lower horizontal map is an equivalence since the functor $\Map_{G\Set}(r(T),-)$ commutes with coproducts since  $r(T)$ is a transitive $G$-set.
	 
	Since the counit of an adjunction is a natural transformation, we get  the following commuting diagram
	\begin{equation*}\label{friuhiwuevwevwevw}
	\xymatrix{
	r_{!}r^{*}\coprod_{R\in G\backslash S} \yo(r(R))\ar[d]^{\operatorname{counit}}\ar[r]^-{\simeq}&r_{!}r^{*}\yo(S)\ar[d]^{\operatorname{counit}}\\
		\coprod_{R\in G\backslash S} \yo(r(R))\ar[r]^-{\eqref{efewjflwef23r}}&\yo(S)
	}
	\end{equation*} 
	 {It remains to} show that the left vertical arrow is an equivalence. To this end we consider the diagram 
	\[\xymatrix{
	r_{!}(r^{*}r_{!})\coprod_{R\in G\backslash S} \yo(R) \ar[r]^{\simeq }&(r_{!}r^{*})r_{!}\coprod_{R\in G\backslash S} \yo(R)\ar[d]^{\operatorname{counit}\circ r_{!}} \ar[r]^{\simeq }&r_{!}r^{*}\coprod_{R\in G\backslash S} \yo(r(R))\ar[d]^{\operatorname{counit}} \\
	r_{!}\coprod_{R\in G\backslash S} \yo(R) \ar@{=}[r]\ar[u]_{r_{!}(\operatorname{unit})}&r_{!}\coprod_{R\in G\backslash S} \yo(R) \ar[r]^{\simeq}&
		\coprod_{R\in G\backslash S} \yo(r(R))
		}\]
	The left square commutes by the  usual relation between the unit and the counit of an adjunction.
	 {Since $r_!$ commutes with colimits and $r_!\yo(R) \simeq \yo(r(R))$ by adjointness, the horizontal morphisms on the right are equivalences.}
	Since $r$ is fully faithful, the unit appearing at the left  is an equivalence.
	Hence, the counit on the right is an equivalence as claimed.
\end{proof}

In order to simplify the notation in the arguments below we introduce now the following abbreviation.
Let $\pt$ denote the one-point $G$-bornological coarse space.
\begin{ddd}\label{rgiuregergegergre}
	We define the functor
	\[\wt E_\pt:= \wt E(-,\Yo^s(\pt))\colon \PSh(G\Set)^{op}\to \bC \ .\qedhere\] 
\end{ddd}
We consider $\wt E_{\pt}$ as a contravariant functor from $\PSh(G\Set)$ to $\bC$ which sends colimits to limits.

The counit \eqref{g5k3n4kf34f34f34f} induces a transformation
\begin{equation}\label{rejvbnkjf4f43f}
u\colon \wt E_\pt\to \wt E_\pt\circ r_!\circ  r^*\ .
\end{equation}

\begin{rem}\label{regerger}
Recall from \cite[{Def.~2.61}]{coarsetrans} that we call a coarse homology theory with transfers \emph{strongly additive}
if its sends free unions (see \cite[Ex.~2.16]{equicoarse}) of families of $G$-bornological coarse  spaces to products. Note further  that for $S$ in $G\Set$ the $G$-bornological coarse space 
$S_{min,min}$ is the free union of the family $(R_{min,min})_{R\in G\backslash S}$.
This is used to see that the morphism \eqref{oihofefwefwef} below is an equivalence.
\end{rem}

\begin{lem}
	\label{lem:u-equiv}
	If $E$ is strongly additive, then the transformation \eqref{rejvbnkjf4f43f}
	is an equivalence.
\end{lem}
\begin{proof}
	Let $S$ be in $G\Set$.	Using \cref{lem:annoying} and the fact that $\wt E_\pt$ sends colimits to limits, the specialization $u_S$ of \eqref{rejvbnkjf4f43f} to $S$ is given by the map
	\[\wt E_\pt(S)\to \prod_{R\in G\backslash S}\wt E_{\pt}(r(R))\ .\]
	 {Recall from \cref{gbioegeerrgdfbv} that the inclusions $R_{min,min} \to S_{min,min}$ are bounded coverings. Then by} the definition of $\wt E_\pt$ this  {map} is equivalent to the map
	\[E(S_{min,min})\to \prod_{R\in G\backslash S}E(R_{min,min})\]
	obtained from the transfers along the inclusions of the orbits of $S_{min,min}$. Since $S_{min,min}$ is discrete, we have an isomorphism \[S_{min,min}\cong \coprod^{\free}_{R\in G\backslash S}R_{min,min}\] of $G$-bornological coarse spaces. 
	By strong additivity of $E$, the map
	\begin{equation}\label{oihofefwefwef}
	E(S_{min,min})\simeq E\big(\coprod^{\free}_{R\in G\backslash S}R_{min,min}\big)\to \prod_{R\in G\backslash S}  E (R_{min,min})
	\end{equation}
	is an equivalence. Therefore, $u_S$ is an equivalence.
\end{proof}

The following lemma is the crucial technical ingredient in the proof of the main result of the present section (\cref{guoiuoi34ug34g34g3}). It allows us to move $G$-sets from one argument of the functor $\wt E$ to the other.

We consider a $G$-set $S$.
\begin{lem}
	\label{lem:swap}
	There is an equivalence 
	\[s\colon \wt E(-,\Yo^{s}(S_{min,min}))\to \wt E_\pt(-\times  \yo(S)) \]	
	of contravariant functors from  $\PSh(G\Set)$ to $\bC$. \end{lem}
\begin{proof}
 	Using the canonical isomorphisms  of functors \[(-)_{min,min}\otimes S_{min,min}\cong (-\times S)_{min,min}\cong (-\times S)_{min,min}\otimes \pt\] 
	from $G\Set^{op}$ to  $G\BC_{tr}$ we obtain an equivalence of functors \[m(-,S_{min,min}) \simeq m(-\times S,\pt)\ .\]
	We compose this equivalence with $E$ and form the right-Kan extension along the functor $\yo^{op} \colon G\Set^{op}\to \PSh(G\Set)^{op}$. 
	We obtain an equivalence
	\begin{equation}\label{rgeokerpogregreg}
RK(\underline{E}(\ldots,\Yo^{s}(S_{min,min})))(-)\simeq RK(\underline{E}(\ldots\times S,\pt))(-)
\end{equation}
	of contravariant functors from $\PSh(G\Set)$ to $\bC$ which send colimits to limits. 
	Here $RK$ denotes the right-Kan extension in the variable indicated by $\ldots$, and $-$ is the argument of the resulting functor. 
	By definition of $\wt E$ we have an equivalence
	\begin{equation}\label{efwoig344343gebergerg}
RK(\underline{E}(\ldots,\Yo^{s}(S_{min,min})))(-)\simeq \wt E(-,\Yo^{s}(S_{min,min}))\ .
\end{equation} 
	For the right-hand side we note the equivalence  $\yo(\ldots\times S)\simeq \yo(\dots)\times \yo(S)$, and that the functor
	$-\times \yo(S)$ preserves colimits. This implies  a natural equivalence 
	\begin{equation}\label{efwoig344343gebergerg1}
RK(\underline{E}(\ldots\times S,\pt))(-)\simeq \wt E_{\pt}(-\times \yo(S))\end{equation} 
	since both functors send colimits to limits and coincide on representables.	 
	Inserting \eqref{efwoig344343gebergerg} and \eqref{efwoig344343gebergerg1} into \eqref{rgeokerpogregreg} we obtain the desired equivalence.
\end{proof}

We now state the main result of the present section. Recall that $\bC$ is a complete and cocomplete, stable $\infty$-category.
Furthermore, $E$ is an equivariant $\bC$-valued coarse homology theory with transfers. We let $\wt E$ be defined as in  \cref{fwpeou2903rf}.  
We consider an object  $A$  in $\PSh(G\Set)$ and  a transitive $G$-set $R$   in $G_{\cF}\Orb$. Let
\begin{equation}\label{ihiuewfwefewfw}
p_{R}\colon \wt E(*,\Yo^{s}(R_{min,min}))\to \wt E(A,\Yo^{s}(R_{min,min}))
\end{equation}
be the map induced by $A\to *$
\begin{prop}\label{guoiuoi34ug34g34g3}
	\label{thm:desc-orbits}
	Assume:
	\begin{enumerate}
		\item $E$ is strongly additive (see \cref{regerger});
		\item\label{eiwfowefefewf} $r^*A$ in $\PSh(G\Orb)$ is equivalent to $E_\cF G$.
	\end{enumerate} Then the morphism $p_R$ in \eqref{ihiuewfwefewfw} is an equivalence.
\end{prop}
\begin{proof} We consider the following commutative diagram in $\bC$:
	\[\xymatrix{
		\wt E( *,\Yo^{s}( R_{min,min}))\ar[d]^s_\simeq\ar[r]^-{p_{R}}&\wt E(A,\Yo^{s}(R_{min,min}))\ar[d]^s_\simeq\\
		\wt E_\pt(\yo(r(R)))\ar[r]\ar[d]^{u}_\simeq&\wt E_\pt(A\times \yo( r(R)))\ar[d]^u_\simeq\\
		\wt E_\pt(r_!r^*(\yo(r(R))))\ar[r]\ar[d]_{\simeq}&\wt E_\pt (r_! r^*(A\times \yo(r(R))))\ar[d]_\simeq\\
		\wt E_\pt( r_!(\yo(R)))\ar[r]&\wt E_\pt( r_!(r^*A\times\yo(R))) 	}\]
		Here $s$ is the natural equivalence from \cref{lem:swap}, and the morphism $u$ {from \eqref{rejvbnkjf4f43f}} is a natural equivalence by \cref{lem:u-equiv}. 
	
	We further use the canonical equivalence
	$r^{*}\yo ( r(R))\simeq \yo(R)$ for the lower left vertical equivalence, and in addition the fact that $ r^{*}$   preserves products for  the lower right vertical equivalence. The lower horizontal morphism is an equivalence since 	\[r^{*}A\times \yo(R)\simeq E_{\cF}G\times \yo(R)\simeq \yo(R)\ ,\]
	where the first equivalence holds true by {Assumption~\ref{eiwfowefefewf}} and the second equivalence  follows from the fact that
	  $R$ has stabilizers in $\cF$, also by assumption.
\end{proof}

\begin{rem}\label{rgrioegoregergreg}  	As  explained  in  \cref{efweu9u9wefwef}, the $\infty$-category  $\PSh(G\Orb)$ is a model for the homotopy theory of $G$-spaces.   	
	Compactness of $E_{\cF}G$ as a presheaf  on $G\Orb$ will play a crucial role in our arguments. This condition is closely related to  the {existence of} a  $G$-compact   model $E^{top}_{\cF}G$ of $E_{\cF}G$.
	
	Identifying presheaves on $G\Orb$ with sheaves on $G\Set$, we can consider $E_{\cF}G$ as an object of $\PSh(G\Set)$  which satisfies the sheaf condition. But 
	 compactness of $E_{\cF}G$ as an object of $\PSh(G\Set)
	$ is a too strong condition. 
		 For this reason we consider compact objects $A$ in $\PSh(G\Set)$ which after sheafification, i.e., after application of $r^{*}$, become equivalent to $E_{\cF}G$. The existence of such an object is an 
		 important assumption in the following. In \cref{lem:fin-dim-model}, we will show that the existence of a finite-dimensional model for $E^{top}_{\cF}G$ (a much weaker condition than $G$-compactness) implies the existence of such a compact presheaf $A$.		 
\end{rem}

A $G$-simplicial complex is a simplicial complex on which $G$ acts by morphisms of simplicial complexes.
{We denote by $G\Simpl$ the category of $G$-simplicial complexes and $G$-equivariant simplicial maps.}
Let $K$ be a $G$-simplicial complex.
\begin{ddd}\label{gu92zue98wregwew}
$K$ is \emph{$G$-finite} if it consists of finitely many $G$-orbits of simplices.
\end{ddd}

We let   $G_{\cF}\Simpl^{\fin}$ denote the full subcategory of $G\Simpl$ of  $G$-finite
$G$-simplicial complexes with stabilizers in $\cF$.

We have a canonical functor
\[ k := (-)_{d,d,d} \colon G_{\cF}\Simpl^{\fin}\to G\UBC\]
which equips a $G$-simplicial complex with the structures induced by the spherical quasi-metric. Hence we have a functor
\[\cO^\infty\circ k\colon G_{\cF}\Simpl^{\fin}\to G\Sp\cX\ ,\]
where $\cO^\infty$ denotes the cone-at-infinity functor  {from \cref{iwoijgwegewfefewfefw}}.

Let   $A$ be in $\PSh(G\Set)$. 
\begin{prop}\label{gio3jo24f22f3}
	Assume:
	\begin{enumerate} \item $E$ is strongly additive; \item $A$ is compact;   \item\label{regieorg43tr34tg34t34t} $r^*A$ is equivalent to $E_\cF G$.\end{enumerate}
	Then the natural transformation
	\[ \wt E(*,(\cO^\infty \circ k)(-))\to \wt E(A,(\cO^\infty\circ k)(-))\]
	of functors from $G_{\cF}\Simpl^{\fin}$ to $\bC$
	induced by $A\to *$ is a natural equivalence.
\end{prop}
\begin{proof}
	 For $R$ in $G_{\cF}\Orb$, the object
	$(\cO^\infty\circ k)(R)$ of $G\Sp\cX$ is equivalent to $\Sigma\Yo^{s}(R_{min,min})$ by \cite[Prop.~9.35]{equicoarse}.  
	{Since $\wt E(*,-)$ and $\wt E(A,-)$ preserve colimits in the second argument  by \cref{lem:homtheory},}
	 the map $\wt E(*,(\cO^\infty\circ k)(R)) \to \wt E(A,(\cO^\infty\circ k)(R))$ is equivalent to the map $\Sigma\wt E(*,\Yo^{s}(R_{min,min})) \to \Sigma\wt E(A,\Yo^{s}(R_{min,min}))$,
	which is an equivalence by \cref{thm:desc-orbits}.

The functor $k$ sends equivariant decompositions of $G$-finite $G$-simplicial complexes  {to equivariant uniform decompositions} of $G$-uniform bornological coarse spaces by \cite[Lem.~10.9]{equicoarse}.
The functor $\cO^{\infty}$ is  excisive for  those decompositions by \cite[Cor.~9.36 and Rem.~9.37]{equicoarse}. Furthermore, it is homotopy invariant by \cite[Cor.~9.38]{equicoarse}.

Since $\wt E(*,-)$ and $\wt E(A,-)$ preserve colimits in the second argument by \cref{lem:homtheory}, the functors $ \wt E(*, (\cO^\infty \circ k)(-))$ and $\wt E(A,(\cO^\infty \circ k)(-))$ are excisive for equivariant decompositions of $G$-finite $G$-simplicial complexes. Furthermore, they are both homotopy invariant.

 A natural transformation between two such functors which is an equivalence on $G$-orbits with stabilizers in $\cF$ is an equivalence on $G$-finite 
 $G$-simplicial complexes with stabilizers in $\cF${: by induction on the number of equivariant cells, this follows from application of the Five-Lemma to the Mayer--Vietoris sequences arising from the pushout squares describing simplex attachments}. This implies the assertion.
\end{proof}

\section{Duality of \texorpdfstring{$\boldsymbol{G}$}{G}-bornological spaces}
In this section we develop a notion of duality for $G$-bornological spaces that we will use later to compare certain assembly and forget-control maps.

The category $G\Born$  {(see Definitions~\ref{erjgoijgoiergjergegrrr} and \ref{erjgoijgoiergjergegrrr1})} of $G$-bornological spaces and proper equivariant maps has a symmetric monoidal structure $\otimes$.
If $Y$ and $X$ are $G$-bornological spaces, then $Y\otimes X$ is the $G$-bornological space with underlying $G$-set $Y\times X$ (with diagonal action) and the bornology generated by the subsets $A\times B$
for bounded subsets $A$ of $Y$ and $B$ of $X$. Note that this tensor product is not the cartesian product in $G\Born$.

Recall   that a subset $L$ of a $G$-bornological space $X$ is called locally finite if $L\cap B$ is finite for every bounded subset $B$ of $X${; see} \cref{rgieog3434gergeg}.

For a set 
$A$ we let $|A|$ in $\nat\cup \{\infty\}$ denote the number of elements of $A$. 

For a subset $L$ of $X\times G$ we consider \begin{equation}\label{ehfiohefoiwhf}L_{1}:=L\cap (X\times \{1\})\end{equation}
as a subset of $X$ in the natural way.

Let $X$ be a $G$-bornological space and $L$ be a  $G$-invariant  subset  of $X\times G$.

\begin{lem}\label{wtquzdguz1}
	$L$ is a locally finite  subset of $X\otimes G_{max}$ if  and only if 
	$\sum_{g\in G} |L_{1}\cap gB|<\infty$
	for every bounded subset $B$ of $X$.
\end{lem}
\begin{proof}
	The  subset $L$ of $X\otimes G_{max}$ is locally finite if and only if 
	$L\cap (B\times G)$ is finite  for every bounded subset $B$ of $X$.
	Since $L$ is $G$-invariant we have  bijections
	\[L\cap (B\times G)\cong \bigsqcup_{g\in G} L\cap (B\times \{g^{-1}\})\cong \bigsqcup_{g\in G} L\cap (gB\times \{1\})\cong \bigsqcup_{g\in G} L_{1}\cap gB\ .\]
	This   implies the assertion. 
\end{proof}

Let $X$ be a $G$-bornological space and $L$ be a  $G$-invariant  subset  of $X\times G$.
\begin{lem}\label{wtquzdguz2}
	$L$ is a  locally finite   subset  of $X\otimes G_{min}$ if and only if
	$L_{1}\cap gB$ is finite for every bounded subset $B$ of $X$ and {every} $g$ in $G$.
\end{lem}

\begin{proof}
	The  subset $L$ of $X\otimes G_{min}$ is locally finite if and only
	$L\cap (B\times \{g\})$ is finite for every bounded subset $B$ of $X$ and $g$ in $G$.
	Since $L$ is $G$-invariant we 
	have   bijections
	\[L\cap (B\times \{g\})\cong L\cap (g^{-1}B\times \{1\})\cong L_{1}\cap g^{-1}B\ .\]
	This implies the assertion. \end{proof}

Let $X$ and $X^{\prime}$ be   two $G$-bornological spaces with the same underlying $G$-set.  

\begin{ddd}\label{grepogergeerge}
	We say that $X$ is \emph{dual to} $X^{\prime}$ if the sets of $G$-invariant locally finite subsets of
	$X\otimes G_{max}$ and $X^{\prime}\otimes G_{min}$ coincide.
\end{ddd}

 If $X$ and $X'$ are two $G$-bornological coarse spaces, then we say that $X$ is dual to $X^{\prime}$ if the underlying $G$-bornological space of $X$ is dual to the one of~$X^{\prime}$.
 
\begin{rem}
	Note that duality is not an equivalence relation. In particular, the order is relevant.
\end{rem}

\begin{ex}
	Let $S$ be a $G$-set with finite stabilizers.
	\begin{enumerate}
		\item $S_{min}$ is dual to $S_{lax}$, where $S_{lax}$ ($lax$ stands for locally $max$) is $S$ with the bornology generated by the $G$-orbits.
		\item $S_{lin}$ is dual to $S_{max}$, where $S_{lin}$ ($lin$ stands for locally $min$) is $S$ with the bornology given by subsets which have at most finite intersections with each   $G$-orbit.\qedhere
	\end{enumerate}
\end{ex}

Let $X$ be a $G$-bornological  space.
\begin{ddd}\label{regij4o34g3g}
	$X$ is called \emph{$G$-bounded} if there exists a bounded subset $B$ of $X$ such that $GB=X$.
\end{ddd}

\begin{ddd}
\label{def:gproper}
	$X$ is called \emph{$G$-proper} if the set $\{g\in G\:|\: gB\cap B\not=\emptyset\}$ is finite for every bounded subset $B$ of $X$.
\end{ddd}

If $X$ is a $G$-bornological space, then we let $X_{max}$ denote the $G$-bornological space with the same underlying $G$-set and the maximal bornology.

Let $X$ be a $G$-bornological space and $Y$ be a bornological space (which we consider as a $G$-bornological space with the trivial $G$-action).

\begin{lem}\label{wtquzdguz3}Assume:
	\begin{enumerate}
		\item\label{rgijgioergregegreger1} $X$ is $G$-proper.
		\item\label{rgijgioergregegreger2} $X$ is $G$-bounded. \end{enumerate}
	Then $Y\otimes X$ is dual to $Y\otimes X_{max}$.
\end{lem}
\begin{proof}
	Let $L$ be a $G$-invariant subset of $Y\times X\times G$. In view of Lemma \ref{wtquzdguz1} and Lemma \ref{wtquzdguz2}  local finiteness of $L$ in $Y\otimes X\otimes G_{max}$ or $Y\otimes X_{max}\otimes G_{min}$ is characterized by conditions on the subset $L_{1}$ of $Y\times X${; see} \eqref{ehfiohefoiwhf} for notation.
	
	We must check that the following conditions on $L_{1}$ are equivalent:
	\begin{enumerate}
		\item $|(A\times X)\cap L_{1}|<\infty$ for every  bounded subset $A$ of $Y$.
		\item $\sum_{g\in G}| (A\times gB)\cap L_{1}|<\infty$ for all bounded subsets $A$ of $Y$ and bounded subsets $B$ of $X$.
	\end{enumerate}
	We assume that $L_{1}$ satisfies Condition~\ref{rgijgioergregegreger1}.  
	Let $B$ be a bounded subset of $X$ and $A$ be a bounded subset of $Y$.
	Since $X$ {is $G$-proper, the family $(gB)_{g\in G}$ has finite multiplicity, say bounded by  $m$ in $\nat$. 
	We get
	\[\sum_{g\in G} |(A\times gB)\cap L_{1}| \le  m |(A\times X)\cap L_{1}|<\infty\ .\]}
	Consequently, $L_{1}$ satisfies Condition~\ref{rgijgioergregegreger2}.
	
	We now assume that $L_{1}$ satisfies Condition~\ref{rgijgioergregegreger2}. Let $A$ be a bounded subset of $Y$.
	Since $X$ is $G$-bounded we can choose a bounded subset $B$ of $X$ such that $GB=X$.
	Then
	\[|(A\times X)\cap L_{1}|\le \sum_{g\in G}| (A\times gB)\cap L_{1}|<\infty\ .\]
	Hence $L_{1}$ satisfies Condition~\ref{rgijgioergregegreger1}.
\end{proof}

 The following lemma explains why the  notion of duality is relevant. 
  Assume that $X$ and $X^{\prime}$ are $G$-bornological coarse spaces with the same underlying $G$-coarse space. 
  Recall the notation   $\Yo^{s}_{c}$ for the universal continuous equivariant coarse homology theory,  see  \eqref{brtoij4o5g4fwefwefewfewfbbt4}.
 \begin{lem} \label{biorgregergrege111}
 If  $X$ is dual to $X^{\prime}$, then we have a canonical equivalence in $G\Sp\cX_{c}$ \[\Yo^{s}_{c}(X\otimes G_{can,max})\simeq \Yo^{s}_{c}(X^{\prime}\otimes G_{can,min})\ .\]
 \end{lem}
 \begin{proof}
This lemma is a special case of the following Lemma \ref{biorgregergrege1111} for the case $I=*$.
\end{proof}

We will need  a functorial variant of \cref{biorgregergrege111}. 
We  consider  a small    category  $I$ and  a functor   
 $X_{0} \colon I \to G\Coarse$.  
   Assume further that we are given two lifts  $X,X^{\prime}$ of $X_{0}$ to functors from $I$ to $G\BC$ along the forgetful functor $G\BC\to G\Coarse$
   as depicted in the following diagram:
   \[ \xymatrix{&&G\BC\ar[d]\\I\ar[rr]_-{X_{0}}\ar[urr]^-{X,X^{\prime}}&&G\Coarse}\]

 Extending  the notion of continuous equivalence (\cref{grioer34g43g43ggg4}), we call two functors $I\to G\BC$ continuously equivalent if they become equivalent after application of $\Yo_{c}^{s}$.
  
 \begin{lem} \label{biorgregergrege1111}
  If $X(i)$ is dual to $X^{\prime}(i)$ for every $i$ in $I$,
 then  $X \otimes G_{can,max}$ and $X^{\prime} \otimes G_{can,min}$ are  continuously equivalent.
 \end{lem}
 \begin{proof} 
 For $i$ in $I$
  let $\cL_{X(i)}$ and $\cL_{X^{\prime}(i)}$  be the posets  of invariant locally finite subsets 
 of $X(i)\otimes G_{can,max}$   and  $X'(i) \otimes G_{can,min}$   equipped with their induced structures, respectively. 
  {We first show that the assumption of the lemma implies an equality of posets}
$\cL_{X(i)}=\cL_{X^{\prime}(i)}$. Indeed, the assumption says that the collections of underlying sets of  the elements of $\cL_{X(i)}$ and $\cL_{X^{\prime}(i)}$ are equal. In addition, 
 for $L$ in $\cL_{X(i)}$ its   coarse structure  coincides with the one induced from $X^{\prime}(i)\otimes G_{can,min}$.
 Finally,   in view of the definition of the notion of local finiteness,   the induced bornological structures
 from $X(i)\otimes G_{can,max}$ and $X^{\prime}(i)\otimes G_{can,min}$
  are the minimal  one  in both cases.  

  We have a functor $I\to \mathbf{Poset}$ which sends $i$ in $I$ to the poset $\cL_{X(i)}$ and $i\to i^{\prime}$ to the map
  $\cL_{X(i)}\to \cL_{X(i^{\prime})}$ induced by the proper map $X(i)\to X(i^{\prime})$.
  We let $I^{X}$  be the Grothendieck construction for this functor. 
   
  We have a functor from $I^{X}$ to {spans} in $G\BC$ which evaluates on the object $(i,L)$ of $I^{X}$ with $L\in \cL_{X(i)}$ to
  \[X(i)\otimes G_{can,max}\leftarrow L =L^{\prime}\to X^{\prime}(i)\otimes G_{can,min}\ .\]
  Here $L^{\prime}$ is the set $L$ considered as an element of $\cL_{X^{\prime}(i)}$.
  
  We now apply $\Yo_{c}^{s}$ and form the left Kan extension of the resulting  diagram  along the forgetful functor $I^{X}\to I$. Then we get a functor {from $I$ to the category of spans in $G\Sp\cX_{c}$} which evaluates at $i$ in $I$ to
 \[\Yo^{s}_{c}(X(i)\otimes G_{can,max})\xleftarrow{\simeq}  \colim_{L\in \cL_{X(i)}} \Yo_{c}^{s}(L)=\colim_{L^{\prime}\in \cL_{X^{\prime}(i)}} \Yo_{c}^{s}(L^{\prime})\xrightarrow{\simeq} \Yo_{c}^{s}(X^{\prime}(i)\otimes G_{can,min})\ .\]
 By continuity of $\Yo^{s}_{c}$, see \cref{rguierog34trt34g}, the left and the right morphisms are equivalences as indicated. Therefore, this diagram provides the equivalence claimed in the lemma.
\end{proof}

 \section{Continuous equivalence of coarse structures}

In general, the value of an equivariant  coarse homology theory on $G$-bornological coarse spaces depends non-trivially on the coarse structure. In this section we show that in the case of a contiuous equivariant coarse homology theory one can change the coarse structure to some extent without changing the value of the homology theory. This is formalized in the notion of a continuous equivalence{; see \cref{grioer34g43g43ggg4}}.

 Let $X$ be a $G$-bornological space with two compatible $G$-coarse structures $\cC$ and $\cC^{\prime}$ such that $\cC\subseteq \cC^{\prime}$.
We write $X_\cC$ and $X_{\cC^{\prime}}$ for the associated $G$-bornological coarse spaces.
\begin{lem}\label{rgirogergergergerg}
Assume that for every locally finite subset $L$ of $X$ the coarse structures on $L$ induced by $\cC$ and $\cC'$ coincide.
Then $\id_X \colon X_\cC \to X_{\cC'}$ is a continuous equivalence.
\end{lem}
\begin{proof}
Let $\cL $  denote  the poset  of locally finite subsets of $X$. 
Then the claim follows from the commutative square
\[\xymatrix{
 \colim_{L \in \cL} \Yo^s_c(L_{X_{\cC}})\ar[r]^-{\simeq}\ar[d]^-{\simeq} & \Yo^s_c(X_\cC)\ar[d]^{\Yo^{s}_{c}(\id_{X})} \\
 \colim_{L \in \cL} \Yo^s_c(L_{X_{\cC^{\prime}}})\ar[r]^-{\simeq} & \Yo^s_c(X_{\cC'})
}\]
 The horizontal maps are equivalences by continuity{; see} \cref{rguierog34trt34g}. The left vertical map is an equivalence 
since $L_{X_{\cC}}=L_{X_{\cC^{\prime}}}$ for every $L$ in $\cL$ by assumption,
where $L_{X_{\cC}}$ indicates that we equip $L$ with the coarse structure induced from $X_{\cC}$.
\end{proof}

The  identity on the underlying sets induces a    morphism \begin{equation}\label{ewfioj2of23f23ffff}G_{can,max}\to G_{max,max}\end{equation} of $G$-bornological coarse spaces. If $X$ is a $G$-bornological coarse space, then we get an induced morphism
\begin{equation}\label{ewfioj2of23f23ff}
X\otimes G_{can,max}\to X\otimes G_{max,max}\ .
\end{equation}
 \begin{lem}\label{wtquzdguz4}
	If $X$ is $G$-bounded, then the morphism \eqref{ewfioj2of23f23ff}
	is a continuous equivalence.
	 \end{lem}
\begin{proof}
	Let $L$ be a $G$-invariant locally finite subset of the underlying bornological space of $X\otimes G_{can,max}$. By \cref{rgirogergergergerg}, it suffices to show that the coarse structure induced on $L$ from  $X\otimes G_{max,max}$ is contained in the coarse structure induced from $X\otimes G_{can,max}$ (since the other containment is obvious).
	
	Since $X$ is $G$-bounded  (see \cref{regij4o34g3g}) by assumption, there exists a bounded subset $A$ of $X$ such that $GA=X$. Let $U$ be an invariant entourage of $X$ containing the diagonal.
	 {It will suffice to show that $(U \times (G \times G)) \cap (L \times L)$ is an element of the coarse structure induced on $L$ by $X \otimes G_{can,max}$. Note that there is an implicit reordering of the factors in the product to make sense of the intersection.}
	
	Note that  $U[A]$ is bounded in $X$ and that $U\subseteq G(U[A]\times A)$. Because $L$ is locally finite, 
	$L^{\prime}:=L\cap (U[A]\times G)$ is finite. Let $W$ be the projection of $L^{\prime}$ to $G$. It is a finite subset of~$G$.
	 {We claim that}
	\[(U\times (G\times G))\cap (L\times L)\subseteq (U\times G(W\times W)) \cap (L\times L)\ .\]
	Indeed, the condition that 
	\[  {(ga,ga^{\prime},h,h^{\prime})} \in (G(U[A]\times A)\times (G\times G))\cap (L\times L)\]
	with  $a\in U[A]$ and $a^{\prime}\in A$ is equivalent to 
	\[ {(a,a^{\prime},g^{-1}h,g^{-1}h^{\prime})} \in ((U[A]\times A)\times (G\times G))\cap (L\times L)\ .\]
	This implies that $g^{-1}h\in W$ and $g^{-1}h^{\prime}\in W$, and hence
	$(h,h^{\prime})\in G(W\times W)$.	
	
	Hence we conclude that the restriction of $U\times (G\times G)$ to $L$ is contained in the entourage $(U\times G(W\times W))\cap (L\times L)$ induced from $X\otimes G_{can,max}$.
\end{proof}

\begin{ddd}
	We let $G\Sp\cX_{bd}$ denote the full subcategory of $G\Sp\cX$ generated under colimits by the images of $G$-bounded $G$-bornological coarse spaces under $\Yo^s$.
\end{ddd}

\begin{ex}\label{rgieorogergergerg}
	Let $K$ be a $G$-simplicial complex. We consider the $G$-uniform bornological coarse space $K_{d,d,d}$  obtained from  $K$  with the structures induced by the spherical path quasi-metric.
	We claim that if $K$ is $G$-finite, then $\cO^{\infty}(K_{d,d,d})$ belongs to $G\Sp\cX_{bd}$.
	Indeed, $K$ has finitely many $G$-cells. In view of the homological properties of $\cO^{\infty}$
	we know that $\cO^{\infty}(K_{d,d,d})$ is a finite colimit of objects of the form
	$\cO^{\infty}(S_{disc,min,min})$ for $S$ in $G\Orb$; compare the proof of \cref{gio3jo24f22f3}.
	Because $\cO^{\infty}(S_{disc,min,min})\simeq \Sigma \Yo^{s}(S_{min,min})$ by \cite[Prop.~9.35]{equicoarse} and $S_{min,min}$ is $G$-bounded,
	we conclude the claim.
\end{ex}

 The morphism \eqref{ewfioj2of23f23ffff}  in turn induces a
  natural transformation between endofunctors  \begin{equation}\label{vrkhi34fjrvr111}-\otimes G_{can,max}\to -\otimes G_{max,max} \colon G\Sp\cX\to G\Sp\cX\ .\end{equation}

\begin{kor}\label{rgioreo44reg} If $X$ belongs to $G\Sp\cX_{bd}$, then \eqref{vrkhi34fjrvr111} induces a continuous equivalence
\[X\otimes G_{can,max}\to  X\otimes G_{max,max}\ .\]\end{kor}
\begin{proof}
  This follows directly from \cref{wtquzdguz4} since the symmetric monoidal structure  $\otimes$ on $G\Sp\cX$ commutes with colimits in each variable separately{; see} \cite[Lem.~4.17]{equicoarse}.
\end{proof}

Recall \cref{fiofhjweio23r23r23r}  of the $G$-set of coarse components $\pi_{0}(X)$ of a $G$-coarse space $X$.  

Let $X$ be a $G$-set with two $G$-coarse structures $\cC$ and $\cC^{\prime}$ such that $\cC\subseteq \cC^{\prime}$.
We write 
$X_{\cC,max}$ and $X_{\cC^{\prime},max}$ for the associated $G$-bornological coarse spaces with the maximal bornology.
\begin{lem}\label{wtquzdguz5}
	If the canonical map $\pi_{0}(X_{\cC}) \to \pi_{0}(X_{\cC^{\prime}})$ is an isomorphism,
	then the morphism
	\[X_{\cC,max}\otimes G_{can,min}\to X_{\cC^{\prime},max}\otimes G_{can,min}\] is a   continuous equivalence. \end{lem}
\begin{proof}
	Let $L$ be a locally finite subset of the underlying  $G$-bornological space of $X_{\cC,max}\otimes G_{can,min}$.
	By \cref{rgirogergergergerg}, it suffices to show that every entourage of the coarse structure induced on $L$ by $X_{\cC^{\prime},max}\otimes G_{can,min}$ is contained in an entourage of the coarse structure induced from $X_{\cC,max}\otimes G_{can,min}$.

	Let $W:=G(B\times B)$ be an entourage of $G_{can,min}$ for some bounded subset $B$ of $G_{can,min}$.  {We can assume that $B$ contains the neutral element and is closed under inverses since this will only enlarge the entourage $W$.}
	Furthermore, let 
	$V$ be in $\cC^{\prime}$.  It suffices to show that
	$(V \times W)\cap (L\times L)$ is contained in an entourage of the form
	$(U\times W^{2})\cap (L\times L)$ for some entourage $U$ in $\cC$, where  $W^{2}:=W\circ W$ denotes the composition of $W$ with itself,  see \eqref{rv3roih43iuoff3fwe}.
	 {Note that we are implicitly permuting the factors of the products to make sense of the intersection.}

	The subset  $B^{\prime}:=W[B]$ of $G$ is finite.
	Note that  $L_{1}$, see \eqref{ehfiohefoiwhf}, and hence also $B^{\prime}L_{1}$ are finite. Since $\pi_{0}( {X_\cC})\cong \pi_{0}( {X_{\cC'}})$,
	there exists an invariant entourage $U$ of $X$ such that
	$V\cap (L_{1}\times B^{\prime}L_{1})\subseteq U$. 
	 {We show that this implies}
	\[ (V\times W)\cap (L\times L)\subseteq (U\times W^{2})\cap (L\times L)\ .\]
	Indeed, for  $l,l^{\prime}$ in $L_{1}$ the condition
	$((gl,g),(g^{\prime}l^{\prime},g^{\prime}))\in V \times W$ implies {$(g,g')\in W$. Hence there exists $h$ in $G$ such that $hg$ and $hg'$ are contained in $B$. Then $(hg',1)$ and thus $(g^{-1}g',(hg)^{-1})$ are in $W$. Since $hg$ is in $B$, so is $(hg)^{-1}$. Hence $g^{-1}g'$ is in $W[B]$ and $g'$ is in $gW[B]=gB'$.} 
	We write $g^{\prime}=gb$ for $b$ in $B^{\prime}$. Then 
	$((l,1),(bl^{\prime},b))\in U\times W^{2}$ and hence also
	$((gl,g),(g^{\prime}l,g^{\prime}))\in U\times W^{2}$ by $G$-invariance of $U$ and $W^{2}$.
\end{proof}

\section{Assembly and forget-control maps}
Morally, an assembly map is the map induced in an equivariant homology theory by the projection $W\to *$ for some $G$-topological space $W$ with certain relations with classifying spaces.
In the present section $W$ will be the Rips complex associated to a $G$-bornological coarse space $X$.
 
On the other side, the  prototype for a forget-control map  is the map
$F^{\infty}(X)\to  \Sigma F^{0}(X)$ induced by the cone boundary.

These two maps will be twisted  by $G$-bornological coarse spaces derived from the $G$-set $G$ equipped with   suitable coarse and bornological structures. The notation for the assembly map associated to a $G$-bornological coarse space will be $\alpha_{X}$, and the forget-control map will be denoted by $\beta_{X}$.

 {In this section, we compare the assembly map $\alpha_{X}$ and the forget-control map $\beta_{X}$. The main results are  \cref{rgiofwewefeewfewf} and \cref{cor:comparison}.}

The comparison argument will go through  intermediate versions of the forget-control map denoted by $\beta^{\pi_{0}}_{X}$ and $\beta^{\pi_{0}^{weak}}_{X}$. The structure of the comparison argument is as follows:
\begin{enumerate}
\item $\beta_{X}$ and $\beta_{X}^{\pi_{0}}$ are  compared in  \cref{lem:beta-beta-pi0}.
\item $\beta_{X}^{\pi_{0}}$ and $\beta_{X}^{\pi_{0}^{\weak}}$ are compared in   \cref{lem:beta-pi0-beta-pi0weak}.
\item $\beta_{X}^{\pi_{0}^{\weak}}$ and $\alpha_{X}$ are compared in  \cref{lem:alpha-beta-pi0weak}.
\end{enumerate}
The combination of these results yields one of the main results (\cref{rgiofwewefeewfewf}).

Before we consider the forget-control maps themselves, we investigate  preliminary versions of them defined on $G$-simplicial complexes.  Let $G\Simpl$ denote the category of $G$-simplicial complexes. A $G$-simplicial complex  $K$
  comes with the invariant spherical path quasi-metric which induces a $G$-uniform bornological coarse structure on $K$. We refer to  \cref{ufewiofheiwfhuewifewew} and \cref{wrgoijo42reg} for the corresponding notation. We thus have the following functors  \begin{equation}\label{relkjeorigrfr43rrf43f}
k_{d,d,d}\ , k_{d,d,max}\ , k_{d,max,max}\colon G\Simpl\to G\UBC\ ,\ K\mapsto K_{d,d,d}\  , K_{d,d,max}\ , k_{d,max,max}\ , 
\end{equation}
and 
\[k_{d,d}\ ,k_{d,max}\ ,k_{max,max}\colon G\Simpl\to G\BC\ , \quad K\mapsto K_{d,d}\ , K_{d,max}\ , K_{max,max}\ .\]
Note that $\cF\circ k_{d,d,d}\simeq k_{d,d}$, $\cF\circ k_{d,d,max}\simeq k_{d,max}$ and $\cF\circ k_{d,max,max}\simeq k_{max,max}$, where 
$\cF$ is  the forgetful functor \eqref{veroihi3uuhfove}.
 
We consider the transformations between functors $G\Simpl\to G\Sp\cX$
obtained by precomposing  the cone boundary map \eqref{feojp2r2} with $k_{d,d,d}$ or $k_{d,max,max}$:
\begin{equation}\label{jjkjkvwev}
\beta^{max} \colon (\cO^{\infty}\circ k_{d,max,max})\otimes G_{can,min}\to (\Sigma \Yo^{s} \circ k_{max,max})\otimes   G_{can,min}
\end{equation}
and
\begin{equation}\label{jjkjkvwev1}
\beta^d \colon (\cO^{\infty}\circ k_{d,d,d})\otimes G_{max,max}\to (\Sigma \Yo^{s} \circ k_{d,d})\otimes   G_{max,max} \ .
\end{equation}
Recall \cref{gu92zue98wregwew} of the notion of $G$-finiteness of a $G$-simplicial complex $K$.
\begin{ddd} 
		A $G$-simplicial complex $K$ is \emph{$G$-proper} if the $G$-bornological space $K_{d}$ is $G$-proper (see \cref{def:gproper}).
\end{ddd}

We let $G\Simpl^{\itconn,\itprop,\itfin} $ denote the full subcategory of $  G\Simpl$ of connected, $G$-proper and $G$-finite $G$-simplicial complexes.

Extending  the notion of continuous equivalence (\cref{grioer34g43g43ggg4}), we call two transformations between  $G\Sp\cX$-valued functors continuously equivalent, if they become equivalent after application of $C^{s}${; see} \eqref{oih4oifhoifh23f2f23dd2d2d2}.

\begin{prop}\label{fiwuowefewfwfewf}
The restrictions of the transformations $\beta^{max}$  \eqref{jjkjkvwev} and $\beta^{d}$  \eqref{jjkjkvwev1} to $G\Simpl^{\itconn,\itprop,\itfin} $ are canonically continuously equivalent.
\end{prop}

\begin{proof}
Let $K$ be an object of $G\Simpl^{\itconn,\itprop,\itfin} $. Then we
have a commuting square
	\[\xymatrix{ 
		\cO^{\infty}(K_{d,d,max})\otimes G_{can,min}\ar[r]\ar[d]^{\simeq}&  \Sigma \Yo^{s} (K_{d,max}\otimes G_{can,min})\ar[d] \\
		\cO^{\infty}(K_{d,max,max})\otimes G_{can,min}\ar[r]^-{\beta^{max}}&\Sigma \Yo^{s}(K_{max,max} \otimes G_{can,min})  }\]
	 {which is natural in $K$.} 	The left vertical map is an equivalence since $K_{d,d,max}\to K_{d,max,max}$    is a coarsening and $\cO^{\infty}$ sends coarsenings to equivalences \cite[Prop. 9.33]{equicoarse}.
	The right vertical map is a continuous equivalence by \cref{wtquzdguz5} because both $K_{d,max}$ and  $K_{max,max}$ are coarsely connected. Note that this is the only place where we use that $K$ is connected. 
	
We now claim that we can apply \cref{biorgregergrege1111} in order to conclude that the  {map}
	\[ \cO^{\infty}(K_{d,d,max})\otimes G_{can,min} \to \Sigma \Yo^{s} (K_{d,max}\otimes G_{can,min}) \]
	is canonically continuously equivalent to the {map}
	\[ \cO^{\infty}(K_{d,d,d})\otimes G_{can,max} \to  \Sigma \Yo^{s} (K_{d,d}\otimes G_{can,max}) \ .\]
	 {Recall from \cref{wfefuehwifefw} the hybrid coarse structure $X_h$ associated to a $G$-uniform bornological coarse space $X$. Moreover, recall} from  \eqref{t4blopg34g3g3g} that the cone boundary is given by the map
	\[ \cO^{\infty}(Z)\simeq \Yo^{s}((\R\otimes Z)_{h})\to \Yo^{s}(\R\otimes \cF(Z))\simeq \Sigma \Yo^{s}(\cF(Z))\ ,\]
	where the second map is induced by the identity of the underlying sets, and the third equivalence follows from excision.
	
 We apply  \cref{biorgregergrege1111} to the index category
	\[I:=G\Simpl^{\itconn,\itprop,\itfin}\times \Delta^{1}\] and the functor
	$X_{0} \colon I\to G\Coarse$ given on objects by
	\begin{enumerate}
	\item $(K,0)\mapsto [(\R\otimes K_{d,d,max})_{h}]^{\cC}$ and
	\item  $(K,1)\mapsto [\R\otimes K_{d,max}]^{\cC}$\ ,
	\end{enumerate}
	where the   notation $[...]^{\cC}$ indicates that we take the underlying $G$-coarse spaces.
	While the action of this functor on the morphisms in $I$ coming from morphisms $K\to K^{\prime} $ in $G\Simpl^{\itconn,\itprop,\itfin}$ is clear, it sends the morphism
	$(K,0)\to (K,1)$ coming from $0\to 1$ in $\Delta^{1}$ to the map    \[[(\R\otimes K_{d,d,max})_{h}]^{\cC}\to [\R\otimes K_{d,max}]^{\cC}\]
	given by the identity on the underlying sets.
	The lifts $X$ and $X^{\prime}$ of this functor to $G\BC$ are given on objects by
\begin{enumerate}
	\item $(K,0)\mapsto (\R\otimes K_{d,d,max})_{h}$
	\item  $(K,1)\mapsto \R\otimes K_{d,max}$
	\end{enumerate}
	for $X$, and by
	\begin{enumerate}
	\item $(K,0)\mapsto (\R\otimes K_{d,d,d})_{h}$
	\item  $(K,1)\mapsto \R\otimes K_{d,d}$
	\end{enumerate}
	for $X^{\prime}$,
	while the lifts on the level of morphisms are clear.

	 {We claim that for} every $(K,i)$ in $I$ the value $X(K,i)$ is dual to  $X^{\prime}(K,i)$. Indeed, since the $G$-bornological space $K_{d}$ is $G$-proper and $G$-bounded (since $K$ is $G$-finite), $K_{d}$ is dual to $K_{max}$ by \cref{wtquzdguz3} (applied with $Y$ a point). Furthermore, 
	 the $G$-bornological space $\R\otimes K_{d}$ is dual to $\R\otimes K_{max}$, again  by \cref{wtquzdguz3} (applied  with $Y=\R$). This finishes the verification of the claim.

	Finally, we have the natural commuting square
	\[\xymatrix{\cO^{\infty}(K_{d,d,d})\otimes G_{can,max}\ar[r]\ar[d]&\Sigma \Yo^{s}(K_{d,d} \otimes G_{can,max}) \ar[d]\\
		\cO^{\infty}(K_{d,d,d})\otimes G_{max,max}\ar[r]^-{\beta^d} &  \Sigma \Yo^{s} (K_{d,d}\otimes G_{max,max})  }\]
	The right vertical map is a continuous equivalence by Lemma \ref{wtquzdguz4} since $K_{d,d}$ is $G$-bounded.
	Since $K$ is $G$-finite, by \cref{rgieorogergergerg}
	we know that $\cO^{\infty}(K_{d,d,d})\in G\Sp\cX_{bd}$.
	Hence the left vertical morphism is a continuous equivalence by \cref{rgioreo44reg}.
\end{proof}

If the $G$-simplicial complex $K$ is not connected, then the proof of Proposition \ref{fiwuowefewfwfewf}  establishes a modified assertion.
For its formulation we first introduce some notation.

Let $X$ be a $G$-coarse space and let $U_{\pi_{0}}$ be the entourage from \eqref{v4kjhuiuhedcwcwcwc}.

\begin{ddd}\label{fioffef}We let $X_{\pi_{0}}$ denote the $G$-set $X$ with the $G$-coarse structure $\cC_{\pi_{0}}$ generated by $U_{\pi_{0}}$.
\end{ddd}
Note the following.
\begin{enumerate}
	\item {The identity of the underlying set yields a controlled map $X\to X_{\pi_{0}}$ which induces an isomorphism $\pi_{0}(X)\xrightarrow{\cong} \pi_{0}(X_{\pi_{0}})$.}
	\item If $X$ is coarsely connected, then $X_{\pi_{0}}\cong X_{max}$.
\end{enumerate}

We actually obtain  functors
\[k_{d,\pi_{0},max}\colon G\Simpl\to G\UBC , \quad K\mapsto K_{d,\pi_{0},max}\]
and
\[k_{\pi_{0},max}\colon G\Simpl\to G\BC , \quad K\mapsto K_{\pi_{0},max}\ .\]

Similar to the transformation $\beta^{max}$ from \eqref{jjkjkvwev}, we define a natural transformation of functors  $G\Simpl\to G\Sp\cX$
\begin{equation}\label{jjkjkvwev11} {\beta^{\pi_0}} \colon (\cO^{\infty}\circ k_{d,\pi_{0},max})\otimes G_{can,min}\to  (\Sigma \Yo^{s} \circ k_{\pi_{0},max})\otimes   G_{can,min}\ .\end{equation}

Let $G\Simpl^{\itprop,\itfin}$ denote the full subcategory of $G\Simpl$ of $G$-proper and $G$-finite $G$-simplicial complexes.
The proof of \cref{fiwuowefewfwfewf} shows the following proposition.

\begin{prop}\label{fuwef9iewfewfewfwef}
The restrictions of the transformations $\beta^{\pi_0}$  from \eqref{jjkjkvwev11}  and $\beta^d$ from \eqref{jjkjkvwev1} to $G\Simpl^{\itprop,\itfin}$
are canonically continuously equivalent.
\end{prop}

The following definition is adapted from  \cite[Def.~3.24]{roe}. Let $X$ be a bornological coarse space.
\begin{ddd}\label{roerogergregergreg}
	$X$ is \emph{uniformly discrete} if the bornology is the minimal bornology (see  \cref{ufewiofheiwfhuewifewew}) and for every entourage $U$ of $X$ there is a uniform bound for  the cardinalities of the sets   $U[x]$ for all points $x$ in $X$. 
\end{ddd}
\begin{rem}\label{3goi334t3tt}
	In \cite{buen} we called this property   \emph{strongly bounded geometry}.  It is not invariant under coarse equivalences. The adjective \emph{strongly} distinguishes this notion from the notion of \emph{bounded geometry} which is invariant under coarse equivalences. 
\end{rem}

\begin{ex}
	The $G$-bornological coarse space  $G_{can,min}$ is uniformly discrete.
\end{ex}

\begin{rem} \label{girjgorgergergreg}
	Let $X$ be a $G$-bornological coarse space and $U$ be an invariant entourage of $X$.
	The condition that $X$  is uniformly discrete has the following consequences:
	\begin{enumerate}
		\item $P_{U}(X)$ is a finite-dimensional, locally finite 
		simplicial complex. {Furthermore, for $X=G_{can,min}$ the $G$-simplicial complex $P_U(G_{can,min})$ is $G$-finite, i.e.\ it belongs to $G_{\Fin(G)}\Simpl^{\fin}$; see \cref{gu92zue98wregwew}.}
		\item {Since $X$ carries the minimal bornology and $P_U(X)$ is locally finite,} the bornology on $P_{U}(X)_{b}$ (which by definition is generated by the subsets $P_{U}(B)$ for all bounded subsets $B$ of $X$) coincides with the bornology $P_{U}(X)_{d}$ induced from the spherical path quasi-metric.\qedhere
	\end{enumerate}
\end{rem}

Let $X$ be a $G$-bornological coarse space and let $U$ be an invariant entourage of $X$.
\begin{ddd}\label{ewiufhiwfewefew}
We let  $\cC_{\pi_0^\weak}$ denote the coarse structure on $P_U(X)$  generated by the entourage
	\[\bigcup_{W\in  \pi_0(X)}P_U(W)\times P_U(W) \ .\qedhere\]
\end{ddd}
We have obvious inclusions of $G$-coarse structures \begin{equation}\label{rgerioogegerg34t}
\cC_{\pi_0}\subseteq \cC_{\pi_0^\weak}\subseteq \cC_{max}
\end{equation}
 on $P_{U}(X)$.
The coarse structure $\cC_{\pi_{0}}$  was introduced in  \cref{fioffef}  and  depends on  the coarse structure of $P_{U}(X)_{d}$ given by the path quasi-metric.
In contrast, the coarse structure $\cC_{\pi_{0}^\weak}$ is given by \cref{ewiufhiwfewefew} using  the coarse structure of $X$.
In analogy to  {\cref{def:rips}}, we have functors \begin{equation}\label{eeerregegrgreg}
P_{\pi_{0}} \colon G\BC^{\cC}\to G\UBC\ , \quad (X,U)\mapsto P_{U}(X)_{d,\pi_{0},max}
\end{equation}
and \begin{equation}\label{eeerregegrgreg1}
P_{\pi_{0}^{weak}} \colon G\BC^{\cC}\to G\UBC\ , \quad (X,U)\mapsto P_{U}(X)_{d,\pi^{weak}_{0},max}\ .
\end{equation}
In view of the first inclusion in \eqref{rgerioogegerg34t} we have
a natural transformation
\begin{equation}\label{rwefhfjwefwekjnk}
P_{\pi_{0}}\to P_{\pi_{0}^{weak}}\ . 
\end{equation}
 The following construction is analogous to \cref{gioowegfwefwfwef}. If we precompose the fibre sequence
 \eqref{feojp2r2} with one of  \eqref{eeerregegrgreg} or  \eqref{eeerregegrgreg1}, then we obtain fiber sequences of functors $G\BC^{\cC}\to G\Sp\cX$ which send $(X,U)$ to
\begin{align}\label{efweflkwef234rergrg2}
\mathclap{\Yo^{s}(P_{U}(X)_{\pi_0,max}) \to \Yo^s(\cO (P_{U}(X)_{d,\pi_0,max}))\to\cO^{\infty}(P_{U}(X)_{d,\pi_0,max}) \xrightarrow{\partial}\Sigma \Yo^{s}(P_{U}(X)_{\pi_0,max})}\notag\\
\end{align}
and to
\begin{align}\label{efweflkwef234rergrg}
\mathclap{\Yo^{s}(P_{U}(X)_{\pi_0^\weak,max}) \to \Yo^s(\cO (P_{U}(X)_{d,\pi_0^\weak,max}))\to\cO^{\infty}(P_{U}(X)_{d,\pi_0^\weak,max}) \xrightarrow{\partial}\Sigma \Yo^{s}(P_{U}(X)_{\pi_0^\weak,max})\ ,}\notag\\
\end{align}
respectively. The transformation   \eqref{rwefhfjwefwekjnk} induces a natural transformation of fibre sequence  from \eqref{efweflkwef234rergrg2} to \eqref{efweflkwef234rergrg}.

\begin{ddd} \label{rgoij34g34geregegge} We define fiber sequences of functors $G\BC \to G\Sp\cX$\begin{equation}\label{g4tjknk3jf34f34f34f}
F^{0}_{\pi_0}\to F_{\pi_0}\to F_{\pi_0}^{\infty} \xrightarrow{\partial} \Sigma F^{0}_{\pi_0}
\end{equation}
 and \begin{equation}\label{g4tjknk3jf34f34f34f1}
 F^{0}_{\pi_0^\weak}\to F_{\pi_0^\weak}\to F_{\pi_0^\weak}^{\infty} \xrightarrow{\partial} \Sigma F^{0}_{\pi_0^\weak}
\end{equation}
	 by left Kan extension of \eqref{efweflkwef234rergrg2} and \eqref{efweflkwef234rergrg}
		along the forgetful functor \eqref{vreoi34fg3vefv}, respectively.
\end{ddd}
 
Again we have a natural transformation of fibre sequences from \eqref{g4tjknk3jf34f34f34f} to \eqref{g4tjknk3jf34f34f34f1}.
 
 Let $X$ be a $G$-bornological coarse space. The morphisms in the following definition are induced by the natural transformation denoted by $\partial$ in \cref{gioowegfwefwfwef} or \cref{rgoij34g34geregegge}.

\begin{ddd}\label{viojweiovjvwevwefwecwcwcw}
	The map \begin{equation}\label{bthjrvnkj3rv3rvevevev}
\beta_{X} \colon F^\infty(X) \otimes G_{max,max}\to \Sigma F^0(X) \otimes G_{max,max}
\end{equation}
 	in $G\Sp\cX$ is called the \emph{forget-control map}.
\end{ddd}
	
The maps 
\begin{equation}\label{geroigergerrgeg}
\beta^{\pi_0}_{X} \colon F_{\pi_0}^\infty(X) \otimes G_{can,min}\to \Sigma F_{\pi_0}^0(X) \otimes G_{can,min}
\end{equation}
and \begin{equation}\label{geroigergerrgeg1}
\beta^{\pi_0^\weak}_{X} \colon F_{\pi_0^\weak}^\infty(X) \otimes G_{can,min}\to \Sigma F_{\pi_0^\weak}^0(X) \otimes G_{can,min}
\end{equation}
are intermediate versions of the forget-control map and used in the comparison argument. 
 
 Let $X$ be a $G$-bornological coarse space.
 
 \begin{lem}\label{lem:beta-beta-pi0}
  Assume:
	\begin{enumerate}
		\item $X$ is uniformly discrete.
		\item $X$ is $G$-proper.
		\item $X$ is $G$-finite{, i.e.~$G \backslash X$ is a finite set}.		
	\end{enumerate}
	Then the maps $\beta_X$ and $\beta^{\pi_0}_X$ in \eqref{geroigergerrgeg} and \eqref{geroigergerrgeg1} are canonically continuously equivalent.
 \end{lem}
\begin{proof} In view of \cref{gioowegfwefwfwef} and  \cref{rgoipewerfwefewfew},
the morphism $\beta_{X}$ is given as a colimit of the diagram of morphisms
\begin{equation}\label{rgierogregege}
 \cO^{\infty}(P_{U}(X)_{d,d,b})\otimes G_{max,max}\to \Sigma \Yo^{s}(P_{U}(X)_{d,d,b})\otimes G_{max,max}
\end{equation}  indexed by the poset $\cC^{G}(X)$ (obtained by precomposing   \eqref{jjkjkvwev1}  with  the functor  $P_{-}(X) \colon \cC^{G}(X)\to G\Simpl $).
Similarly, the morphism $\beta_{X}^{\pi_{0}}$ is given as a colimit of the diagram of morphisms \begin{equation}\label{rgierogregege1}
\cO^{\infty}(P_{U}(X)_{d,\pi_{0},max})\otimes G_{can,min}\to  \Sigma \Yo^{s}(P_{U}(X)_{d,\pi_{0},max})\otimes G_{can,min}\ .
\end{equation} indexed by $\cC^{G}(X)$ (again obtained by  precomposing  \eqref{jjkjkvwev11} with the functor  $P_{-}(X)$).

 Since $X$ is uniformly discrete and $G$-finite, for every $U$ in $\cC^{G}(X)$  the $G$-simplicial complex $P_{U}(X)$ is $G$-finite.
 In addition, the bornology induced from the metric coincides with the bornology induced from $X${; see} \cref{girjgorgergergreg}.
 Finally, since $X$ is $G$-proper,  the $G$-simplicial complex $P_{U}(X)$ is  also $G$-proper. Hence $P_{U}(X)$  belongs to $G\Simpl^{\itprop,\itfin}$.
 
 We can now apply \cref{fuwef9iewfewfewfwef} and conclude that the diagrams (parametrized by $U$ in $\cC^{G}(X)$) of  morphisms \eqref{rgierogregege} and \eqref{rgierogregege1}  
 are  canonically equivalent.
 Therefore, their colimits $\beta_{X}$ and $\beta_{X}^{\pi_{0}}$ are canonically equivalent, too.
 \end{proof}

 Let $X$ be a $G$-bornological coarse space.
 
 \begin{lem}\label{lem:beta-pi0-beta-pi0weak}
  Assume:
	\begin{enumerate}
		\item $X$ is uniformly discrete.
		\item $X$ is $G$-proper.
		\item $X$ is $G$-finite.		
	\end{enumerate}
	Then $\beta^{\pi_0}_X$ and $\beta^{\pi_0^\weak}_X$ are canonically continuously equivalent.
 \end{lem}
\begin{proof} We consider  an invariant entourage
  $U$ of $X$ and form
   the commutative square
 \[\xymatrix{
  \cO^{\infty}(P_{U}(X)_{d,\pi_0,max})\otimes G_{can,min}\ar[r]^-{\partial}\ar[d] & \Sigma \Yo^{s} (P_{U}(X)_{\pi_{0},max}\otimes   G_{can,min})\ar[d] \\
  \cO^{\infty}(P_{U}(X)_{d,\pi_0^\weak,max})\otimes G_{can,min}\ar[r]^-{\partial} & \Sigma \Yo^{s} (P_{U}(X)_{\pi_{0}^\weak,max}\otimes   G_{can,min})
 }\]
in $G\Sp\cX$, where the vertical morphisms are induced by \eqref{rwefhfjwefwekjnk}.
 In view of Lemma \ref{rgoipewerfwefewfew}, after taking colimits over $U$ in the poset $\cC^{G}(X)$, the horizontal maps become equivalent to $\beta^{\pi_0}_X$ and $\beta^{\pi_0^\weak}_X$, respectively.
 
 The left vertical morphism is an equivalence since it is obtained by applying $\cO^{\infty}$ to a   coarsening and $\cO^{\infty}$ sends coarsenings to equivalences by \cite[Prop. 9.34]{equicoarse}. 
 
 It remains to show that the right vertical map becomes a continuous equivalence after taking the colimit over $\cC^G(X)$.  We let $\cF(U)$ denote the poset of invariant locally finite  subsets of the $G$-bornological space $P_{U}(X)_{max}\otimes G_{min}$. We then consider
the following commutative diagram 
\[\xymatrix{\colim\limits_{U\in \cC^{G}(X)} \colim\limits_{L\in \cF(U)} \Yo_{c}^{s}( L_{P_{U}(X)_{\pi_{0},max}\otimes G_{can,min}})\ar[r]\ar[d]
		&\colim\limits_{U\in \cC^{G}(X)} \colim\limits_{L\in \cF(U)} \Yo_{c}^{s}( L_{P_{U}(X)_{\pi_0^\weak,max}\otimes G_{can,min}})\ar[d]\\
		\colim\limits_{U\in \cC^{G}(X)}   \Yo_{c}^{s}(  P_{U}(X)_{\pi_0,max}\otimes G_{can,min})\ar[r]
		&\colim\limits_{U\in \cC^{G}(X)}   \Yo_{c}^{s}(  P_{U}(X)_{\pi_0^\weak,max}\otimes G_{can,min})}\]
		where the subscript indicates from which space the bornological coarse structure on $L$ is induced.
	In view of  \cref{rguierog34trt34g},  continuity of $\Yo^{s}_{c}$ implies that the vertical maps are equivalences.
	
	For $L$ in $\cF(U)$ we know that $L_{1}:=L\cap (P_{U}(X)\times \{1\})$ is finite.  There exists an invariant entourage $U^{\prime}$ of $X$ such that $U\subseteq U^{\prime}$ and such that the condition on a subset
	$F$ of $L_{1}$ 
	\begin{itemize}
	\item $F$ is contained in $P_U(W)$ for some $W$ in $\pi_{0}(X)$\end{itemize}
	implies the condition
	\begin{itemize}
	\item 
	 $F$ is contained in a single simplex of $P_{U^{\prime}}(X)$.
	 \end{itemize}
Then the coarse structures induced on $L$ from 
	$  P_{U^{\prime}}(X)_{\pi_{0},max}\otimes   G_{can,min} $ and $  P_{U^{\prime}}(X)_{\pi_0^\weak,max}\otimes   G_{can,min} $ coincide. By a cofinality consideration the upper horizontal map is hence an equivalence. It follows that the lower horizontal map is an equivalence as desired.
\end{proof}

Recall from \cref{efweu9u9wefwef} that we have
functors \begin{equation}\label{iobperberberbre}
G\Top \xrightarrow{\ell} G\Top[W_{G}^{-1}] \xrightarrow[\simeq]{\ff} \PSh(G\Orb)\ .
\end{equation}
 Let $\bC$ be a   cocomplete  stable $\infty$-category and $H \colon G\Top \to \bC$ be a functor. \begin{ddd}\label{ergergregergreg}
 The functor $H$ is an \emph{equivariant homology theory}  if it is equivalent to 
 the {restriction} along \eqref{iobperberberbre} of a colimit-preserving functor $ \PSh(G\Orb)\to \bC$.
\end{ddd}
\begin{rem}\label{fjieowjfoijfiofwefwefweff}
Note that in \cite[Def.~10.3]{equicoarse} we use the term \emph{strong  equivariant homology theory} for the objects defined in \cref{ergergregergreg}
  in order to distinguish it from the classical  notion of an equivariant homology theory as defined  \cite[Def.~10.4]{equicoarse}. For the purpose of the present paper we will employ the more natural definition above and drop the word strong.
   \end{rem}
   
In view of the universal property of presheaves, the $\infty$-category $\Fun^{\colim}(\PSh(G\Orb),\bC)$ of colimit-preserving functors  is equivalent to the $\infty$-category $\Fun(G\Orb,\bC)$. Therefore, in order to specify an equivariant homology theory  or such a colimit preserving functor essentially  uniquely,  it suffices to specify the corresponding functor in $\Fun(G\Orb,\bC)$

\begin{ddd}\label{hgergergegregrege}
We define \[\tilde \cO^{\infty}_{\hlg} \colon G\Top[W_{G}^{-1}] \to G\Sp\cX\] to be the colimit-preserving functor
essentially uniquely determined by the functor
\[G\Orb\to   G\Sp\cX\ , \quad S\mapsto \cO^{\infty}(S_{disc,max,max})\ .\]
Furthermore, define the equivariant homology theory
 \begin{equation}\label{vreopkopk34pof34f34f}
 \cO^{\infty}_{\hlg}:=\tilde \cO^{\infty}_{\hlg}\circ \ell \colon G\Top  \to G\Sp\cX
\end{equation} 
\end{ddd}
\begin{rem}\label{retgioojtitj43893z894343g43}
Note that the functor $\cO^{\infty}_{\hlg}$ differs from the functor (denoted by the same symbol) defined in \cite[Def.~10.10]{equicoarse}. Both versions of this functor   coincide on CW-complexes.  In the present paper, we prefer to use the definition above since it fits better with the needs in   \cref{rgiofergegergr43534543534535}.
\end{rem}

In view of \cite[Prop.~9.35]{equicoarse} the functor $\tilde \cO^{\infty}_{\hlg}$ is equivalent to the functor essentially uniquely determined by the functor
\[ G\Orb \to G\Sp\cX\ , \quad S\mapsto   \Sigma \Yo^{s}(S_{min,max})\ .\] 

In analogy to \cref{def:rips} we consider the functor
\[ P^{top} \colon G\BC^{\cC}\to G\Top[W_{G}^{-1}]\ , \quad X\mapsto  \ell(P_{U}(X))  \ , \]
where $P_{U}(X) $ in $G\Top$ is the underlying $G$-topological space of the $G$-uniform space $P_{U}(X)_{d}$ and $\ell$ is the localization as in \eqref{iobperberberbre}. 
\begin{ddd}
We define the \emph{Rips complex functor}
\[ \Rips \colon G\BC\to  G\Top[W_{G}^{-1}] \]
as the left Kan extension of the functor
$P^{top}$ along the forgetful functor \eqref{vreoi34fg3vefv}.
\end{ddd}

If   $X$ is a $G$-bornological coarse space, then by \cref{rgoipewerfwefewfew} we have:
 \begin{kor}  \label{ergieogergergeeggreg}  The Rips complex of $X$ is  given by  \[\Rips(X)\cong \colim_{U\in \cC^{G}(X)} \ell(P_{U}(X))\ .\]  
\end{kor}

\begin{rem}
Note that the present definition of the Rips complex differs from the definition given in \cite[Def.~11.2]{equicoarse}.
In the reference we defined the Rips complex of $X$ as the $G$-topological space $\colim_{U\in \cC^{G}(X)} P_{U}(X)$.
This definition fits well with the version of $\cO^{\infty}_{\hlg}$ used there{; see} \cref{retgioojtitj43893z894343g43}.
In contrast, in the present paper we replace the colimit by the homotopy colimit.
\end{rem}

For a $G$-bornological coarse space $X$ we consider $\pi_0(X)$ as a discrete $G$-topological space. 
For every $U$ in $\cC^{G}(X)$ we have a projection 
\[P_{U}(X)\to \pi_{0}(X)\]
of $G$-topological spaces. Applying $\ell$ and forming the colimit over $\cC^{G}(X)$, we obtain
 a canonical projection morphism
\begin{equation}\label{vbeoprevervevevrvrev}
\Rips(X)\to \ell(\pi_0(X))
\end{equation} in $G\Top[W^{-1}_{G}]$.

In the following, we calculate the Rips complex of the bornological coarse space $G_{can,min}$ explicitly. \begin{lem}\label{rviuohiofe2df2edf2}
We have an equivalence
\[\ff(\Rips(G_{can,min}))\simeq E_{\Fin} G\ ,\]
 {where $\ff \colon G\Top[W_{G}^{-1}]\to \PSh(G\Orb)$ denotes the equivalence from \eqref{efwoiwggwegf}.}
\end{lem}
\begin{proof}  We must verify that $\ff(\Rips(G_{can,min}))$ satisfies the condition stated in \cref{def:efg}. Because colimits in presheaves are formed objectwise and  the equivalence $\ff$ preserves colimits, by \cref{ergieogergergeeggreg} we have the equivalence
\[ \ff(\Rips(G_{can,min}))( {S})\simeq \colim_{U\in \cC^{G}(G_{can,min})} \ff(\ell(P_{U}(G_{can,min})))( {S})  \]
 {for every transitive $G$-set $S$.}
 {By definition of $\ff$, we have}
\[ \ff(\ell(P_{U}(G_{can,min})))( {S})\simeq \ell(\Map_{G\Top}( {S}_{disc},P_{U}(G_{can,min})))\ .\]
Since all stabilizers of points in $P_{U}(G_{can,min})$ are finite, we see that
\[ \Map_{G\Top}( S_{disc},P_{U}(G_{can,min}))\cong \emptyset \]
if {$S$ has infinite stabilizers}.
If $S$ has finite stabilizers, then the argument given in the proof of \cite[Lem.~11.4]{equicoarse} shows that
\[\colim_{U\in \cC^{G}(G_{can,min})} \pi_{n} ( \Map_{G\Top}(S_{disc},P_{U}(G_{can,min}) ))\] is trivial for all $n$ in $\nat$. 
This implies that
\[ \colim_{U\in \cC^{G}(G_{can,min})} \ell( \Map_{G\Top}(S_{disc},P_{U}(G_{can,min}) ))\simeq *\ .\qedhere\]
\end{proof}

\begin{ddd}
 The \emph{assembly map} $\alpha_X$ is the map \begin{equation}\label{f3oi4pf34f34f43f3}
\alpha_X \colon \tilde \cO^\infty_\homolg(\Rips(X)) \otimes G_{can,min} \to \cO^\infty_\homolg( \pi_0(X) ) \otimes G_{can,min} 
\end{equation}
 induced by the projection  \eqref{vbeoprevervevevrvrev}.
\end{ddd}
Note that on the target of this map we used \eqref{vreopkopk34pof34f34f} in order to suppress the symbol $\ell$.

Let $G\Simpl^{\fin}$ denote the category of $G$-finite $G$-simplicial complexes. Recall the functor $k_{d,max,max}$ defined in   \eqref{relkjeorigrfr43rrf43f}.
\begin{lem}\label{rgiowr4wwggwg}We have a canonical equivalence of functors $G\Simpl^{\fin} \to G\Sp\cX$
\[(\cO_{\hlg}^{\infty})_{|G\Simpl^{\fin}} \simeq (\cO^{\infty}\circ k_{d,max,max})_{|G\Simpl^{\fin}}\ .\]
\end{lem}

\begin{proof}
The functor $\cO^{\infty}\circ k_{d,max,max}$  (see \eqref{relkjeorigrfr43rrf43f}) is  excisive for decompositions of $G$-simplicial complexes   by \cite[Lem. 10.9 and Cor.~9.36]{equicoarse}. Furthermore, it is homotopy invariant by \cite[Cor.~9.38]{equicoarse}. The functor    $\cO^{\infty}_{\hlg}$ has the same properties. By \cref{hgergergegregrege}, we have an   equivalence \[(\cO^{\infty}_{\hlg})_{|G\Orb}\simeq  (\cO^{\infty}\circ k_{d,max,max})_{|G\Orb} \ .\] for
$S$ in $G\Orb$. This implies the desired equivalence.
\end{proof}

Let $X$ be a $G$-bornological coarse space.
 \begin{lem}\label{lem:alpha-beta-pi0weak}
  Assume:
	\begin{enumerate}
		\item $X$ is uniformly discrete.
		\item $X$ is $G$-proper.
		\item $X$ is $G$-finite.		
	\end{enumerate}
	Then $\alpha_X$ and $\beta^{\pi_0^\weak}_X$ from \eqref{f3oi4pf34f34f43f3} and \eqref{geroigergerrgeg1} are canonically equivalent.
 \end{lem}
\begin{proof}
The assumptions on $X$ imply that $P_{U}(X)$ is $G$-finite
for every invariant coarse entourage $U$ of $X$. Therefore, by \cref{rgiowr4wwggwg}
we have a  {canonical} equivalence 
	\[\cO^\infty_\homolg(P_U(X))\simeq \cO^\infty(P_U(X)_{d,max,max})\ .\]
	Similarly, we have a {canonical} equivalence
	\[\cO^\infty_\homolg(\pi_0(X))\simeq \cO^\infty(\pi_0(X)_{disc,max,max})\ .\]
	These equivalences yield the lower square in the following diagram. The upper square is induced by a coarsening. Therefore the vertical maps are equivalences by \cite[Prop.~9.33]{equicoarse}.
 \begin{equation}\label{eq:asquare}\xymatrix{
     \colim\limits_{U \in \cC^G(X)} \cO^\infty(P_U(X)_{d,\pi_0^\weak,max}) \otimes G_{can,min}\ar[d]^\simeq\ar[r]& \cO^\infty(\pi_0(X)_{disc,min,max}) \otimes G_{can,min}\ar[d]^\simeq \\
  \colim\limits_{U \in \cC^G(X)} \cO^\infty(P_U(X)_{d,max,max}) \otimes G_{can,min}\ar[d]^\simeq\ar[r] & \cO^\infty(\pi_0(X)_{disc,max,max}) \otimes G_{can,min}\ar[d]^\simeq \\
  \colim\limits_{U \in \cC^G(X)} \cO^\infty_\homolg(P_U(X)) \otimes G_{can,min}\ar[r]^-{\alpha_{X}} & \cO^\infty_\homolg(\pi_0(X)) \otimes G_{can,min}
 }\end{equation}
By \cref{ergieogergergeeggreg}, \eqref{vreopkopk34pof34f34f} and the fact that $\tilde \cO^{\infty}_{\hlg}$ preserves colimits we have the equivalence
	\begin{equation*}\label{verh3jfnkjrververv}
\tilde \cO^\infty_\homolg(\Rips(X)) \simeq \colim_{U\in \cC^{G}(X)}\cO^\infty_\homolg(P_U(X))\ .
\end{equation*} 
Hence the lower horizontal map in \eqref{eq:asquare} is equivalent to $\alpha_X$ as indicated.

 {The upper horizontal arrow from \eqref{eq:asquare} fits into} the commutative square
\[\xymatrix{
     \colim\limits_{U \in \cC^G(X)} \cO^\infty(P_U(X)_{d,\pi_0^\weak,max}) \otimes G_{can,min}\ar[d]^{\beta_{X}^{\pi_{0}^{\weak}}}\ar[r]& \cO^\infty(\pi_0(X)_{disc,min,max}) \otimes G_{can,min}\ar[d]^{\simeq}\\
     \colim\limits_{U \in \cC^G(X)} \Sigma\Yo^s(P_U(X)_{\pi_0^\weak,max})\otimes G_{can,min}\ar[r]^-{\simeq}&\Sigma\Yo^s(\pi_0(X)_{min,max})\otimes G_{can,min}
     }\]
     Here the right vertical map is an equivalence by \cite[Prop.~9.35]{equicoarse}. 
     
     We now show that the lower horizontal map is an equivalence. The argument is similar to \cite[Lem.~10.7]{equicoarse}. By choosing a representative in $P_{U}(X)$  for every element of $\pi_0(X)$ we obtain a map
     $\pi_0(X)\times \{1\}\to P_{U}(X)\times \{1\}$. This map has a unique extension to a $G$-equivariant map
     $\pi_0(X)\times G \to  P_{U}(X)\times G$. We now observe that this  map is a morphism of $G$-bornological coarse spaces
     \[s \colon \pi_0(X)_{ min,max}\otimes G_{can,min}\to  P_{U}(X)_{\pi^{\weak}_{0},max}\otimes G_{can,min}\ .\]
     It is a right inverse of the projection \[p \colon P_{U}(X)_{\pi^{\weak}_{0},max}\otimes G_{can,min}\to \pi_0(X)_{ min,max}\otimes G_{can,min}\ ,\] and the composition $s\circ p$ is close to the identity by construction.
     It follows that $p$ is a coarse equivalence and this implies that  lower horizontal map is an equivalence.

     It follows that the upper horizontal map in \eqref{eq:asquare} is equivalent to $\beta_X^{\pi_0^\weak}$.
\end{proof}

Let $X$ be a $G$-bornological coarse space. Combining \cref{lem:alpha-beta-pi0weak,lem:beta-pi0-beta-pi0weak,lem:beta-beta-pi0}, we obtain the following corollary.

 \begin{kor}\label{rgiofwewefeewfewf}
 \label{kor:finitecase}
   Assume:
	\begin{enumerate}
		\item $X$ is uniformly discrete;
		\item $X$ is $G$-proper;
		\item $X$ is $G$-finite.		
	\end{enumerate}
	Then the assembly map $\alpha_X$ and the forget-control map $\beta_X$ from \eqref{f3oi4pf34f34f43f3} and \eqref{bthjrvnkj3rv3rvevevev} are canonically continuously equivalent.
 \end{kor}

In the following, we derive a version of \cref{rgiofwewefeewfewf} without the assumption of $G$-finiteness. To this end, we must modify the definition of the forget-control map.

Let $G\BC^\fin$ denote the full subcategory of $G\BC$ consisting of $G$-finite
$G$-bornological coarse spaces.
Let $E\colon G\BC^\fin\to \bC$ be some functor to a cocomplete target~$\bC$.

\begin{ddd}\label{ergiuhergoegergergergre}
We define $E_{\fin}$ as the left Kan extension
\[ \xymatrix{G\BC^\fin\ar[d]\ar[r]^-{E}&\bC\\ G\BC\ar@{-->}[ur]_-{E_\fin}&} \]
along the inclusion functor $G\BC^\fin\to G\BC$ of the restriction of $E$ to $G\BC^\fin$.
\end{ddd}
We have a canonical transformation of functors
\begin{equation}\label{rvgtgrtgtrgtrtrervervv}
E_{\fin}\to E \colon G\BC\to \bC\ .
\end{equation} 

Let $X$ be a $G$-bornological coarse space. By $\cK(X)$ we denote the poset of all invariant $G$-finite subspaces of $X$ with the induced $G$-bornological coarse structures.
\begin{lem}\label{rgioegergergerg}
	We have a canonical equivalence
	\[ \colim_{W\in \cK(X)} E(W)\simeq E_{\fin}(X) \ .\]
\end{lem}
\begin{proof}
	By the objectwise formula for the left Kan extension 
	we have  
	\[ \colim_{(W\to X)\in G\BC^\fin/X} E(W)\simeq E_{\fin}(X) \ .\]
	Since the image of a $G$-finite subspace under a morphism of $G$-bornological coarse spaces is again $G$-finite the subcategory $\cK(X)$ is cofinal in $ G\BC^\fin/X$. This implies the assertion.
\end{proof}

Recall \cref{roerogergregergreg} of the notion of uniform discreteness. In the following, we consider the transformation \eqref{rvgtgrtgtrgtrtrervervv} for the functor
$E:=\tilde \cO^{\infty}_{\hlg}\circ \Rips \colon G\BC\to G\Sp\cX$. 
Let $G\BC^{\udisc}$ be the full subcategory of $G\BC$ of uniformly discrete $G$-bornological coarse spaces.

\begin{lem}
	\label{lem:ohlg} The restriction of the transformation
	\[(\tilde \cO^\infty_{\hlg}\circ \Rips)_\fin\to\tilde \cO^\infty_{\hlg}\circ \Rips\]
	to $G\BC^{\udisc}$ is an equivalence.
\end{lem}
\begin{proof}
	If $X$ is uniformly discrete, then  for every $U$ in $\cC^{G}(X)$ the complex $P_{U}(X)$ is a {locally finite}  $G$-simplicial complex. Consequently, $P_{U}(X)$ as a $G$-topological space is a filtered colimit over its $G$-compact subsets.
 {In fact, this filtered colimit is a homotopy colimit, so it is preserved by the functor $\ell$.}
	 The subsets
	$P_{U}(L)$ of $P_{U}(X)$ for invariant $G$-finite subsets {$L$} of $X$ are cofinal in the $G$-compact subsets of $P_{U}(X)$. All this is used below to justify the  equivalence marked by $!$. At this point we further use the fact that $\tilde \cO^{\infty}_{\hlg}$ preserves  colimits in $G\Top[W_{G}^{-1}]$.     Hence if $X$ is uniformly discrete, then we have the following equivalences (the first one is due to Lemma~\ref{rgioegergergerg})
\begin{align*}
	  \tilde \cO^\infty_{\hlg}(\Rips(X))_\fin & \simeq \colim_{L\in \cK(X)} \tilde \cO^\infty_{\hlg}(\Rips(L)) \\
	  & \simeq \colim_{L\in \cK(X)} \tilde\cO^\infty_\hlg(\colim_{U \in \cC^G(X)} \ell(P_U(L))) \\
	  & \simeq \colim_{L\in \cK(X)} \colim_{U \in \cC^G(X)} \tilde\cO^\infty_\hlg(\ell(P_U(L))) \\
	  & \simeq \colim_{U \in \cC^G(X)} \colim_{L\in \cK(X)} \tilde\cO^\infty_\hlg(\ell(P_U(L))) \\
	  & \stackrel{!}{\simeq} \colim_{U \in \cC^G(X)} \tilde\cO^\infty_\hlg(\ell(P_U(X))) \\
	  & \simeq \tilde\cO^\infty_\hlg(\colim_{U \in \cC^G(X)} \ell(P_U(X))) \\
	  & \simeq \tilde\cO^\infty_{\hlg}(\Rips(X))\ .\qedhere
\end{align*}
\end{proof}

We now consider the functor
$\cO^{\infty}_{\hlg}\circ \pi_{0} \colon G\BC\to G\Sp\cX$. A similar argument as for {\cref{lem:ohlg}} shows:
\begin{lem}\label{rgi9oergergergreg}
The transformation
\[(  \cO^{\infty}_{hlg}\circ \pi_{0} )_{\fin}\to   \cO^{\infty}_{hlg}\circ  \pi_{0} \]
is an equivalence.
 \end{lem}
 We do not need to restrict to uniformly discrete spaces here since a discrete $G$-topological space is always a filtered  {(homotopy)} colimit of its $G$-finite subspaces.

In the following we use the abbreviations
$F^{x}_{\fin}$ for $(F^{x})_{\fin}$ for $x\in \{\emptyset,0,\infty\}$, and we write $\beta_{X,\fin}$ for the image
of $\beta_{X}$ under the $(-)_{\fin}$-construction.

Let $X$ be a $G$-bornological coarse space.
\begin{prop}
	\label{prop:a-b-comp}
	Assume:
	\begin{enumerate}
		\item $X$ is uniformly discrete;
		\item $X$ is $G$-proper.
	\end{enumerate}
	Then the assembly map
	\[\alpha_X\colon \tilde \cO^\infty_{\hlg}(\Rips(X))\otimes G_{can,min}\to \cO^\infty_{\hlg}(\pi_0(X))\otimes G_{can,min}\]
	is canonically continuously equivalent to the forget-control map
	\[\beta_{X,\fin}\colon F^\infty_{\fin}(X)\otimes G_{max,max}\to \Sigma F^0_{\fin}(X)\otimes G_{max,max}\ .\]
\end{prop}
\begin{proof}
	Since every invariant subspace of $X$ is again uniformly discrete and $G$-proper, the proposition follows immediately from \cref{kor:finitecase}, \cref{rgioegergergerg}, \cref{lem:ohlg} and \cref{rgi9oergergergreg}.
\end{proof}

Let $X$ be a $G$-bornological coarse space and let $S$ be a $G$-set. 
\begin{kor}
	\label{cor:comparison}
	Assume:
	\begin{enumerate}
		\item $X$ is uniformly discrete;
		\item $X$ is $G$-proper;
		\item $X$ is coarsely connected.
	\end{enumerate}
	Then the $S$-twisted assembly map
	\[\alpha_{X,S}\colon \tilde  \cO^\infty_{\hlg}(\ell(S_{disc})\times \Rips(X))\otimes G_{can,min}\to \cO^\infty_{\hlg}(S_{disc})\otimes G_{can,min}\]
	is canonically continuously equivalent to the forget-control map
	\[\beta_{S_{min,min}\otimes X}\colon F^\infty_{\fin}(S_{min,min}\otimes X)\otimes G_{max,max}\to \Sigma F^0_{\fin}(S_{min,min}\otimes X)\otimes G_{max,max}\ .\]
\end{kor}
\begin{proof} As in the proof of \cref{gioergregregregtw}, for every $U$ in $\cC^{G}(X)$ we have the natural isomorphism of $G$-simplicial complexes  
	  \[P_{\diag(S)\times U}(S_{min,min}\otimes X)\cong S_{disc}\times P_U(X)\ .\] 
	  We now apply $\ell$ and use that
	  $\ell(S_{disc}\times P_{U}(X))\simeq \ell(S_{disc})\times \ell(P_{U}(X))$ (note that $\ell$ preserves products since all $G$-topological spaces are fibrant).
	 We then form
	   the colimit over $U$ in $\cC^{G}(X)$ and use that $\ell(S_{disc})\times -$ preserves this colimit since $G\Top[W_{G}^{-1}]$ (being equivalent to $\PSh(G\Orb)$) is an $\infty$-topos.
We 	  eventually  obtain the isomorphism 
	  \[ \Rips(S_{min,min}\otimes X)\cong \ell(S_{disc})\times \Rips(X) \]
	  in $G\Top[W_{G}^{-1}]$.
	  
	  Since $X$ is coarsely connected, the projection $S_{min,min}\otimes X\to S_{min,min}$ induces a bijection on $\pi_0$. Furthermore, $\pi_0(S_{min,min})\cong S_{disc}$. The corollary now follows from \cref{prop:a-b-comp}
\end{proof}

\section{Induction}\label{efiowe32fwefwef}
Let $H$ be a subgroup of $G$. Then we have various induction functors:
\begin{enumerate}
\item $\Ind_{H}^{G} \colon H\Set\to G\Set${; see} \eqref{kljoeirjfoirejfreffwfe}.
\item \label{ewfoiwfwefewfwf} $\Ind_{H}^{G,top} \colon  {H}\Top\to G\Top$, $X\mapsto G_{disc}\times_{H}X$.
\item $\Ind_{H}^{G,htop} \colon {H}\Top[W_{H}^{-1}]\to G\Top[W_{G}^{-1}]$, the derived version of $\Ind_H^{G,top}$.
\item \label{ijf32ioj23d23d} $\Ind_{H}^{G} \colon \PSh(H\Orb)\to \PSh(G\Orb)$, the left-adjoint of the restriction functor $\Res_H^G \colon \PSh(G\Orb) \to \PSh(H\Orb)$. The latter is given by restriction along the functor $(\Ind_{H}^{G})_{|H\Orb} \colon H\Orb\to G\Orb$;
\item $\Ind_{H}^{G} \colon H\BC\to G\BC${; see} \eqref{gioreggwefewffewfwfwfwf}.
\item $\Ind_{H}^{G,Mot} \colon H\Sp\cX\to G\Sp\cX${; see} \eqref{3ljg34opg34f34f3}.
\item  $\Ind_{H}^{G,\cU} \colon H\UBC\to G\UBC${; see} \eqref{vrevoijoi3jfoi3f34f3}.
\end{enumerate}

We also have an analogous list  of  restriction functors $\Res_{H}^{G,-}$.

\begin{rem}
 
Using the description of $\ff$ given in \cref{efweu9u9wefwef},
the adjunction
\[ \Ind^G_H \colon H\Top \leftrightarrows G\Top : \Res_H^G \]
implies that we have natural equivalences
\[\hspace{-0.5cm} \Res^{G}_{H}(\ff(X))\simeq \ell(\Map_{G\Top}(G_{disc}\times_{H}S,X))\simeq
\ell(\Map_{H\Top}(S,\Res^{G,top}_{H}(X)))\simeq  \ff(\Res^{G,htop}_{H}(X)) \]
for $X$ in $G\Top$.
So $\Res_H^G$ and $\Res_H^{G,htop}$ correspond to each other 
under $\ff$. It then follows that also their left adjoints $\Ind_H^G$ and $\Ind_H^{G,htop}$ become
 identified under $\ff$.
\end{rem}

Let $X$ be an $H$-bornological coarse space. We can consider the $G$-bornological coarse spaces $G_{min,min}\otimes H_{min,min}\otimes X$ and $G_{min,min}\otimes X$, where $G$ acts both times on the first factor.
 {In the following, let} $B_{H}$ denote the $H$-completion functor which replaces the original bornology of a space by the bornology generated by $HB$ for all originally bounded subsets $B$.
 {For a $G$-bornological coarse space $Y$, we denote by $Y_{max-\cB}$ the same coarse space equipped with the maximal bornology.}
In the lemma below, the group $H$ acts on $G\times X$ by $h(g,x):=(gh^{-1},hx)$.

\begin{lem}
	\label{lem:coequa}
	The following is a coequalizer in $G\BC$:
	\[(G_{min,min}\otimes H_{min,min}\otimes X)_{max-\cB}\rightrightarrows B_{H}(G_{min,min}\otimes X)\to \Ind_H^G(X)\ ,\]
	where the first two maps are given by $(g,h,x)\mapsto (gh,x)$ and $(g,h,x)\mapsto (g,hx)$ respectively.
\end{lem}

\begin{proof} This is \cite[Rem.~6.6]{equicoarse}. 
\end{proof}

Let $Y$ be a $G$-bornological coarse space and let $X$ be an $H$-bornological coarse space.
\begin{lem}\label{ergijeogegregrgregg}
	\label{lem:frobenius}
	We have an isomorphism
	\begin{equation}\label{43foijiiof43f34f34f}
\Ind_H^G(\Res_H^G(Y)\otimes X) \cong
Y\otimes \Ind_H^G(X)\ ,
\end{equation}
	which is natural in $Y$ and $X$. 
\end{lem}

\begin{proof}
	Consider the $G$-bornological coarse spaces $((G\times H)_{min,min}\otimes Y\otimes X)_{max-\cB}$ and $B_{H}(G_{min,min}\otimes Y\otimes X)$ where $G$ acts on the first factor, and $(Y\otimes ( G\times H)_{min,min}\otimes X)_{max-\cB}$ and $B_{H}(Y\otimes G_{min,min}\otimes X)$ where $G$ acts now diagonally on the first two factors. The isomorphisms \[((G\times H)_{min,min}\otimes Y\otimes X)_{max-\cB}\to (Y\otimes ( G\times H)_{min,min}\otimes X)_{max-\cB}\] given by $(g,h,s,x)\mapsto (ghs,g,h,x)$ and \[B_{H}(G_{min,min}\otimes Y\otimes X)\to B_{H}(Y\otimes G_{min,min}\otimes X)\] given by $(g,s,x)\mapsto (gs,g,x)$ induce an isomorphism of the coequalizer diagrams for $\Ind_H^G(\Res_H^G(Y)\otimes X)$ and $Y\otimes \Ind_H^G(X)$ from \cref{lem:coequa}.
In the case of $Y\otimes \Ind_H^G(X)$ we implicitly use the facts (which can both be checked in a straightforward manner) that the functor 
$Y\otimes- \colon G\BC\to G\BC$ preserves colimits of colim-admissible diagrams in $G\BC$ in the sense of \cite[Def. 2.20]{equicoarse}, and that the  coequalizer diagram in  \cref{lem:coequa} is 
colim-admissible.
\end{proof}

The equivalence from \cref{lem:frobenius} extends to equivariant  coarse motivic  spectra in the $Y$-variable.
Thus let  $Y$ be in $G\Sp\cX$  and let $X$ be as before.

\begin{kor}\label{43foijiiof43f34f34f1111} We have an equivalence
\begin{equation}
\Ind_H^{G,Mot}(\Res_H^{G,Mot}(Y)\otimes X) \cong  
Y\otimes \Ind_H^G(X)\ ,
\end{equation}
	which is natural in $Y$ and $X$. 
\end{kor}
\begin{proof}
This follows from  \cref{lem:frobenius} and  the fact that the operations
$\Ind_{H}^{G}$, $\Res_{H}^{G}$ and $-\otimes X$ all descend from $G\BC$ to $G\Sp\cX$.
\end{proof}

\begin{rem}\label{giu9o34ergergegrege}
We have versions of \cref{ergijeogegregrgregg} for 
\begin{enumerate}
\item $Y$ a $G$-coarse space, $X$ a $H$-coarse space, and the isomorphism \eqref{43foijiiof43f34f34f} for $G$-coarse spaces, and
\item $Y$ a $G$-set, $X$ a $H$-set, and the isomorphism \eqref{43foijiiof43f34f34f} for $G$-sets, 
\end{enumerate}
with the same isomorphism on the level of underlying sets.
\end{rem}

Let $Y$ be an $H$-invariant subset of a $G$-coarse space $X$. We consider
$Y$ as an $H$-coarse space with the structures induced from $X$.
For every coarse entourage $U$ of $X$ we define the coarse entourage $U_{Y}:=(Y\times Y)\cap U$ of $Y$.

\begin{lem}\label{3g89234ereger}
	The set of entourages $\{U_{Y}\:|\: U\in \cC^{G}(X)\}$ is cofinal in $\cC^{H}(Y)$.
\end{lem}
\begin{proof}
By definition of $\cC(Y)$, the set  $\{U_{Y}\:|\: U\in \cC(X)\}$ is cofinal in (actually equal to) $\cC(Y)$. Since $\cC^{G}(X)$ is cofinal in $\cC(X)$ (since $\cC(X)$ is a $G$-coarse structure), it then follows that   $\{U_{Y}\:|\: U\in \cC^{G}(X)\}$ is cofinal in $\cC^{H}(Y)$.
\end{proof}

\begin{lem}
	\label{lem:comparison-right}
	The  inclusion $H_{can,min}\to \Res_H^G(G_{can,min})$ induces an equivalence 
	\[F_{\fin}^{0}(\Ind_H^G(H_{can,min}))\to F_{\fin}^{0}(\Ind_H^G(\Res_H^G(G_{can,min})))\]
\end{lem}
\begin{proof}
Note that $\Ind_H^G(H_{can,min})$ is $G$-finite so that we can omit the index $\fin$ on the domain of the morphism. 
It suffices to show that the inclusion of $\Ind_H^G(H_{can,min} )$ into any $G$-invariant $G$-finite subset of $\Ind_H^G(\Res_H^G(G_{can,min}))$ induces an equivalence after applying $F^{0}$.
We now observe that $G$-finite subsets of $\Ind_H^G(\Res_H^G(G ))$ correspond to $H$-finite subsets of
$G $. We furthermore use that
$\Ind_{H}^{G}$ commutes with $F^{0}$ by  \cref{gregee4erg}.
It then remains to show that for every $H$-invariant and $H$-finite subset $L$ of $ G$ containing $H $ the inclusion $i \colon H\to L$  induces an equivalence 
$F^{0}(H_{can,min})\to F^{0}(L_{G_{can,min}})$.

Let $L$ be an $H$-invariant and $H$-finite subset of $ G$ containing $H$. We choose an $H$-equivariant  left-inverse   $s\colon L\to H$   of the inclusion $i$.
For every orbit $R$ in    $H\backslash L$ we pick a point $l_R$ in the orbit $R$. Since $H\backslash L$ is finite, the subset    $V:=H\{(s(l_R),l_R) \mid R\in H\backslash L \}$ of $G\times G$ belongs to the coarse structure $ \cC(G_{can,min})$. We set
$\cC^{\prime}:=\{  U\in \cC^{G}(G_{can,min}) \mid V\subseteq U\}$,
$U_{L}:=(L\times L)\cap U$. By \cref{3g89234ereger},
 the set $\{U_{L} \mid U\in \cC^{\prime} \}$
 is cofinal in $\cC^H(L)$. In view of \cref{rgoipewerfwefewfew} applied to $E=\Yo^{s}\circ \cF$  and \cref{gioowegfwefwfwef} of $F^{0}$, it therefore suffices to show that the morphism
 \begin{equation}\label{hoijoh5h45h45h}
\Yo^{s}(P_{U_H}(H_{can,min})_{d,b})\to \Yo^{s}(P_{U_{L}}(L_{G_{can,min}})_{d,b})
\end{equation} induced by $i$
is an equivalence for every entourage $U$ in $\cC^{\prime}$.

Since $L$ is $H$-finite, the map  $s$ is automatically a morphism $L_{G_{can,min}}\to H_{can,min}$. We argue that  
  the morphism
	\[ \Yo^{s}(P_{U_{L}}(L_{G_{can,min}})_{d,b})\to \Yo^{s}(P_{U_H}(H_{can,min})_{d,b})\]
	induced by $s$ is an inverse to \eqref{hoijoh5h45h45h}.
	
The composition $s\circ i$ is the identity. By definition of $U$, the composition \[P_{U_{L}}(L_{G_{can,min}})_{d,b}\stackrel{s}{\to}  P_{U_H}(H_{can,min})_{d,b}\stackrel{i}{\to} P_{U_{L}}(L_{G_{can,min}})_{d,b}\]
  has distance at most 1 from the identity. Since $\Yo^{s}$ is coarsely invariant, its sends this composition to a morphism which is equivalent to the identity. 
   This finishes the proof.
\end{proof}

In the following, we indicate by a subscript $G$ or $H$ for which group the Rips complex functor is considered.
\begin{lem}\label{ergi0erogregergg}
We have an equivalence of functors from $G\BC$ to $H\Top[W_{H}^{-1}]$
\[ \Res_{H}^{G,htop}\circ \Rips_{G}\cong \Rips_{H}\circ \Res_{H}^{G}\ . \]
\end{lem}

\begin{proof}
This immediately follows from the obvious isomorphism
\[\Res_{H}^{G,top}(P_{U}(X))\cong P_{U}( \Res_{H}^{G}(X))\]
for every $X$ in $G\BC$ and $U$ in $\cC^{G}(X)$, 
 \cref{ergieogergergeeggreg}, the equivalence
 \[(\ell\circ \Res_{H}^{G,top})_{|G\Simpl}\simeq (\Res_{H}^{G,htop}\circ\ell)_{|G\Simpl}\ ,\] 
 and the observation that $\cC^{G}(X)$ is cofinal in 
 $\cC^{H}(\Res^{G}_{H}(X))$.
\end{proof}

\begin{lem}\label{ergioergege}
We have an equivalence of functors from $H\BC$ to $G\Top[W_{G}^{-1}]$
\[ \Ind_{H}^{G,htop}\circ \Rips_{H}\cong \Rips_{G}\circ  \Ind_{H}^{G}\ .\]
\end{lem}
\begin{proof}
For every $X$ in $H\BC$ and $U$ in $\cC^{H}(X)$ we have {by \cref{gioergregregreg}} a natural isomorphism 
\[ \Ind_{H}^{G,top}(P_{U}(X)) \cong P_{\Ind_{H}^{G}(U)}(\Ind_{H}^{G}(X))\ . \]
We now apply  \cref{ergieogergergeeggreg}, the equivalence
\[ (\ell\circ \Ind_{H}^{G,top})_{|H\Simpl}\simeq (\Ind_{H}^{G,htop}\circ\ell)_{|H\Simpl}\ ,\] 
 and the observation that the induction map $\Ind_{H}^{G} \colon \cC^{H}(X)\to  \cC^{G}( \Ind_{H}^{G}(X))$ on the level of posets of entourages
is cofinal.
\end{proof}

\begin{lem}
\label{lem:comparison-left}
	The inclusion $H_{can,min}\to \Res_H^G(G_{can,min})$ induces a continuous equivalence
	\[F_{\fin}^{\infty}(\Ind_H^G(H_{can,min}))\otimes G_{max,max}\to F_{\fin}^{\infty}(\Ind_H^G(\Res_H^G(G_{can,min})))\otimes G_{max,max}\ .\]
\end{lem}

\begin{proof}
	By \cref{prop:a-b-comp} (using only the continuous equivalence of the domains), the map is continuously equivalent to 
	\[\tilde \cO^\infty_{\hlg}(\Rips_{G}(\Ind_H^G(H_{can,min})))\otimes G_{can,min}\to \tilde \cO^\infty_{\hlg}(\Rips_{G}(\Ind_H^G(\Res_H^G(G_{can,min}))))\otimes G_{can,min}\ .\]
	 Using  \cref{ergi0erogregergg} and \cref{ergioergege},
	we see that
	this map is equivalent to
	\[\tilde \cO^\infty_{\hlg}(\Ind_H^{G,htop}(\Rips_{H}(H_{can,min}))) \to \tilde \cO^\infty_{\hlg}(\Ind_H^{G,htop}(\Res_H^{G,htop}(\Rips_{G}(G_{can,min})))) \]
	twisted by $G_{can,min}$.
The latter map is an equivalence since the map \[\Rips_{H}(H_{can,min})\to  \Res_H^{G,htop}(\Rips_{G}(G_{can,min}))\] induced by the inclusion of $H$ into $G$ is mapped by the equivalence $\overline{\mathrm{Fix}}$
	to the  essentially unique  equivalence $E_{\Fin}H\simeq \Res_{H}^{G}(E_{\Fin}G)${; see} \cref{rviuohiofe2df2edf2}. 
\end{proof}

\section{The main theorem}\label{rgiofergegergr43534543534535}

The main result of the present section is \cref{thm:novikov1-neu-relativ}. 
Before giving its proof, we will show how to deduce \cref{thm:main-injectivity-kor1} from 
 \cref{thm:novikov1-neu-relativ}.
 
 The structure of the proof of  \cref{thm:novikov1-neu-relativ} is as follows:
 \begin{enumerate}
 \item \cref{lem:firststep} reduces the proof to the verification that a certain morphism  $\bL(S)\to \bM_{A}(S)$  is an equivalence for every
 $S$ in $G_{\cF}\Orb$.
 \item \cref{prop:comparison2} identifies this morphism with the composition  
  of a descent morphism and  a  forget-contol map 
depending on   subgroups $H$ in the family~$\cF$.
  \item In \cref{thm:leftinverse-rel}, we use the descent  {result} to show that
  the descent morphism is an equivalence, and therefore  reduce the problem to the verification that forget-control maps are equivalences for subgroups $H$ in the family $\cF$. This step employs transfers. 
\item In \cref{thm:replacement2} we use the geometric assumptions on the subgroups $H$ in order to deduce from \cite{trans}  that the 
 forget-control maps in the $H$-equivariant context are equivalences.
 \end{enumerate}

Let $G$ be a group and let $M\colon G\Orb\to \bC$ be a functor. Let $A$ be in $\PSh(G\Set)$ and let $\cF$ be a family of subgroups.  
\begin{theorem}
	\label{thm:novikov1-neu-relativ}
	Assume that:
	\begin{enumerate}
		\item $M$ is a CP-functor (see \cref{bioregrvdfb});
		\item \label{it:nov2} $r^*A$ is equivalent to $E_\Fin G$ in $\PSh(G\Orb)$ (see \eqref{efkjbnkjffrefw} for the definition of $r^{*}$);
		\item \label{it:nov3} for all $H$ in $\cF$ the object $\Res_H^G(A)$ of $\PSh(H\Set)$ is compact;
		\item $\cF$ is a subfamily of $\FDC$ (see \cref{grregergerg}.\ref{giorgjerogiergreg}) such that $\Fin\subseteq \cF$.
	\end{enumerate}
	Then the relative assembly map $\As_{\Fin,M}^\cF$ (\cref{ergoiegererg}) admits a left inverse.
\end{theorem}

Before we begin with the proof, we will first deduce \cref{thm:main-injectivity-kor1} from \cref{thm:novikov1-neu-relativ}.  

\begin{rem}\label{griog34tt43t3t}
For every group $K$ the functor $r_{!} \colon \PSh(K\Orb)\to \PSh(K\Set)$ induced by $r \colon K\Orb\to K\Set$ (see \eqref{efkjbnkjffrefw}) preserves  {compacts} since it has a right-adjoint $r^{*}$ which preserves all colimits. 

We claim that for any subgroup $H$ of $K$ the functor
\[ \Res^{K}_{H} \colon \PSh(K\Set)\to \PSh(H\Set) \]
preserves compacts. The claim
 follows from the fact that $\Res^{K}_{H}$ preserves representables and colimits. Here are some more details: 
 For any $S$ in $K\Set$ the restriction $\Res^{K}_{H}(\yo(S))$ is represented by the $H$-set $\Res^{K}_{H}(S)$.  It follows that
$\Res^{K}_{H}(\yo(S))$ is representable again. We now use that a compact object $A$ in
$\PSh(K)$ is a retract of a finite colimit of representables.
Since $\Res^{K}_{H}$ preserves colimits, we conclude that $\Res^{K}_{H}(A)$ is again a retract of a finite
colimit of representables. 

In the following, we write $r^{K}_{!}$ and $r^{H}_{!}$  for the corresponding functors for subgroups  $K$ and $H$ of $G$.
We then have a commuting diagram
\begin{equation}\label{gioh4joi4jogg334g34g}
\xymatrix{\PSh(K\Orb)\ar[r]^{r^{K}_{!}}\ar[d]^{\Res^{K}_{H}}&\PSh(K\Set)\ar[d]^{\Res^{K}_{H}}\\ \PSh(H\Orb)\ar[r]^{r^{H}_{!}}&\PSh(H\Set)}
\end{equation}
  \end{rem}

\begin{lem}
	\label{lem:compact}
	For every subgroup $H$ of $G$ in the family $\cp$ (see \cref{grregergerg}.\ref{giorgjerogiergreg1}), the object $\Res_H^G(r_!E_\Fin G)$ of $\PSh(H\Set)$ is compact.
\end{lem}
\begin{proof} 
	Let $H$ in $\cp$ be given. Then there exists a subgroup $H'$ of $G$ containing $H$ such that $E_\Fin H'$ is compact.  Using \eqref{gioh4joi4jogg334g34g}  and obvious relations between various restriction functors, we obtain the equivalences
\begin{align*}
\mathclap{
\Res_H^G(r_!E_\Fin G)\simeq \Res_{H}^{H'}\Res_{H'}^G(r_!E_\Fin G)\simeq \Res_H^{H'}(r_!^{H'}(\Res_{H'}^G(E_\Fin G)))\simeq \Res_{H}^{H'}(r^{H'}_!(E_\Fin H'))\ .
}
\end{align*}
Since $\Res^{H^{\prime}}_{H}$ and $r_{!}^{H'}$ preserve compacts  by \cref{griog34tt43t3t}, this implies that  $\Res^{G}_{H}(r_{!}E_{\Fin}G)$
	is compact as claimed.
\end{proof}

Recall from \cref{efweu9u9wefwef} that  $E^{top}_{\cF}G$ denotes  a $G$-CW complex modeling the classifying space of the family $\cF$. 
Let $\ell \colon G\Top\to G\Top[W^{-1}_{G}]$ be the localization{; see} \cref{efweu9u9wefwef}.
\begin{lem}
	\label{lem:fin-dim-model}
Assume  that  there exists a finite diagram $S \colon I\to G_{\cF}\Set$ such that
\[\ell(E^{top}_{\cF}G)\simeq \colim_{I}\ell(S_{disc}).\footnote{In classical terms this assumption is equivalent to the assumption that
$\hocolim_{I}S_{disc}$ has the homotopy type of $E^{top}_{\cF}G$.}\]
Then there exists a compact object $A$ in $\PSh(G\Set)$ such that $r^*A$ is equivalent to $E_\cF G$ (see \eqref{efkjbnkjffrefw} for the definition of $r^{*}$).

In particular, such an $A$ exists if one can represent $E^{top}_{\cF}G$ by a finite-dimensional $G$-CW complex.
\end{lem}
\begin{proof}
In analogy to the functor  $\mathrm{Fix}$ from  \eqref{h56h5h56h5h56h56h},
we define the functor
	\[ \mathrm{ \tilde{Fix}}\colon G\Top \to \PSh(G\Set)\ , \quad X\mapsto   \ell( \Map_G((-)_{disc},X)) \ .\]
	 We then note that $r^{*}\circ \mathrm{\tilde{Fix}}\simeq \mathrm{Fix} \simeq \ff \circ \ell$.

 Since $r^{*}$ and $\ff$ preserve colimits, 
	\begin{align*}
	 E_\cF G &\simeq \ff(\ell(E^{top}_{\cF}G)) \simeq \ff(\colim_{I} \ell(S_{disc})) \simeq  \colim_{I}  \ff(\ell(S_{disc})) \\
	 &\simeq \colim_{I} r^{*} \mathrm{ \tilde{Fix}}(S_{disc})\simeq r^{*} \colim_{ I}   \mathrm{ \tilde{Fix}}   (S_{disc})\ .
	\end{align*}
 By definition we have an identification  $ \mathrm{ \tilde{Fix}}(S_{disc})\simeq \yo(S)$.  {It follows that} if we define $A:=\colim_{ I}\yo(S)$, then $A$ is a compact object of $\PSh(G\Set)$ with $r^{*}A\simeq E_{\cF}G$.
 
 The last assertion of the lemma follows from the more general claim that for every finite-dimensional $G$-CW-complex $X$ with stabilizers in $\cF$ there exists a finite diagram $S_{X} \colon I_{X} \to G_{\cF}\Set$ such that $\ell(X) \simeq \colim_{I_{X}} \ell(S_{X,disc})$.
 
  Given such a $G$-CW-complex $X$, there exists a finite-dimensional $G$-simplicial complex $K$ with stabilizers in $\cF$ which is equivariantly homotopy equivalent to $X$ (this works as in the non-equivariant case which can for example be found in \cite[Thm.~2C.5]{hatcher}). After one barycentric subdivision, we may assume that $K$ is locally ordered. Then we may regard $K$ as a diagram $S \colon \Delta^{\leq \dim(K)}_{inj} \to G_\cF\Set$, that is as a finite-dimensional semi-simplicial $G$-set with stabilizers in $\cF$. The homotopy colimit over this finite diagram is equivalent to the barycentric subdivision of $K$; this can be verified explicitly using the Bousfield--Kan formula for the homotopy colimit \cite[Ch.~XII.2]{BK}. Consequently, $\colim_{ \Delta^{\leq \dim(K)}_{inj}}\ell(S_{disc}) \simeq \ell(X)$.
\end{proof}

\begin{rem}\label{rgeriowegewfewfrefwefew}
The argument for \cref{lem:fin-dim-model} shows that  if there exists a finite $G$-CW-model $E^{top}_{\cF}G$, then one can choose $A$ in $\PSh(G\Set)$  such that it is given as a colimit of a finite diagram with values in $G$-finite $G$-sets with stabilizers in $\cF$.
\end{rem}

\begin{proof}[Proof of \cref{thm:main-injectivity-kor1}]
	\cref{thm:main-injectivity-kor1} is a special case of \cref{thm:novikov1-neu-relativ}, where under the Assumption \ref{thm:main-injectivity-kor1}.\ref{thm:main-it1} we can use $r_!E_\Fin G$ for $A$ by \cref{lem:compact}. Here we use that $r^{*}r_{!}\simeq \id$ since $r$  in \eqref{qwelfjqwoifqfeefqefqwef} is fully faithful.
		 Under Assumption  \ref{thm:main-injectivity-kor1}.\ref{thm:main-it2}, we use \cref{lem:fin-dim-model} and that $\Res_H^G$ preserves compacts by \cref{griog34tt43t3t}.
\end{proof}

We now prepare the proof of \cref{thm:novikov1-neu-relativ}. 

Recall \cref{gfioweffwefwefwefwef} of the $\infty$-category of $G$-bornological coarse spaces with transfers and the inclusion functor \eqref{verv3r3oijfoij3f3f}.
Let $H$ be a subgroup of $G$.
\begin{lem}\label{lem:ind-tr}
The induction functor \eqref{gioreggwefewffewfwfwfwf} extends to a functor
\begin{equation*}\label{rvrvkjioerververververv}
\Ind_{H}^{G,tr} \colon H\BC_{tr} \to G\BC_{tr}
\end{equation*}
 such that \begin{equation*}\label{vtoijgoijrogif3f4f3f3}
\xymatrix{H\BC\ar[r]^{\Ind_{H}^{G}}\ar[d]^{\iota_{H}}&G\BC\ar[d]^{\iota_{G}}\\H\BC_{tr}\ar[r]^{\Ind^{G,tr}_{H}}&G\BC_{tr}}
\end{equation*}
 commutes.
\end{lem}
\begin{proof}
Recall from  \cref{gfioweffwefwefwefwef} that $H\BC_{tr}$ and $ G\BC_{tr}$ are built from certain spans whose vertices belong to $G\BC$ and whose morphisms are controlled. We apply the functor $\Ind_{H}^{G}$ (for bornological coarse spaces) to the
vertices and obtain the maps from the version of the induction for the underlying $H$-sets.
We must show that this construction preserves the conditions on the morphisms for the simplices of $ H\BC_{tr}$ and $G\BC_{tr}$ as specified in \cite[Def.~2.27]{coarsetrans}. In particular, this amounts to showing that induction preserves morphisms in  {$G\BC$}, bounded coverings, and cartesian squares in $G\Coarse$.
 
We have seen {in \cref{sec:cones}} that induction preserves controlled and proper morphisms.
Next we discuss bounded coverings.

Let $X,Y$ be in $H\BC$ and let
$f \colon X\to Y$ be an $H$-equivariant bounded covering.
Then we must show that $ \Ind_{H}^{G}(f) \colon \Ind_{H}^{G}(X)\to  \Ind_{H}^{G}(Y)$ is again a bounded covering. We verify the properties listed in \cref{rgeeiorjgergergreg}.
\begin{enumerate}
\item The coarse structure of $\Ind_H^G(X)$ is generated by the images $\Ind_{H}^{G}(U)$ of $\diag(G)\times U$ in $\Ind_{H}^{G}(X)$  for $U$ an entourage of $X$.  We now observe that  
\[ ( \Ind_{H}^{G}(f)\times \Ind_{H}^{G}(f))  (\Ind_{H}^{G}(U))=\Ind_{H}^{G}((f\times f)(U))\ ,\]
which is an entourage of $\Ind_{H}^{G}(Y)$ by definition. This shows that
$\Ind_{H}^{G}(f)$ is controlled.
\item We write $U_{\pi_{0}(X)}:=\bigcup_{U\in \cC(X)}U $. 
Then we have the equality \[U_{\pi_{0}(\Ind_{H}^{G}(X))}=\bigcup_{U\in \cC(X)}\Ind_{H}^{G}(U)=\Ind_{H}^{G}(U_{\pi_{0}(X)})\ .\]
For $U$ in $\cC(Y)$ we furthermore have
\[ (\Ind_{H}^{G}(f) \times \Ind_{H}^{G}(f))^{-1}(\Ind_{H}^{G}(U))\cap  U_{\pi_{0}(\Ind_{H}^{G}(X)) }
=\Ind_{H}^{G}((f\times f)^{-1}(U)\cap U_{\pi_{0}(X)})\ .\]
Since $f$ is a bounded covering,
this shows that $\cC(\Ind_{H}^{G}(X))$ is generated by the entourages of the form
$(\Ind_{H}^{G}(f)\times \Ind_{H}^{G}(f))^{-1}(U)\cap U_{\pi_{0}(\Ind_{H}^{G}(X))}$
for all $U$ in $\cC(\Ind_{H}^{G}(Y))$.
\item \label{qewfqwefefqf} We consider a class $\{g,x\}$ in $\Ind_{H}^{G}(X)$ (we use $\{\ ,\ \}$ to denote  $H$-orbits in $G\times X$ since we want to reserve $[-]$ for coarse components). Its coarse component is then given by \[[\{g,x\}]=\{\{g',x'\}\:|\: (\exists h\in H\:|\:  g'h=g\ ,h^{-1}x'\in [x])\}\ .\]
It follows that  {the map $[\{g,x\}]\to [x]$ sending 
$\{g',x'\}$ in $[\{g,x\}]$ to $g^{-1}g' x'$  in $[x]$ is a bijection which identifies $\Ind_{H}^{G}(f)_{|[\{g,x\}]}$ with $f_{|[x]}$. Therefore, $\Ind_{H}^{G}(f)_{|[\{g,x\}]}$ is an isomorphism of coarse spaces.}
\item \label{eroigjpwergreerggw} Let $g$ be in $G$ and $B$ be bounded in $X$. Then the image $B_{g}$  of $\{g\}\times B$ in $\Ind_{H}^{G}(X)$  is a bounded subset of $\Ind_{H}^{G}(X)$ by definition of the bornology.
Since  $\Ind_{H}^{G}(f)(B_{g})=f(B)_{g}$, its image under $\Ind_{H}^{G}(f)$ is also a bounded subset of $\Ind_{H}^{G}(Y)$. This implies that $\Ind_{H}^{G}(f)$ is bornological. 
\item  {We have to show that for} every bounded subset $B$ of $\Ind_{H}^{G}(X)$ the cardinality of the fibres of the induced map 
$\pi_{0}(B)\to \pi_{0}(\Ind_{H}^{G}(X))$ has a finite bound (which may depend on $B$).
Let $\{g,y\}$ be in $ \Ind_{H}^{G}(Y)$ and consider the component $[\{g,y\}]$ in $\pi_{0}(\Ind_{H}^{G}(Y))$. Then, as seen in \ref{qewfqwefefqf}, we have 
 $[\{g,x\}]\in \pi_{0}(\Ind_{H}^{G}(f))^{-1}[\{g,y\}]$ if and only if $[x] \in \pi_{0}(f)^{-1}([y])$. 
 If $[\{g,x\}]\cap B_{g}\not=\emptyset$, then in addition $[x]\cap  B\not=\emptyset$. Since $f$ is a bounded covering, there is a finite bound on the cardinality of  the sets
 $\{  [x]\in \pi_{0}(X)\:|\:   \pi_{0}(f)([x])=[y]\ , [x]\cap B\not=\emptyset\}$. 
  \end{enumerate}

The argument for \ref{eroigjpwergreerggw} shows that induction preserves 
bornological maps.

We finally show that induction preserves cartesian squares in $G\Coarse$.
Let \[\xymatrix{X\ar[r]\ar[d]&Y\ar[d]^{\phi}\\Z\ar[r]^{\psi}&W}\]
be a cartesian square in  $G\Coarse$.
We first show that 
\[\xymatrix{\Ind_{H}^{G}(X)\ar[r]\ar[d]&\Ind_{H}^{G}(Y)\ar[d]\\\Ind_{H}^{G}(Z)\ar[r]&\Ind_{H}^{G}(W)}\]
is a cartesian square on the level of underlying $G$-sets.
Indeed, $\Ind_{H}^{G}(X)$ is the subset of elements $(\{g,y\},\{g',z\})$ in $\Ind_{H}^{G}(Y)\times \Ind_{H}^{G}(Z)$ such that there exists $h$ in $H$ with  {$g=g'h$} and $\phi(y)=\psi(h^{-1}z)$.
This is {in bijection to} the set of elements
$\{g,(y,z)\}$  in $\Ind_H^G(X)$, where we consider $X$ as a subset of $Y\times Z$. 
Let $U$ and $V$ be entourages of $Z$ and $Y$, respectively. 
Then we have the equality \[(\Ind_{H}^{G}(U)\times \Ind_{H}^{G}(V))\cap (\Ind_{H}^{G}(X)\times \Ind_{H}^{G}(X))=
\Ind_{H}^{G}((U\times V)\cap (X\times X))\ .\]
The entourages on the left generate the coarse structure on 
$\Ind_{H}^{G}(X)$ such that the square above is cartesian in $G\Coarse$.
The entourages on the right  generate the induced coarse structure on $ \Ind_{H}^{G}(X)$.
Hence both structures coincide.
\end{proof}
 
Recall the construction of the functor $m$ from \eqref{g354oijot343g3}. In the following, we put an index $G$ or $H$ in order to indicate the respective group.
 
\begin{lem}\label{grierogerger3t4t}
We have a commuting square
\[\xymatrix@C=4em{
 G\Set^{op}\times H\BC\ar[r]^{\id\times \Ind_{H}^{G}}\ar[d]^{\Res^{G}_{H}\times \id}& G\Set^{op}\times G\BC \ar[dd]^{m_{G}}\\
 H\Set^{op}\times H\BC\ar[d]^{m_{H}}&\\H\BC_{tr}\ar[r]^{\Ind_{H}^{G,tr}}&G\BC_{tr} }\]
 {in $\Cat_\infty$.}
\end{lem}
 \begin{proof}
 We freely use the notation that was used in the definition of $m$.
 {Recall that the effective Burnside category $A^{eff}$ is defined for every category with pullbacks \cite[Def.~3.6]{Barwick:2014aa}, and that $A^{eff}$ is functorial with respect to pullback-preserving functors \cite[3.5]{Barwick:2014aa}. 
 Therefore, the proof of \cref{lem:ind-tr} shows that $\Ind_H^G$ induces a functor
 \[ \Ind_H^{G,eff} \colon A^{eff}(\tilde{H\BC}) \to A^{eff}(\tilde{H\BC})\ .\]
 Then} we can use \cref{giu9o34ergergegrege} to obtain a natural equivalence
 \[ \Ind_H^{G,eff} \circ \tilde{m}_H \circ (\Res_H^G \times \id) \simeq \tilde{m}_G \circ (\id \times \Ind_H^G)\ .\]
 Hence it suffices to show that the endofunctor $P$ from the definition of $m$ is compatible with $\Ind_H^{G,eff}$ in the sense that $P_G \circ \Ind_H^{G,eff} \simeq \Ind_H^{G,tr} \circ P_H$. This is clear since the application of $P$ amounts to pulling back certain bornologies, and the isomorphism in  \cref{lem:frobenius} is compatible with this operation on bornologies. 
     \end{proof}

For the rest of the section we fix a CP-functor $M\colon G\Orb\to \bC$. According to  \cref{bioregrvdfb}, there is a $\bC$-valued  strongly additive and continuous equivariant coarse homology theory $E$ with transfers (see \cref{rgiorhgiuregergergergergerg}) such that
\begin{equation}\label{regoihihihjriu34hr34r3}
M\simeq (E\circ \iota)_{G_{can,min}}\circ i\ .
\end{equation}

Using the functor {$\Ind_H^{G,tr}$ from \cref{lem:ind-tr}},  we can define the composition
\begin{equation}\label{5gpo3jkopjk34opg3ogpg34g34g34}
E^{H}:=E\circ \Ind_{H}^{G,tr} \colon H\BC_{tr}\to  \bC\ .
\end{equation}
Because of \cref{lem:ind-mot} and \cref{lem:ind-tr}, the functor $E^{H}$ is again a $\bC$-valued coarse homology theory with transfers. 
Applying   \cref{fwpeou2903rf} to $E^{H}$, we obtain a functor
\[\wt E^{H}\colon\PSh(H\Set)^{op}\times H\Sp\cX\to \bC\ .\]
We will consider $\wt E^{H}$ as a contravariant functor in its first argument sending colimits to limits.

 The following lemma clarifies the relation between $\wt E^{H}$ and $\wt E$.

\begin{lem}
	\label{lem:induction}
	For every subgroup $H$ of $G$ there is an equivalence
	\[\wt E(-,\Ind_H^{G {,Mot}}(-))\simeq \wt E^{H}(\Res_H^G(-),-)\]
	of functors $\PSh(G\Set)^{op}\times H\Sp\cX\to \bC$.
\end{lem}

\begin{proof}
Recall \cref{jfi3rhfiuofewfewf} of $\underline{E}$ and $\underline{E^{H}}$.
By the universal property of $\PSh(G\Set)$ and since both functors send colimits to limits 
in their first arguments (note that the functor $\Res_{H}^{G} \colon \PSh(G\Set)\to \PSh(H\Set)$ preserves colimits),
  it suffices to provide  an equivalence  
\[ \underline{E}(-,\Ind_H^{G {,Mot}}(-))\simeq \underline{E^{H}}(\Res^{G}_{H}(-),-) \]
of functors $G\Set^{op}\times H\Sp\cX\to \bC$. In view of the definitions of $\underline{E}$ and $\underline{E^{H}}$,
it is enough to provide an equivalence 
\[ m_{G}(-,\Ind_{H}^{G}(-))\simeq ( \Ind^{G,tr}_{H}\circ m_{H})(\Res^{G}_{H}(-),-) \]
of functors
\[ G\Set^{op}\times H\BC\to G\BC_{tr}\ .\]
This equivalence is exactly the assertion of   \cref{grierogerger3t4t}. 
\end{proof}

 {Recall from \cref{def:twist} what it means to twist an equivariant coarse homology theory by a $G$-bornological coarse space.}
For better readability we introduce the abbrevation
\begin{equation}\label{ewecvwevewcewcw}
E_{G}:=E_{G_{max,max}}
\end{equation}  for the twist of $E$ with $G_{max,max}$.
Note that $\wt E_{G}$ denotes the result of \cref{fwpeou2903rf} applied to $E_{G}$. 
We further abbreviate $E^{H}_{G}:=(E_{G})^{H}${; see} \eqref{5gpo3jkopjk34opg3ogpg34g34g34}. Note that the order of constructions matters. We first twist by $G_{max,max}$ and then precompose with the induction from $H$ to $G$.

Since $E$ is strongly additive and extends to a coarse homology theory with transfers, also $E_{G}$ is strongly additive and extends to a coarse homology theory with transfers by \cite[Lem.~3.13]{equicoarse} and \cite[Ex.~2.57]{coarsetrans}.   {Recall the definition of $\iota\colon H\BC\to H\BC_{tr}$ from \eqref{verv3r3oijfoij3f3f}. Then $E^{H}_{G}\circ \iota$ is an $H$-equivariant coarse homology theory, and hence extends to a functor $H\Sp\cX\to \bC$ which we again denote by $E^{H}_{G}\iota$. In this way, the morphism \eqref{revkjenekjvnrvevvverveer} in \cref{thm:leftinverse-rel}.\ref{it:descent3-rel} below is well-defined.}

Let $H$ be a subgroup of $G$. 
  The map \eqref{ghuhgiurehggregreregrrergergerg} in the statement of the next theorem is induced by the projection  $\Res^{G}_{H}(A)\to *$  and the cone boundary $\partial \colon F^{\infty}(H_{can,min})\to \Sigma F^{0}(H_{can,min})${; see} \eqref{gioowegfwefwfwef}.
\begin{theorem}\label{thm:leftinverse-rel}
	We assume:
	\begin{enumerate}
		\item \label{ijgorggergerg-rel} There exists an object $A$ in $\PSh(G\Set)$ such that $r^*A$ is equivalent to $E_\Fin G$ in $\PSh(G\Orb)$ and $\Res_H^G(A)$ is compact in $\PSh(H\Set)$.
		\item\label{it:descent3-rel} For every $H$-set $S$ with finite stabilizers the forget-control map $\beta_{E_G^{H},S_{min,min}\otimes H_{can,min}}$
	\begin{equation}\label{revkjenekjvnrvevvverveer}
E_G^{H}\iota(F^\infty (S_{min,min}\otimes H_{can,min}))\to\Sigma E_G^{H}\iota(F^{0}(S_{min,min}\otimes H_{can,min}))
\end{equation} 
		is an equivalence.
	\end{enumerate}
	Then the map \begin{equation}\label{ghuhgiurehggregreregrrergergerg}
\wt E_G^{H}(*,F^\infty(H_{can,min}))\to \Sigma \wt E_G^{H} (\Res_H^G(A),F^0(H_{can,min}))
\end{equation}
	 is an equivalence. 
\end{theorem}
\begin{proof}By construction, the map \eqref{ghuhgiurehggregreregrrergergerg} is the composition
	\begin{align*}\label{giorg34243g3g4g3g}
	\wt E^{H}_G(*,F^\infty(H_{can,min}))&\xrightarrow{!}\wt E_G^{H}(\Res_H^G(A), F^\infty(H_{can,min}))\\&\xrightarrow{!!} \Sigma \wt E_G^{H}(\Res_H^G(A),F^0(H_{can,min}))\ .
	\end{align*}
	We will show that  {both morphisms are equivalences.}
	
	Since  $\Res_H^G(A)$ is compact, by  \cref{lem:homtheory}, \cref{gioowegfwefwfwef} and \cref{rgoipewerfwefewfew} 
	we
	see that the  morphism marked by $!$ is a colimit over $U$ in $\cC^{H}({H_{can,min}})$ of morphisms
	\begin{equation}\label{234ogfi34g434g4111-rel}
	\wt E_G^{H}(*,\cO^\infty (P_U(H_{can,min})_{d,d,b}))\to \wt E_G^{H}(\Res_H^G(A),\cO^\infty (P_U(H_{can,min})_{d,d,b}))\ .
	\end{equation}  Since $H_{can,min}$ is uniformly discrete, {the $H$-simplicial complex	$P_U(H_{can,min})$ belongs to $H_{\Fin(H)}\Simpl^{\fin}$ and the bornology on $P_U(H_{can,min})_{d,d,b}$ agrees with the one induced from the spherical path quasi-metric; see \cref{girjgorgergergreg}.}
	Note that $E^{G}_{H}$ is strongly additive since $E_{G}$ is so and, as one easily checks,  the induction $\Ind_{H}^{G}$ preserves free unions (see \cite[Ex.~2.16]{equicoarse}   for the notion of a free union). To conclude that \eqref{234ogfi34g434g4111-rel} is an equivalence, we apply \cref{gio3jo24f22f3} with
	\begin{enumerate}
		\item $E_{G}^{H}$ in place of $E$,
		\item $\Res^{G}_{H}(A)$ in place of $A$,
		\item $\Fin(H)$ in place of $\cF$,
		\item and using
		\begin{equation}\label{brtbbrbrvvrtv}
		r^{*}\Res^{G}_{H}(A)\simeq \Res^{G}_{H}(r^{*}A)\simeq  \Res^{G}_{H}(E_{\Fin}G)\simeq E_{\Fin(H)}H
		\end{equation} in order to verify Assumption \ref{regieorg43tr34tg34t34t} of \cref{gio3jo24f22f3}.
	\end{enumerate}
It follows that the morphism $!$  is an equivalence.
	
	We consider the morphism marked by $!!$. The  object $\Res_H^G(A)$  in $\PSh(H\Set)$ is 
	 equivalent to the colimit of some diagram obtained from $S \colon I \to H\Set$ by composing with the Yoneda embedding $H\Set\to \PSh(H\Set)$.
	
	We claim that $S(i)\in H_{\Fin(H)}\Set$ for every $i$ in $I$. If $i$ in $I$, then there exists a morphism $\yo(S(i))\to \Res_H^G(A)$. Hence we get a morphism
	$r^{*} \yo(S(i))\to r^{*}\Res_H^G(A)$. 
	Let $R$ be some $H$-orbit in $S(i)$. Because $r^*(\yo(r(R)))\simeq \yo(R)$, we get a morphism $\yo(R)\to r^{*}\Res_H^G(A)$, i.e., $(r^{*}\Res_H^G(A))(R)\neq\emptyset$.
	Because $r^*\Res_H^G(A)$ is equivalent to $E_{\Fin(H)}H$ by \eqref{brtbbrbrvvrtv}, we conclude that $R\in H_{\Fin(H)}\Orb$.
	Since $R$ was an arbitrary $H$-orbit in $S(i)$  this implies that $S(i) \in H_{\Fin}\Set$ as claimed.
	
	Since  equivalences are stable under
	limits, and since  $\wt E^{H}_{G}$ in its first argument  sends colimits to limits, in order to show that $!!$ is an equivalence it suffices to show that 
	the  forget-control map  
	\[\beta_{E^{H}_G(\yo(S),-),H_{can,min}}\colon \wt E_G^{H}(\yo(S),F^{\infty}(H_{can,min}))\to \Sigma \wt E_G^{H}(\yo(S),F^{0}(H_{can,min}))\] is an equivalence for every $S$ in $H_{\Fin(H)}\Set$. 
	Inserting the definition of $\wt E_G^{H}$, this morphism is equivalent to  the morphism \[
	E_G^{H}\iota(S_{min,min}\otimes F^{\infty}(H_{can,min}))\to \Sigma E_G^{H}\iota(S_{min,min}\otimes F^{0}(H_{can,min}))\ .\]  
	By \cref{gregee4ergtw}, this morphism can furthermore be identified with the morphism
	 \begin{equation*}\label{frvrveerververv}
 \beta_{E_G^{H},S_{min,min}\otimes H_{can,min}} \colon E_G^{H} {\iota}(F^{\infty}(S_{min,min}\otimes H_{can,min}))\to \Sigma E_G^{H} {\iota}(F^{0}(S_{min,min}\otimes H_{can,min}))
\end{equation*}
	which is an equivalence by Assumption \ref{it:descent3-rel}.
\end{proof}

\begin{rem}\label{rem_descent_concrete_S}
{Assume} $\Res_H^G(A)$ in $\PSh(H\Set)$ is a colimit of a diagram obtained from $S\colon I\to H\Set$ with values in $H$-finite $H$-sets by composing with the Yoneda embedding $H\Set\to \PSh(H\Set)$. Then, by inspection of the argument, it suffices to  require  Assumption~\ref{it:descent3-rel} of Theorem~\ref{thm:leftinverse-rel} only for $H$-finite $H$-sets $S$ with finite stabilizers.
\end{rem}

Recall the standing assumption that $M$ is a CP-functor and that $E$ is a strongly additive equivariant coarse homology theory satisfying \eqref{regoihihihjriu34hr34r3}.

\begin{theorem} 
	\label{thm:replacement2}
	If $H_{can}$ has $H_\Fin$-FDC, then
	Assumption \ref{it:descent3-rel} of \cref{thm:leftinverse-rel} is fullfilled.   
\end{theorem}
\begin{proof}
We apply \cite[Thm.~1.1]{trans}  with
 \begin{enumerate}
 \item the group $H$ in place of $G$,
 \item the $H$-bornological coarse space,
 $S_{min,min}\otimes H_{can,min}$ in place of $X$,
 \item the $\bC$-valued  $H$-equivariant coarse homology theory $E_{G}^{H}\circ \iota$ in place of $E$.
 \end{enumerate}
We can conclude that  Assumption \ref{it:descent3-rel}  {of \cref{thm:leftinverse-rel}} is fullfilled if the following conditions are satisfied:  
 \begin{enumerate}
 \item 
  $E_{G}^{H}\circ \iota$ is weakly additive,
  \item $E_{G}^{H}\circ \iota $    admits weak transfers, 
  \item $\bC$ is compactly generated, 
  \item $S_{min,min}\otimes H_{can,min}$ has $H$-FDC,
  \item  $H$ acts discontinuously on $S_{min,min}\otimes H_{can,min}$.
  \end{enumerate}
 
It follows from the assumption that $M$ is a CP-functor that $\bC$ is compactly generated.
Furthermore,  by the standing assumption, $E$ is a strongly additive coarse homology theory with transfers. As noticed above, then $E_{G}^{H}$ is also strongly additive and admits transfers.
 By \cite[Sec.~2.2]{trans} strong additivity implies weak additivity and by \cite[Lem.~2.59]{coarsetrans} the existence of transfers implies the existence of weak transfers.
 
 If $H_{can}$ has $H_\Fin$-FDC , then $S_{min,min}\otimes H_{can,min}$ has $H$-FDC by definition. And finally, $H$ acts discontinuously on $S_{min,min}\otimes H_{can,min}$ for every $S$ in {$G\Set$, in particular for every $S$ in $G_{\Fin}\Set$.}
\end{proof}

Let $A$ be in $\PSh(G\Set)$.
Recall the notation $(-)_{\fin}$ from  \cref{ergiuhergoegergergergre}.
	We define the following functors from $G\Orb$ to $\bC$:
	\begin{align}
	\bL&:=\wt E_{G} (*, F_{\fin}^{\infty}((-)_{min,min}\otimes G_{can,min}))\label{def:funtorsequence}\\
	\bM_*&:=\wt E_{G} (*,\Sigma F_{\fin}^{0}((-)_{min,min}\otimes G_{can,min}))\label{def:funtorsequence2}\\
	\bM_A&:=\wt E_{G} (A,\Sigma F_{\fin}^{0}((-)_{min,min}\otimes G_{can,min})) \label{def:funtorsequence3}
	\end{align}
The boundary of the cone sequence (see  \cref{gioowegfwefwfwef})  induces a transformation $\bL\to \bM_*$, and the map $A\to *$ induces a transformation $\bM_*\to \bM_A$.

\begin{prop}
	\label{lem:comparison}
	The transformation $\bL\to\bM_*$
	is equivalent to the transformation
	\[(E_{G_{can,min}}\iota)(\tilde \cO^\infty_\hlg(\ell(-)_{disc}\times \Rips(G_{can,min})))\to (E_{G_{can,min}}\iota) (\cO^\infty_\hlg((-)_{disc}))\]
	induced by  the projection $\Rips(G_{can,min})\to *$.
\end{prop}
\begin{proof}
	By definition of $\wt E_{G}$ (see \eqref{ewecvwevewcewcw} and  \cref{fwpeou2903rf}), the map $\bL\to \bM_*$ is equivalent to the map
	\begin{align*}
	\mathclap{
	E\iota(F_{\fin }^{\infty}((-)_{min,min}\otimes G_{can,min})\otimes G_{max,max})\to \Sigma E\iota(F_{\fin}^{0}((-)_{min,min}\otimes G_{can,min})\otimes G_{max,max})\ .
	}
	\end{align*}
	By the \cref{cor:comparison} and the assumption that $E\circ \iota$ is continuous (note that $G_{can,min}$ is $G$-proper, uniformly discrete and coarsely connected), this map is equivalent to the map
	\[E\iota(\cO^\infty_{\hlg}(\ell(-)_{disc}\times \Rips(G_{can,min})))\otimes G_{can,min})\to E\iota(\cO^\infty_{\hlg}((-)_{disc})\otimes G_{can,min})\]
	induced by the projection $\Rips(G_{can,min})\to *$.  Since twisting by $G_{can,min}$ commutes with precomposition by $\iota$, this is the map in the statement of the proposition.
\end{proof}

Let $A$ be in $\PSh(G\Set)$. Let $\cF$ be a family of subgroups of $G$ such that $\Fin\subseteq \cF$.
Recall \cref{ergoiegererg} of the relative assembly map.
\begin{prop}
	\label{lem:firststep}
	Assume that $\bL(S)\to \bM_A(S)$ is an equivalence for all $S$ in $G_\cF\Orb$.
	Then the relative assembly map
	$\As_{\Fin,M}^\cF$
	admits a left inverse.
\end{prop}

\begin{proof}
	 {Forming the colimit over $G_{\cF}\Orb$, the assumption implies} that the composition
	\[
	\colim_{S\in G_\cF\Orb}\bL(S)\to   \colim_{S\in G_\cF\Orb}\bM_*(S) \to \colim_{S\in G_\cF\Orb}\bM_A(S)\]
	is an equivalence. 
	Hence the first morphism
	\begin{equation}\label{eq:assmap}
	\colim_{S\in G_\cF\Orb}\bL(S)\to \colim_{S\in G_\cF\Orb}\bM_*(S)
	\end{equation}
	admits a left inverse.  {Since $\cC$ is stable, it suffices} to show that the morphism \eqref{eq:assmap} is equivalent to 
	the suspension of the relative assembly map $\As_{\Fin,M}^\cF$.
	
	By \cref{lem:comparison}, the map \eqref{eq:assmap} is equivalent to the map
	\begin{equation}\label{ererhierhgioergerg}
\colim_{S\in G_\cF\Orb}E_{G_{can,min}} \iota(\tilde \cO^\infty_\hlg (\ell(S_{disc})\times \Rips(G_{can,min})))\to\colim_{S\in G_\cF\Orb} E_{G_{can,min}}(\cO^\infty_\hlg(S_{disc}))\ .
\end{equation}
	We now use the equivalence 
	\begin{equation}\label{efwlkjlewf23ff}
\ell(E^{top}_{\cF}G)\simeq \colim_{S\in G_{\cF}\Orb} \ell(S_{disc})\ .
\end{equation} 
in $G\Top[W_{G}^{-1}]$. 
Since $E_{G_{can,min}}\iota$ (as a functor on $G\Sp\cX$) and the functors
\[ -\times\Rips(G_{can,min}) \colon G\Top[W_{G}^{-1}]\to G\Top[W_{G}^{-1}] \]
and $\tilde \cO^{\infty}_{\hlg}$ preserve colimits, the map \eqref{ererhierhgioergerg} is equivalent to the map
	\begin{equation}
	\label{iuhjfvhj4r78zfudjf}E_{G_{can,min}}\iota(\tilde \cO^\infty_\hlg(\ell(E^{top}_\cF G)\times \Rips(G_{can,min})))\to E_{G_{can,min}}\iota(\tilde \cO^\infty_\hlg(\ell(E^{top}_\cF G)))\ .\end{equation}
		 
	By \cref{rviuohiofe2df2edf2} we have an equivalence  $\Rips(G_{can,min}) \simeq \ell(E^{top}_\Fin G)$. Furthermore, since $\Fin\subseteq \cF$ we have an equivalence   \[\ell(E^{top}_\cF G)\times \ell(E^{top}_\Fin G)\simeq \ell(E^{top}_\cF G\times E^{top}_\Fin G)\simeq   \ell(E^{top}_\Fin G)\] induced by the projection 
	$E^{top}_{\cF}G\to *$. Consequently, the map \eqref{iuhjfvhj4r78zfudjf} and hence  	the map \eqref{eq:assmap} are further equivalent to
	\[E_{G_{can,min}}\iota(\tilde \cO^\infty_\hlg(\ell(E^{top}_\Fin G)))\to E_{G_{can,min}}\iota(\tilde \cO^\infty_\hlg(\ell(E^{top}_\cF G)))\ .\] Using \eqref{efwlkjlewf23ff}   again and  its analogue for the family $\Fin$, and \cref{hgergergegregrege} of $\tilde \cO^{\infty}_{\hlg}$, this map is equivalent to
	\[\colim_{S\in G_\Fin \Orb} E_{G_{can,min}}\iota(  \cO^\infty(S_{max,max}))\to \colim_{S\in G_\cF \Orb} (E_{G_{can,min}}\circ \iota)(\cO^\infty(S_{max,max}))\ .\]
	By \cite[Prop.~9.35]{equicoarse}, this map is equivalent to
	\begin{equation*}
\colim_{S\in G_\Fin \Orb}\Sigma E_{G_{can,min}}\iota(S_{min,max})\to \colim_{S\in G_\cF \Orb}\Sigma E_{G_{can,min}}\iota(S_{min,max})\ .
\end{equation*}
Using 
\eqref{regoihihihjriu34hr34r3}
 we can rewrite this morphism further in the form 
 \begin{equation}\label{vkjrhjekfewfwefewfew}
\colim_{S\in G_\Fin \Orb}\Sigma M(S )\to \colim_{S\in G_\cF \Orb}\Sigma M( S)\ .
\end{equation}
By comparison with \cref{ergoiegererg}, we see that \eqref{vkjrhjekfewfwefewfew}  is the suspension of the relative assembly map $\As_{\Fin,M}^{\cF}$ as desired.
\end{proof}

 {Let $H$ be a subgroup of $G$.}
\begin{prop}
	\label{prop:comparison2}
	 {The map 
	\begin{equation*}
	\wt E^{H}_{G }(*,F^\infty(H_{can,min}))\to \Sigma \wt E^{H}_{G}(\Res_H^G(A), F^0(H_{can,min}))
	\end{equation*}
	from \eqref{ghuhgiurehggregreregrrergergerg}}
	is equivalent to the map
	\begin{equation}\label{eq:somemap}
	\bL(G/H)\to \bM_A(G/H)\ ,\end{equation}
	where $\bL$ and $\bM_A$ are as in \eqref{def:funtorsequence} and \eqref{def:funtorsequence3}. 
\end{prop}
\begin{proof}

	By \cref{lem:induction},  the map {\eqref{ghuhgiurehggregreregrrergergerg}}
	 is equivalent to the composition
	\begin{align}\label{rgoirewgrwegrgeg}
	\wt E_{G }(*,\Ind_H^{G,Mot}(F_{H}^\infty(H_{can,min})))&\to \Sigma\wt E_{G}(*,\Ind_H^{G,Mot}(F_{H}^0(H_{can,min})))\\&\to \Sigma\wt  E_{G }(A,\Ind_H^{G,Mot}(F_{H}^0(H_{can,min})))\ ,\nonumber\end{align}
	where we also use the notation from \cref{gregee4erg}.  {By \cref{gregee4erg}, induction commutes with $F^\infty$ and $F^0$.}
	Since $H$ is $H$-finite, $E$ is continuous and $E_{G}$ is the twist of $E$ with $G_{max,max}$ (by convention \eqref{ewecvwevewcewcw}), the map $\Ind_H^G(H_{can,min})\to \Ind_H^G(\Res_H^G(G_{can,min}))$ induces an equivalence from the first map in  \eqref{rgoirewgrwegrgeg}  to
	\[\wt E_{G}(*,F_{\fin}^\infty(\Ind_H^G(\Res_H^G(G_{can,min})))\to \Sigma\wt E_{G}(*,F_{\fin}^0(\Ind_H^G(\Res_H^G(G_{can,min}))))\]
	by \cref{lem:comparison-left} and \cref{lem:comparison-right}.
	
	 We now investigate the second map in \eqref{rgoirewgrwegrgeg}.	
	By \cref{lem:comparison-right} and since $H$ is $H$-finite, the map $\Ind_H^G(H_{can,min})\to \Ind_H^G(\Res_H^G(G_{can,min}))$ induces an equivalence from the second map in the composition \eqref{rgoirewgrwegrgeg} to
	\[ \Sigma\wt  E_{G }(*,F_\fin^0(\Ind_H^G(\Res_H^G(G_{can,min}))))\to \Sigma\wt  E_{G}(A,F_{\fin}^0(\Ind_H^G(\Res_H^G(G_{can,min}))))\ .\]
	
	We conclude that \eqref{rgoirewgrwegrgeg} is equivalent to 
	\begin{align}\label{rgoirewgrwegrgeg111}
	\wt  E_{G}(*,F_{\fin}^\infty(\Ind_H^G(\Res_H^G(G_{can,min})))&\to \Sigma\wt  E_{G }(*,F_{\fin}^0(\Ind_H^G(\Res_H^G(G_{can,min}))))\\&\to  \Sigma\wt  E_{G}(A,F_{\fin}^0(\Ind_H^G(\Res_H^G(G_{can,min}))))\ . \nonumber\end{align}
	Now using the isomorphism
	\[\Ind_H^G(\Res_H^G(G_{can,min}))\stackrel{\text{Lemma}~\ref{lem:frobenius}}{\cong} \Ind_{H}^{G}(\pt)\otimes G_{can,min} \cong (G/H)_{min,min}\otimes G_{can,min}\]
	and invoking \eqref{def:funtorsequence} and  \eqref{def:funtorsequence3}, we obtain an equivalence from the composition \eqref{rgoirewgrwegrgeg111}   to
	\[\bL(G/H)\to \bM_*(G/H)\to \bM_{A}(G/H)\]
	as claimed.
\end{proof}

\begin{proof}[Proof of \cref{thm:novikov1-neu-relativ}]
	By \cref{lem:firststep}, we have to show that $\bL(S)\to \bM_A(S)$ is an equivalence for all $S$ in $G_\cF\Orb$. By \cref{prop:comparison2}  {and \cref{thm:leftinverse-rel}}, it hence suffices to show that the assumptions of \cref{thm:leftinverse-rel} are satisfied  {for every $H$ in $\cF$}. 
	Assumption \ref{ijgorggergerg-rel} from \cref{thm:leftinverse-rel} follows from Assumptions \ref{it:nov2} and \ref{it:nov3} of \cref{thm:novikov1-neu-relativ}.
	Since $\cF$ was assumed to be a subfamily of $\FDC$, Assumption \ref{it:descent3-rel} of \cref{thm:leftinverse-rel} follows from \cref{thm:replacement2}. 
\end{proof}

We observe that the FDC-assumption on $\cF$ in   \cref{thm:novikov1-neu-relativ} is used to verify   Assumption \ref{it:descent3-rel} of \cref{thm:leftinverse-rel}. If one is interested in the case $\cF=\All$ and assumes that $E_{\Fin}^{top}G$ has a finite $G$-CW-model, then we can reformulate  Assumption \ref{it:descent3-rel} of \cref{thm:leftinverse-rel}
as an assumption that certain forget-control maps for    $H$-equivariant coarse homology theories introduced below are equivalences
for all finite subgroups $H$ of $G$.

For an  equivariant coarse homology theory   $E\colon G\BC  \to \bC$ and a finite subgroup $H$ we define an $H$-equivariant coarse homology theory ${}^{H}\hspace{-3pt}E  $  and its twist  ${}^{H}\hspace{-3pt}E_{H}$ by $H_{max,max}$ (compare with \eqref{ewecvwevewcewcw}) by
\begin{equation}\label{riuhfweiufhiuwef2f2f23f}  {}^{H}\hspace{-3pt}E:= E  \circ \Ind_{H}^G\ , \quad  {}^{H}\hspace{-3pt}E_{H}:=({}^{H}\hspace{-3pt}E)_{H_{max,max}} \ . \end{equation}

Let $G$ be a group,  let   $E\colon G\BC  \to \bC$ be  an  equivariant coarse homology theory,  and set $M:=E_{G_{can,min}}\circ i \colon G\Orb\to \bC$.

\begin{theorem}\label{thm_descent_classic}
Assume that:
\begin{enumerate}\item  $\bC$ is stable, complete and cocomplete.
\item $E$ is continuous and strongly additive.
\item $E$ extends to an equivariant coarse homology theory with transfers $E^{tr}$.
\item\label{item_finite_BC_descent} $E^{top}_{\Fin}G$ can be represented by a finite $G$-CW complex.
\item \label{vgweroivjrewioverwfewrfw} The forget-control map
\begin{equation*}
\label{eq_classic_descent_assumption}
 {}^{H}\hspace{-3pt}E_{H} (   \Res^{G,Mot}_{H}(F^\infty ( G_{can,min})))\to\Sigma  {}^{H}\hspace{-3pt}E_{H}(  \Res^{G,Mot}_{H}(F^{0}( G_{can,min})))
\end{equation*}
is an equivalence  {for every finite subgroup $H$ of $G$.}
\end{enumerate}
Then the assembly map $\As_{ \Fin,M}$ admits a left inverse.
\end{theorem}

\begin{rem}
Note that the first three conditions  together are almost equivalent to the condition that $M$ is a CP-functor (see \cref{bioregrvdfb}). The assumption that $\bC$ is compactly generated is omitted because it is only used in \cref{thm:replacement2}.  

Our reason to use the equivariant coarse homology  $E$ as the primary object in this formulation is because it appears explicitly in Condition \ref{eq_classic_descent_assumption}.
\end{rem}

\begin{proof}
In the proof of Proposition~\ref{lem:firststep} we have shown that the suspension of the assembly map $\As_{\Fin,M}$ is equivalent to the morphism $\colim_{S\in G_\All\Orb}\bL(S)\to \colim_{S\in G_\All\Orb}\bM_*(S)$. Since the object $G/G$ is final in $G_\All\Orb$, this morphism is equivalent to the morphism $\bL(G/G)\to \bM_*(G/G)$. Therefore in order   to  show that it admits a left inverse, it suffices to show that the composition $\bL(G/G)\to \bM_*(G/G) \to \bM_A(G/G)$ is an equivalence. By \cref{prop:comparison2} we can equivalently show that the assumptions of \cref{thm:leftinverse-rel} with $H := G$ are satisfied.

Assumption \ref{ijgorggergerg-rel} of \cref{thm:leftinverse-rel} follows from Lemma~\ref{lem:fin-dim-model} applied to the   family $\cF =  \Fin  $ 
and Assumption \ref{item_finite_BC_descent}.
In view of \cref{rgeriowegewfewfrefwefew} and \cref{rem_descent_concrete_S} it suffices  to verify Assumption \ref{it:descent3-rel} of \cref{thm:leftinverse-rel}
for all $G$-finite $G$-sets $S$ with finite stablizers.

Using $E_{G}:=E_{G_{max,max}}\simeq (E^{tr})_G^G\iota$ by Definition  \eqref{5gpo3jkopjk34opg3ogpg34g34g34} we see that the map \eqref{revkjenekjvnrvevvverveer} in Assumption \ref{ijgorggergerg-rel} of \cref{thm:leftinverse-rel}  is the  map
\begin{equation}
\label{eq_classic_descent_nbew4df9023dfs}
E_G(F^\infty (S_{min,min}\otimes G_{can,min}))\to\Sigma E_G(F^{0}(S_{min,min}\otimes G_{can,min}))\ .
\end{equation}
We must  show that \eqref{eq_classic_descent_nbew4df9023dfs} is an equivalence for every $G$-finite $G$-set  $S$ with finite stablizers.
By Lemma~\ref{gregee4ergtw} we can interchange the twist by $S_{min,min}$ with $F^\infty$ and $F^0$.
Hence \eqref{eq_classic_descent_nbew4df9023dfs} is equivalent to 
\begin{equation}
\label{eq_classic_descent_nbew4df9023dfs1111}
E_G(S_{min,min}\otimes F^\infty ( G_{can,min}))\to\Sigma E_G(S_{min,min}\otimes F^{0}( G_{can,min}))\ .
\end{equation}
Since $S$ is a finite union of $G$-orbits, in order to show that \eqref{eq_classic_descent_nbew4df9023dfs1111} is an equivalence, by excision we can assume that $S=G/H\in G_{\Fin}\Orb$.
Then $S_{min,min}\cong \Ind_{H}^{G}(*)$.
By Lemma~\ref{ergijeogegregrgregg} we get
\[\Ind_{ H}^G(\Res_{H}^G(G_{max,max}) \otimes *) \cong G_{max,max} \otimes S_{min,min}\ .\]
The inclusion $H_{max,max}\to \Res^{G}_{H}(G_{max,max})$ is an equivalence in $H\BC$.
Consequently, $G_{max,max} \otimes S_{min,min}$ is equivalent to 
$\Ind_{ H}^G( H_{max,max}  ) $ in $G\BC$.
 In view of the definition \eqref{ewecvwevewcewcw} of $E_{G}$
 we can replace \eqref{eq_classic_descent_nbew4df9023dfs1111} by
 \begin{equation}
\label{eq_classic_descent_nbew4df9023dfs11111111}
E(\Ind_{ H}^G( H_{max,max}  )\otimes F^\infty ( G_{can,min}))\to\Sigma E(\Ind_{ H}^G( H_{max,max}  )\otimes F^{0}( G_{can,min}))\ .
\end{equation}
 Using \cref{43foijiiof43f34f34f1111} and \eqref{riuhfweiufhiuwef2f2f23f} we can rewrite \eqref{eq_classic_descent_nbew4df9023dfs11111111} in the form
 \begin{equation*}
\label{eq_classic_descent_nbew4df9023dfs1111111111}
 {}^{H}\hspace{-3pt}E_{H} (   \Res^{G,Mot}_{H}(F^\infty ( G_{can,min})))\to\Sigma  {}^{H}\hspace{-3pt}E_{H}(  \Res^{G,Mot}_{H}(F^{0}( G_{can,min})))
\end{equation*}
which is an equivalence by Assumption \ref{vgweroivjrewioverwfewrfw}.
\end{proof}

\bibliographystyle{amsalpha}
\bibliography{born-transbas}
\end{document}